\documentclass[11pt]{article}

\usepackage[T1]{fontenc}
\usepackage{lmodern}
\usepackage{amsmath,amssymb,amsthm,mathtools}
\usepackage[margin=1in]{geometry}
\usepackage{booktabs,array,tabularx}
\usepackage{float}
\usepackage{caption}
\newcommand{\tablenote}[1]{%
  \par\smallskip
  \begin{minipage}{\linewidth}
  \footnotesize\raggedright\textit{Notes.} #1
  \end{minipage}}
\usepackage{needspace}
\usepackage{aliascnt}
\usepackage{microtype}
\usepackage[colorlinks=true,linkcolor=linkblue,citecolor=linkblue,urlcolor=linkblue]{hyperref}
\usepackage[nameinlink,noabbrev,capitalise]{cleveref}
\usepackage{enumitem}
\usepackage{xcolor}
\usepackage{tikz}
\usetikzlibrary{arrows.meta,positioning,calc}

\allowdisplaybreaks

\definecolor{ink}{RGB}{35,45,55}
\definecolor{soft}{RGB}{245,247,249}
\definecolor{refblue}{RGB}{38,96,146}
\definecolor{refred}{RGB}{164,67,58}
\colorlet{linkblue}{blue!60!black}

\floatstyle{ruled}
\newfloat{algorithm}{tbp}{loa}
\floatname{algorithm}{Algorithm}
\newlist{algsteps}{enumerate}{2}
\setlist[algsteps,1]{label=\arabic*.,leftmargin=2em,itemsep=0.2em,topsep=0.4em}
\setlist[algsteps,2]{label=(\alph*),leftmargin=2em,itemsep=0.15em,topsep=0.2em}

\newcommand{\R}{\mathbb R}
\newcommand{\E}{\mathbb E}
\newcommand{\norm}[1]{\left\lVert #1\right\rVert}
\newcommand{\abs}[1]{\left\lvert #1\right\rvert}
\newcommand{\ip}[2]{\left\langle #1,#2\right\rangle}
\newcommand{\opnorm}[1]{\left\lVert #1\right\rVert_{\mathrm{op}}}
\newcommand{\fnorm}[1]{\left\lVert #1\right\rVert_{\mathrm F}}
\newcommand{\Reg}{\operatorname{Reg}}
\newcommand{\MinReg}{\mathfrak R}
\newcommand{\Alg}{\mathsf{Alg}}
\newcommand{\Env}{\mathsf{Env}}
\newcommand{\argmin}{\operatorname*{argmin}}

\newcommand{\conv}{\operatorname{conv}}

\newcommand{\tr}{\operatorname{tr}}
\newcommand{\Span}{\operatorname{span}}
\newcommand{\im}{\operatorname{im}}
\newcommand{\Var}{\operatorname{Var}}
\newcommand{\TV}{\operatorname{TV}}
\newcommand{\KL}{\operatorname{KL}}
\newcommand{\Reach}{\mathcal R}
\newcommand{\cK}{\mathcal K}
\newcommand{\cC}{\mathcal C}
\newcommand{\cQ}{\mathcal Q}
\newcommand{\cS}{\mathcal S}
\newcommand{\cY}{\mathcal Y}

\newcommand{\simplex}[1]{\Delta_{#1}}

\newcommand{\Vstar}{V_{\!\star}}

\theoremstyle{plain}
\newtheorem{theorem}{Theorem}[section]
\newaliascnt{lemma}{theorem}
\newtheorem{lemma}[lemma]{Lemma}
\aliascntresetthe{lemma}
\newaliascnt{proposition}{theorem}
\newtheorem{proposition}[proposition]{Proposition}
\aliascntresetthe{proposition}
\newaliascnt{corollary}{theorem}
\newtheorem{corollary}[corollary]{Corollary}
\aliascntresetthe{corollary}
\theoremstyle{definition}
\newaliascnt{definition}{theorem}
\newtheorem{definition}[definition]{Definition}
\aliascntresetthe{definition}
\newaliascnt{assumption}{theorem}

\aliascntresetthe{assumption}
\newaliascnt{remark}{theorem}
\newtheorem{remark}[remark]{Remark}
\aliascntresetthe{remark}
\newaliascnt{example}{theorem}
\newtheorem{example}[example]{Example}
\aliascntresetthe{example}

\crefname{theorem}{Theorem}{Theorems}
\crefname{lemma}{Lemma}{Lemmas}
\crefname{corollary}{Corollary}{Corollaries}
\crefname{proposition}{Proposition}{Propositions}
\crefname{assumption}{Assumption}{Assumptions}
\crefname{definition}{Definition}{Definitions}
\crefname{remark}{Remark}{Remarks}
\crefname{example}{Example}{Examples}
\crefname{algorithm}{Algorithm}{Algorithms}

\title{Online Control via Counterfactual Tracking}
\author{Yunzong Xu\thanks{University of Illinois Urbana-Champaign.
Email: \href{mailto:xyz@illinois.edu}{\texttt{xyz@illinois.edu}}.}}
\date{}

\begin{document}
\maketitle

\begin{abstract}
We study online control of a known linear dynamical system with adversarial
costs and bounded disturbances, measuring regret against a general class of
benchmark policies.  We introduce \emph{counterfactual tracking}, which
separates the challenge of learning from the challenge of controlling the
system.  An online learner builds a reference trajectory by selecting or
averaging the trajectories that the benchmark policies would have generated
under the realized costs and disturbances, and a corrective law steers the
system toward that reference.  Charging each change in the reference its
recovery cost (the cost of steering the system onto the new reference)
reduces the problem to online learning with switching costs.  Conversely,
under additional natural assumptions, we show that this reduction is tight:
the two problems have the same minimax regret up to a system-dependent
factor, uniformly over horizons and policy classes.

The reduction gives sharp regret guarantees for policy classes beyond
standard finite-memory parameterizations.  For a class of $N$ possibly
nonlinear or history-dependent policies, it achieves $O(\sqrt{T\log N})$
regret over $T$ rounds, provided their trajectories remain within a bounded
distance of one another and recovery costs are bounded.  For the full
$\ell_1$ ball of disturbance-response controllers, it achieves
$O(\sqrt{T\log T})$ regret, which is minimax optimal in $T$, without
assuming a common decay rate for disturbance effects.  The framework also
improves the best known regret bounds for linear state-feedback policies.
\end{abstract}

\section{Introduction}
\label{sec:introduction}

Online control studies how to operate a dynamical system when future costs
and disturbances are unknown
\cite{agarwal2019online,hazan2020nonstochastic,simchowitz2020improper,
bhatia2020,kumar2023}.  Unlike ordinary online learning, each decision changes
both the current cost and the state from which future decisions must be made.
We consider the known linear dynamical system
\begin{equation*}
 x_{t+1}=Ax_t+Bu_t+w_t,
\end{equation*}
over a horizon of $T$ rounds.
Here $x_t\in\R^{d_x}$ is the state, $u_t\in\R^{d_u}$ is the input, and
$w_t$ is an external disturbance, and the matrices $A$ and $B$ are known.
An algorithm
chooses $u_t$ before seeing the convex cost function $c_t$ or disturbance
$w_t$, incurs $c_t(x_t,u_t)$, then observes the full cost function and the
next state, from which $w_t$ can be recovered.  Costs and disturbances may be
adversarial, with $\norm{w_t}_2\le W$.

A policy is a rule for choosing inputs from the information available so far.
A policy's \emph{counterfactual trajectory} is the state and input sequence it
would have generated from the same initial state under the realized costs and
disturbances.  We compare against a general policy class $\Pi$: the
algorithm's \emph{regret} is its cumulative cost minus the lowest cumulative
cost achieved by a policy in $\Pi$ along its counterfactual trajectory
\cite{arora2012policy,hazan2020nonstochastic,bhatia2020}.\footnote{Stochastic
control \cite{astrom1970,bertsekas2017} and robust control
\cite{basar1995,zhou1996} usually optimize expected and worst-case total
cost, respectively.
Regret measures excess cost relative to a policy class, so an algorithm with
low regret adapts to the realized costs and disturbances instead of planning
for the worst case; see \cite{hazan2026onlinecontrol} for connections and
differences.}
Sublinear regret means that the average excess cost per round tends to zero.

\Needspace{8\baselineskip}
\subsection{Motivation and research question}
\label{sec:motivation}

The disturbance--action control (DAC) framework
\cite{agarwal2019online} showed that online control with adversarial
disturbances and general convex costs admits efficient algorithms with
$\sqrt{T}$-type regret against stabilizing linear state-feedback policies
$u_t=-Kx_t$.  Subsequent work extends the approach to
linear controllers with internal memory and develops learning methods that
retain an unbounded disturbance history
\cite{simchowitz2020improper,kumar2023}.

Despite this success, existing guarantees rely on explicitly assumed
properties of both the policy class and the system: the benchmark policies
are linear and share a common parameterization, their disturbance effects
decay at a common rate, and a single feedback law stabilizes the
entire state.  These assumptions are sufficient for low regret, but it
remains unclear which of them are necessary.

Consider the policy class first.  We may wish to compete with controllers
designed by different methods, including nonlinear learned controllers,
which share no common parameterization \cite{li2023switching}.  Or we may
wish to compete with all controllers that limit how much disturbances
affect states and inputs, without prescribing how quickly those effects
decay; such classes arise from classical disturbance-rejection objectives
and their formulation through system responses
\cite{dahleh1987l1,anderson2019sls}.

The stabilization assumption creates a similar limitation on the system
side.  In load-balancing models for distribution networks
\cite{weivanderschaft2013}, internal flows can change how load is allocated
but not its total.  Because the total is conserved, no feedback law can drive
the full state to zero, yet redistribution can still balance the load.  No
input can change such a conserved component, so it is the same along every
policy's trajectory.  We show that low regret remains achievable for this
system (\cref{ex:network}), even though a stabilization assumption would
exclude it altogether.

These examples suggest that the assumptions behind existing guarantees
sometimes reflect the analysis rather than the problem itself.  A general
theory of online control should take the system and the policy class as
given (rather than restrict both in advance) and characterize how their
structure determines the difficulty of learning.

We therefore ask:
\begin{quote}
\itshape
For which linear systems $(A,B)$ and policy classes $\Pi$ is sublinear
regret achievable, and, for an arbitrary system and policy class, what
determines the optimal regret rate?
\end{quote}
Answering this question at this level of generality requires both an
algorithm and a converse showing that no algorithm can do substantially
better.

\subsection{Our contributions}
\label{sec:contributions}

Our answer is that the difficulty of online control is determined by the
counterfactual trajectories of the benchmark policies and the cost of
correcting changes between them.  We make this precise through
\emph{counterfactual tracking}, an algorithmic reduction in which an online
learner chooses which reference to follow among the policies'
counterfactual trajectories, and a corrective law controls the physical
state to follow it.  This reduces online control to online learning with
switching costs.  Under additional natural assumptions, a matching converse
shows that the reduction is tight: sublinear regret is possible in online
control exactly when it is possible in the induced switching problem, and
the two problems have the same minimax regret up to a system-dependent
factor.

\subsubsection{Counterfactual tracking: the algorithmic reduction}

At each round, an online learner, which we call the \emph{base learner},
selects or averages the current state--input points on the policies'
counterfactual trajectories, forming a reference.  When the reference
changes, the input can switch to the new reference at once, but the state
cannot: it still reflects earlier inputs.  A corrective law, which may be
any feedback on the tracking error, including a nonlinear one, steers the
state onto the new reference, and a recovery norm $V$ measures the cost of
doing so, namely the cumulative state--input deviation needed to remove a
state gap.  For $L$-Lipschitz costs, we show that
\begin{equation*}
 \text{control regret}
 \ \le\
 \underbrace{\text{online-learning regret}}_{\text{choosing the reference}}
 +\underbrace{L\sum_{t<T}V(\delta_{t+1})}_{
   \text{correcting reference changes}},
\end{equation*}
where $\delta_{t+1}$ is the gap between the next states requested by the old
and new references.  Thus a base learner that controls its regret and
recovery-weighted movement yields a low-regret online-control algorithm.

Unlike approaches that build the benchmark class into a convex
parameterization, the reduction is general and modular.  On the policy
side, the benchmark policies enter only through their counterfactual
trajectories, so they need not be linear, share a parameterization, or
decay at a common rate.  On the system side, only gaps between trajectories
must be corrected; components shared by all trajectories, even unstable
ones, never create such gaps, so the system need not be stabilizable.\footnote{In the population-dynamics model
of Golowich, Hazan, Lu, Rohatgi, and Sun, whose states lie on a simplex and
whose inputs are constrained, the entire state likewise need not decay;
their analysis relies instead on mixing on the simplex
\cite{golowich2024population}.}  On
the learning side, any base learner with controlled movement can be used:
off-the-shelf online-learning algorithms turn directly into online-control
algorithms with regret guarantees, and a new policy class may require a new
base learner but no new control analysis.

\subsubsection{Minimax tightness of the reduction}\label{sec:minimax-tight}

The reduction yields online-control algorithms and regret upper bounds.  A
general online-control algorithm, however, need not track any benchmark
trajectory: it may play inputs that no benchmark policy would choose.  One
might expect this freedom to help.  Our converse shows that it does not: no
algorithm improves on the best rate attainable through the reduction by
more than a factor that depends only on the system, not on the policy class
or the horizon.  This establishes a \emph{minimax equivalence} between
online control and the induced switching problem: tracking benchmark
trajectories loses at most a system-dependent factor, and recovery cost is
the right price for changing references.

The converse requires unconstrained inputs and convex globally Lipschitz
costs.  It also requires that the benchmark
policies not use previously revealed cost functions.  This condition
restricts only the benchmarks, not the algorithms, which may use the entire
revealed cost history; the reduction itself, and the upper bounds it
yields, also allow benchmarks that respond to past costs.

We next apply the reduction to two policy classes beyond standard
finite-memory parameterizations: finite classes of possibly nonlinear
policies, and the full $\ell_1$ ball of disturbance responses.  Unless stated
otherwise, the rates below treat the Lipschitz and strong-convexity
constants, the disturbance bound, and bounds on recovery costs as fixed.

\subsubsection{Application I: learning general policy classes}

Consider a class of $N$ policies, which may be nonlinear or history-dependent
and need not share a common parameterization.  With exponential weights as
the base learner, counterfactual tracking follows the weighted average of the
policies' counterfactual trajectories.  If the distances between these
trajectories and the recovery costs of their gaps are bounded by known
constants, the regret is $O(\sqrt{T\log N})$, and a lower bound on a fixed
scalar system shows that this rate is worst-case optimal in $T$ and $N$
(\cref{sec:policies}).

Under strong convexity, the regret improves to $O(\log N)$, independent of
the horizon.  Because the algorithm follows the average of the policies'
trajectories rather than a sampled policy, strong convexity makes the cost
of this average smaller than the average of the policies' costs whenever
their trajectories differ; this saving, the \emph{Jensen gain}, pays for
both learning and recovery.  Both bounds extend to infinite, possibly
continuous policy classes: with a prior over the class, the relative entropy
of a comparison distribution replaces $\log N$, as in PAC-Bayesian bounds
\cite{alquier2024pacbayes,mcallester1999pac}.

\subsubsection{Application II: learning bounded disturbance responses}

A linear time-invariant policy's \emph{disturbance response} describes how
past disturbances affect states and inputs.  System-level synthesis (SLS)
designs controllers through these responses, allowing performance bounds to
be imposed directly \cite{anderson2019sls,wang2019sls}.  We compete with the
full $\ell_1$ ball of responses around any fixed feasible baseline, such as
the zero-input response, so the system need not be stabilizable.  The
radius of the ball limits disturbance amplification relative to the
baseline, without assuming a common decay rate for disturbance effects or a
bound on the policies' internal memory.  Around a stabilizing baseline, the
ball contains every finite-dimensional linear controller that internally
stabilizes the system, for some finite radius.

For a fixed radius, the regret is $O(\sqrt{T\log T})$, with no explicit
dependence on the state and input dimensions.  Prior bounds for linear policy
classes assume a common decay rate or bounded memory; removing both costs
only a $\sqrt{\log T}$ factor.  A matching lower bound shows that this factor
is necessary even on a scalar system that recovers from any gap in one
step, so the rate is minimax optimal in $T$.  The logarithm is thus a
property of the class, the price of learning a response over all
disturbance ages, not of our algorithm or of slow recovery
(\cref{sec:responses,thm:response-lower}).  The base learner needs no
finite-memory approximation.  It stores $O(T)$ response coefficients and
solves a convex program in each round.

\subsubsection{Further consequences for linear state-feedback policies}

For linear state-feedback policies $u_t=-Kx_t$, the benchmark class of DAC,
the framework improves the best known regret bounds.  Suppose that under
every benchmark policy, the effect of a disturbance $r$ rounds after it occurs is at
most a constant times $(1-\gamma)^r$, a common decay rate with \emph{decay
margin} $\gamma\in(0,1)$.  For convex costs, the regret is
$O(\gamma^{-1}\sqrt{d_xT})$, a better dependence on $\gamma$ than earlier
guarantees (\cref{tab:stable-comparison} compares them under the gradient
normalization of prior work).  A lower bound on a fixed scalar system shows
that the dependence on $T$ and $\gamma$ is sharp; the sharp dependence on
dimension and recovery remains open.  The common decay rate also makes the
algorithm computationally cheaper: as in DAC, the base learner can learn
only the response to recent disturbances, and spectral filters
\cite{brahmbhatt2025} reduce the cost further under additional assumptions
(\cref{sec:familiar-benchmarks}).

Under strong convexity, running the method of Application~I on a finite
cover of feedback gains gives $O(d_xd_u\log T)$ regret, with no dependence
on $\gamma$ beyond its effect on the Lipschitz constant.  This algorithm
need not be computationally efficient, but it shows that slowly decaying
benchmark policies need not increase regret by themselves
(\cref{cor:strong-feedback-main}).

\Cref{tab:applications} summarizes the three classes; for the full response
ball, the $\sqrt{T\log T}$ lower bound persists under strong convexity.

\begin{table}[htbp]
\centering
\caption{Regret for three policy classes.}
\label{tab:applications}
\vspace{0.3em}
\small
\renewcommand{\arraystretch}{1.16}
\begin{tabularx}{\textwidth}{@{}
 >{\raggedright\arraybackslash}p{0.28\textwidth}
 >{\centering\arraybackslash}p{0.34\textwidth}
 >{\centering\arraybackslash}X@{}}
\toprule
Policy class & Convex costs & With strong convexity \\
\midrule
$N$ policies
& $\Theta\!\left(\min\{T,\sqrt{T\log N}\}\right)$
& $O(\log N)$ \\
\addlinespace
full $\ell_1$ disturbance-response ball
& $\Theta\!\left(\sqrt{T\log T}\right)$
& $\Omega(\sqrt{T\log T})$ \\
\addlinespace
stable linear state-feedback policies
& $O\!\left(\gamma^{-1}\sqrt{d_xT}\right)$
& $O(d_xd_u\log T)$ \\
\bottomrule
\end{tabularx}
\tablenote{Constants not shown are fixed, as in the text; the
state-feedback rows also fix $\norm{B}_{\mathrm{op}}$ and the gain and decay
constants.  For
$N$ policies, the distances between trajectories and the recovery costs of
their gaps are bounded.  Convex costs are globally Lipschitz; the strongly
convex entries use the local cost assumptions of the corresponding results.
Each $\Theta$ entry is the worst-case rate over classes of the stated type,
with lower bounds on fixed scalar systems; individual classes may be easier.
In the strongly convex state-feedback entry, $\gamma$ can still enter
through the Lipschitz constant, and the algorithm need not be
computationally efficient.}
\end{table}

\subsection{Relation to prior work}

\paragraph{Online control against linear policies.}
Cohen, Hassidim, Koren, Lazi\'c, Mansour, and Talwar \cite{cohen2018linear}
obtain efficient $O(\sqrt T)$ regret
for known linear systems with stochastic disturbances and adversarial
quadratic costs through a semidefinite relaxation of steady-state
covariances.  The DAC framework of Agarwal,
Bullins, Hazan, Kakade, and Singh \cite{agarwal2019online} handles adversarial disturbances and
general convex costs, and Simchowitz, Singh, and Hazan \cite{simchowitz2020improper}
extend it to partial observation and to stabilizing linear dynamic
controllers; see \cite{hazan2026onlinecontrol} for an introduction to this
line of work.  A DAC policy has the form
\[
 u_t=-K_0x_t+\sum_{i=1}^{H}M^{[i]}w_{t-i},
\]
where $K_0$ is a fixed gain that stabilizes the system and the matrices
$M^{[i]}$ are learned by online convex optimization with memory
\cite{anava2015memory}.  Stability bounds the error from truncating the
disturbance history and relates the costs actually incurred to the
surrogate losses used for learning.

Counterfactual tracking builds on this algorithmic idea.  In DAC, the term
$-K_0x_t$ plays two roles: together with the learned disturbance feedback,
it represents the benchmark policies, and it corrects the state error that
arises when the learned matrices $M^{[i]}$ change, which is what relates the actual
costs to the surrogate losses.  We ask three questions: (i) can these two
roles be separated in the analysis; (ii) if so, can the correction
$-K_0x_t$ be generalized to broader corrective laws, and the learned
disturbance feedback $\sum_iM^{[i]}w_{t-i}$ to general rules for choosing which benchmark
trajectory to follow; and
(iii) can the resulting
approach handle more general pairs $((A,B),\Pi)$?  Counterfactual tracking
answers all three affirmatively.  The base learner chooses a reference among
the benchmark policies' counterfactual trajectories, as the learned matrices
$M^{[i]}$ do implicitly in DAC, and a separate, general corrective law, such as another
stabilizing gain or a nonlinear finite-horizon model predictive control law,
steers the state toward it; the consequences for systems and policy classes
are those listed in \cref{sec:contributions}.  For linear state feedback,
our bounds improve the dependence of earlier guarantees on the decay margin
(\cref{tab:stable-comparison}).

\paragraph{Disturbance responses and long memory.}
SLS describes closed-loop behavior through affine constraints on disturbance
responses \cite{anderson2019sls,wang2019sls}, which are affinely equivalent
to infinite-memory DAC coordinates \cite{minasyan2021timevarying}.  Methods
that retain an unbounded disturbance history still limit the effect of earlier
disturbances uniformly: Kumar, Dean, and Kleinberg assume a geometric decay bound
\cite{kumar2023}, and Simchowitz, Singh, and Hazan assume that responses decay
\cite{simchowitz2020improper}.  Our response-ball guarantee needs no decay
information and allows any feasible baseline; it is a learning guarantee for
the full norm-bounded class rather than a new parameterization
(\cref{app:kdk-comparison}).  Brahmbhatt, Buzaglo, Druchyna, and Hazan construct spectral filters
that compress long responses under additional eigenvalue assumptions
\cite{brahmbhatt2025}.  As an optional cheaper implementation for linear
state feedback, we combine their filters with counterfactual tracking, which
lets the corrective law differ from their stable baseline
(\cref{sec:spectral-compression}).

\paragraph{Fast rates under strong convexity.}
Agarwal, Hazan, and Singh obtain logarithmic regret using stochastic excitation
\cite{agarwal2019logarithmic}.  Foster and Simchowitz obtain logarithmic
regret for adversarial linear-quadratic control with known quadratic costs
by reducing it to delayed online learning \cite{foster2020logarithmic}.
Simchowitz obtains polylogarithmic regret for known dynamics and strongly
convex costs under fully adversarial disturbances, also allowing partial
observation and linear dynamic controllers \cite{simchowitz2020almost}.
Our strongly convex results instead use the regret--variance inequality of
van der Hoeven, Zhivotovskiy, and Cesa-Bianchi \cite{vanderhoeven2022}:
tracking the averaged trajectory yields a Jensen gain that pays for both
learning and recovery.  This gives horizon-independent regret for finite
classes and, through a finite cover of feedback gains, logarithmic regret
for linear state feedback without an additional decay factor, although the
covering algorithm need not be efficient (\cref{sec:limits}).

\paragraph{Competing with general policy classes.}
Earlier guarantees against finitely many policies that need not be linear
rely on structure in the benchmark: rapid forgetting of the state
\cite{abbasiyadkori2013mdp} or a contractive, differentiable
parameterization \cite{lin2023policy}.  Li, Preiss, Li, Lin, Wierman, and
Shamma \cite{li2023switching} study a more general problem, with unknown
dynamics, feedback only on realized costs, and black-box candidates that may
destabilize the system, and obtain $\widetilde O(N^{1/3}T^{2/3})$ regret plus
a horizon-independent term for unstable candidates.  Under known dynamics and
full cost feedback, counterfactual tracking likewise needs no structure in
the benchmark, only counterfactual trajectories that can be simulated and
whose gaps are bounded, and attains the worst-case optimal
$O(\sqrt{T\log N})$ rate.

\paragraph{Reductions to online learning with switching costs.}
Reducing a (stateful) online decision-making problem to (stateless) online
learning is a recurring idea.  Within control, earlier work reduces tracking
and linear-quadratic control to expert prediction
\cite{abbasi2014tracking,abbasi2019expert}.  Reductions to online learning
with switching costs have been developed for online convex optimization
with delayed feedback or memory \cite{qiu2026parameterfree} and for
inventory control, where carryover makes the feasible orders depend on the
past \cite{ichikawa2026inventory,pineci2026hidden}.  These reductions
motivate ours but exploit different problem structures.  Emek, Lavi, Niazadeh, and Shi establish a reduction
closely related to our approach for finite policy classes, for tabular deterministic
Markov decision processes whose transitions are revealed online
\cite{emek2024stateful}: when each benchmark policy can be followed
from any state at a bounded additional cost, competing with finitely many
policies reduces to online learning with a uniform cost per switch.  In our
setting, the cost of changing references instead scales with the gap
between the old and new trajectories: it is small for small changes and
unbounded in general, so how to measure it is a central issue.
Counterfactual tracking charges each change the recovery cost of this gap.
This charge lets the base learner average policies, extends the reduction
to infinite classes such as the full $\ell_1$ response ball, and underlies
the minimax converse.

\paragraph{Learnability and general frameworks.}
Bhatia and Sridharan develop a non-constructive minimax framework for online learning with
dynamics, allowing nonlinear systems and nonconvex costs \cite{bhatia2020}.
Kumar, Dean, and Kleinberg develop online convex optimization with unbounded memory and
establish sharp worst-case dependence on effective memory, a measure of how
past decisions affect current losses \cite{kumar2023}.  Both frameworks
cover more general problems than ours, but this generality can cost
sharpness on some basic known linear systems: the Bhatia--Sridharan bound can be linear on a
two-policy problem with $O(\sqrt T)$ achievable regret, and the bound of
Kumar, Dean, and Kleinberg gives $O(\sqrt{NT})$ rather than $O(\sqrt{T\log N})$ for
ordinary $N$-expert prediction with no state
(\cref{app:prior-general-frameworks}).  To our
knowledge, a necessary-and-sufficient characterization of learnability for
general policy classes has remained open even for known linear systems.
Our minimax equivalence result (\cref{sec:minimax-tight}) provides one through the induced switching problem, under the
assumptions stated in \cref{sec:contributions}.

\subsection{Organization and notation}
\Cref{sec:model} defines the model.  \Cref{sec:reduction} proves the
reduction, and \cref{sec:minimax-characterization} proves the converse and
develops the recovery geometry.  \Cref{sec:policies,sec:responses} give the
two main applications; \cref{sec:familiar-benchmarks} treats linear state
feedback, DAC, and spectral responses.  \Cref{sec:discussion} closes with
open directions. The appendices contain the detailed proofs and implementation
refinements, followed by the comparison with general prior frameworks.

For vectors, $\norm{\cdot}_2$ and $\ip{\cdot}{\cdot}$ denote the Euclidean
norm and inner product; for matrices, $\norm{\cdot}_{\mathrm{op}}$ and
$\norm{\cdot}_{\mathrm F}$ denote the operator and Frobenius norms, and
$\im$ denotes the range.  We write $a_{s:t}=(a_s,\ldots,a_t)$ for a
segment of a sequence, $\simplex N$ for the probability simplex on $N$
points, $\mathbf 1$ for the all-ones vector, $\KL$ for relative entropy, and
$\TV(p,q)=\frac12\int\abs{dp-dq}$ for total-variation distance.  The notation $O(\cdot)$,
$\Omega(\cdot)$, and $\Theta(\cdot)$ hides universal constants and the
problem constants that each result states to be fixed; $\widetilde O(\cdot)$
also hides logarithmic factors, and a subscript, as in $O_\kappa(\cdot)$,
indicates that the hidden factor may depend on the subscripted parameter.

\section{Problem setup}
\label{sec:model}

We compare an online-control algorithm with benchmark policies run on the
same realized costs and disturbances.  The known linear dynamical system is
\begin{equation}
 x_{t+1}=Ax_t+Bu_t+w_t,
 \qquad x_1=x_{\mathrm{ini}},
 \label{eq:system}
\end{equation}
where $A\in\R^{d_x\times d_x}$ and $B\in\R^{d_x\times d_u}$ are known,
$x_t\in\R^{d_x}$ is the state, $u_t\in\R^{d_u}$ is the input, and $w_t$ is
the disturbance.  The interaction lasts $T$ rounds, starting from the known
initial state $x_{\mathrm{ini}}$.  Write
\[
 \mathcal H_t=(x_{1:t},u_{1:t-1},w_{1:t-1},c_{1:t-1})
\]
for the history at the start of round $t$.  In each round $t=1,\ldots,T$:
\begin{enumerate}[label=(\roman*)]
\item the environment chooses a convex cost function
$c_t:\R^{d_x}\times\R^{d_u}\to\R$ and a disturbance $w_t$ with
$\norm{w_t}_2\le W$, possibly at random, as a function of $\mathcal H_t$;
\item without observing $c_t$ or $w_t$, the algorithm chooses $u_t$ as a
function of $\mathcal H_t$ and its private randomness;
\item the algorithm incurs $c_t(x_t,u_t)$ and observes the full cost
function $c_t$ and the next state $x_{t+1}$, from which
$w_t=x_{t+1}-Ax_t-Bu_t$ can be reconstructed.
\end{enumerate}
We call such an environment \emph{adaptive}: through $\mathcal H_t$, it may
react to the algorithm's past inputs, but not to the algorithm's random
choice in the current round.  An environment is \emph{oblivious} if it fixes
$(c_t,w_t)_{t\le T}$ before play.  Our upper bounds hold against adaptive
environments, and our lower bounds use oblivious ones.

Inputs are unconstrained, and we assume neither controllability nor
stabilizability of the full system $(A,B)$: some state directions may be unaffected
by inputs, and no gain $K$ need make $A-BK$ stable.  The general results
make no further assumptions on the system or the benchmark class; the
applications add bounds on the benchmark trajectories or their separation.

\subsection{State--input geometry and cost classes}
\label{sec:cost-classes}

Write $y=(x,u)$ for a state--input point and $y_t=(x_t,u_t)$ for the
algorithm's.  Fix a norm $\norm{\cdot}$ on $\R^{d_x+d_u}$ for differences
$(e,v)$ of state--input points; it may reflect physical units and relative cost
sensitivities, and Lipschitz constants are measured in it.  The quantitative
results for linear classes use $\norm{(e,v)}=\sqrt{\norm{e}_2^2+\norm{v}_2^2}$;
weighted norms allow different scales for state and input.

\begin{definition}[Cost classes]
\label{def:costs}
On a convex region $C$, a cost $c$ is $L$-Lipschitz with respect to $\norm{\cdot}$ if
$|c(y)-c(y')|\le L\norm{y-y'}$ for $y,y'\in C$.  It is
$\mu$-strongly convex in the Euclidean norm if
\[
 c(y')\ge c(y)+\ip{g}{y'-y}+\frac\mu2\norm{y'-y}_2^2
 \qquad(y,y'\in C,\ g\in\partial c(y)).
\]
\end{definition}

Lipschitzness bounds the extra cost of a state--input point $y$ over
another point $\bar y$ by $L\norm{y-\bar y}$, which charges both the state
deviation and the input deviation.  A global Lipschitz bound avoids bounding
the trajectories themselves, but strong convexity cannot coexist with
global Lipschitzness on all of Euclidean space.  Our fast-rate results
therefore assume both conditions only on the regions that the relevant
trajectories visit (\cref{sec:strong-policies,sec:limits}).

\subsection{Benchmark policies and regret}
\label{sec:benchmarks}

A benchmark policy $\pi$ is \emph{causal}: its input $u_t^\pi$ may depend on its own
states $x_{1:t}^\pi$ and past inputs $u_{1:t-1}^\pi$, the past disturbances
$w_{1:t-1}$ and cost functions $c_{1:t-1}$, the time $t$, and private
randomness, but not on the current disturbance $w_t$ or cost $c_t$.  For counterfactual
evaluation, we fix each policy's random seed before play and use the same
seed throughout its rerun; a ``policy'' below means such a seeded, hence
deterministic, realization, and any distribution over policies is
introduced explicitly.

Freeze the realized sequence $(c_t,w_t)_{t\le T}$, initialize
$x_1^\pi=x_{\mathrm{ini}}$, and run each policy on the realized
disturbances:
\begin{equation}
 x_{t+1}^\pi=Ax_t^\pi+Bu_t^\pi+w_t.
 \label{eq:policy-rerun}
\end{equation}
With $y_t^\pi=(x_t^\pi,u_t^\pi)$, regret to a class $\Pi$ is
\begin{equation}
 \Reg_T(\Pi)
 =\sum_{t=1}^T c_t(y_t)
 -\inf_{\pi\in\Pi}\sum_{t=1}^T c_t(y_t^\pi).
 \label{eq:regret}
\end{equation}
As in policy regret \cite{arora2012policy}, the benchmark accounts for how
each policy would have changed the state.

The benchmark imposes no parametric or stability restriction on $\Pi$; our
algorithms only need to simulate the policies' counterfactual trajectories.

\section{Counterfactual tracking: the algorithmic reduction}
\label{sec:correction}
\label{sec:reduction}

Counterfactual tracking is a meta-algorithm: given a base learner that
chooses among trajectories, it produces an online-control algorithm that
chooses inputs.  \Cref{sec:algorithm} states it.  \Cref{sec:references}
derives the identity behind it: the tracking error, the gap between the
physical state and the chosen reference, obeys the system's own dynamics,
driven only by the base learner's switches.  Without switches it stays zero,
and the base learner's regret is the control regret.  \Cref{sec:recovery} prices
each switch by the cost of removing the error it creates, and
\cref{sec:reduction-theorem} charges this price to the base learner and
proves the reduction theorem.

\subsection{The algorithm}
\label{sec:algorithm}

In its most general form, the base learner chooses among distributions over
causal policies; often it simply chooses a policy, that is, a point mass.  We
call such a distribution $q$ a \emph{decision}.  It may put weight on
policies outside the benchmark class $\Pi$, so improper learning is
allowed.  The algorithm follows the decision's
\emph{reference}, the averaged counterfactual trajectory
\begin{equation}
 \bar y_t(q)=\bigl(\bar x_t(q),\bar u_t(q)\bigr)=\E_{\pi\sim q}\,y_t^\pi
 \quad(t\le T),
 \qquad
 \bar x_{T+1}(q)=\E_{\pi\sim q}\,x_{T+1}^\pi,
 \label{eq:mixture-reference}
\end{equation}
assumed to exist.  The decisions available to the base learner form a set
$\cQ$, which need not be convex.

Two properties of the reference matter below.  It is
determined by the history $\mathcal H_t$, so the algorithm can compute it
before acting and, after each round, evaluate the cost of every decision, as
in full-information online learning.  And it obeys the dynamics:
\begin{equation}
 \bar x_1(q)=x_{\mathrm{ini}},
 \qquad
 \bar x_{t+1}(q)=A\bar x_t(q)+B\bar u_t(q)+w_t
 \quad(t\le T).
 \label{eq:reference-feasible}
\end{equation}
For a single policy, \eqref{eq:reference-feasible} is the rerun
\eqref{eq:policy-rerun}; for a distribution, it follows by averaging
\eqref{eq:policy-rerun}, because the dynamics are linear.  We call a family
of references with these two properties \emph{feasible}; later sections
sometimes specify a decision directly by such a reference.

On each round, the \emph{base learner} outputs a decision $q_t\in\cQ$, then
observes a loss function $h_t:\cQ\to\R$ and outputs its next decision.  Its
regret against a fixed $q\in\cQ$ is $\sum_t[h_t(q_t)-h_t(q)]$.  Exponential
weights over the distributions on a finite class and online mirror descent
over a convex set are standard examples
\cite{cesabianchi2006,shalevshwartz2012}.  In
\cref{alg:counterfactual-tracking}, the base learner is also charged a
switching cost for each change of decision, whose role \cref{sec:recovery}
explains.

\Cref{alg:counterfactual-tracking} turns decision $q_t$ into actual input $u_t$.  If the
physical state were on the reference, $x_t=\bar x_t(q_t)$, the algorithm
would simply play the input that $q_t$ recommends, $\bar u_t(q_t)$.  In
general, however, the physical state differs from $\bar x_t(q_t)$, and for
the base learner's regret guarantee to transfer to online control, the
algorithm must remove this difference over time.  It therefore plays
\[
 u_t=\bar u_t(q_t)+\psi\bigl(x_t-\bar x_t(q_t)\bigr),
\]
where the \emph{corrective law} $\psi$ maps the \emph{tracking error}
$e_t=x_t-\bar x_t(q_t)$ to a correction.  The algorithm thus \emph{tracks}
the reference, hence its name.  Why this works is the subject of the rest of
the section.  The crucial observation is that the tracking error evolves as
a linear system with the same dynamics as the original one, driven by the
base learner's changes of decision instead of by the disturbances
(\cref{sec:references}).

\begin{algorithm}[htbp]
\caption{Counterfactual tracking}
\label{alg:counterfactual-tracking}
\small
\textbf{Input:} horizon \(T\); Lipschitz bound \(L\); base learner on
\(\cQ\); recovery pair \((V,\psi)\) (\cref{sec:recovery}).
\begin{algsteps}
\item Receive the first decision \(q_1\) from the base learner.
\item For each round \(t=1,\ldots,T\):
  \begin{algsteps}
  \item Compute \(\bar y_t(q_t)\), set
  \(e_t\gets x_t-\bar x_t(q_t)\), and apply
  \(u_t\gets\bar u_t(q_t)+\psi(e_t)\).
  \item Observe \(c_t\) and \(x_{t+1}\), and reconstruct
  \(w_t\gets x_{t+1}-Ax_t-Bu_t\).
  \item Form the loss \(h_t(q)\gets c_t(\bar y_t(q))\) (or another loss,
  \cref{sec:reduction-theorem}) and
  \(d_t^V(q,q')\gets V\bigl(\bar x_{t+1}(q)-\bar x_{t+1}(q')\bigr)\) on
  \(\cQ\).
  \item Pass the loss \(h_t\) and the switching cost \(Ld_t^V\) to the base
  learner, and receive its next decision \(q_{t+1}\) (stop if $t=T$).
  \end{algsteps}
\end{algsteps}
\end{algorithm}

\subsection{Why switches matter: the tracking-error identity}
\label{sec:references}
\label{sec:error-identity}

Call every change
$q_{t+1}\ne q_t$, however small, a \emph{switch} of the base learner. How do switches affect the physical state?  Since the physical state remembers the inputs of the old decisions,
$x_{t+1}\ne\bar x_{t+1}(q_{t+1})$ in general.  To see how this error evolves, fix any online-control algorithm (generating $u_{1:T}$)
and any base learner (generating $q_{1:T}$).  Write $e_t=x_t-\bar x_t(q_t)$ for the
tracking error and $v_t=u_t-\bar u_t(q_t)$ for the corrective input.  Set $q_{T+1}=q_T$, so that
there is no terminal switch, and define, for $t\le T$, the \emph{switch
discrepancy} $\delta_{t+1}=\bar x_{t+1}(q_t)-\bar x_{t+1}(q_{t+1})$, the
difference at time $t+1$ between the states on the trajectories of the old
decision $q_t$ and the new decision $q_{t+1}$.  Then
\begin{equation}
 e_{t+1}
 =\underbrace{x_{t+1}-\bar x_{t+1}(q_t)}_{=\,Ae_t+Bv_t}
 +\underbrace{\bar x_{t+1}(q_t)-\bar x_{t+1}(q_{t+1})}_{=\,\delta_{t+1}}
 =Ae_t+Bv_t+\delta_{t+1},
 \label{eq:error-recursion}
\end{equation}
where the first underbraced identity subtracts \eqref{eq:reference-feasible}
for $q_t$ from \eqref{eq:system}, and the common disturbance $w_t$ cancels:
\[
 x_{t+1}-\bar x_{t+1}(q_t)
 =A\bigl(x_t-\bar x_t(q_t)\bigr)+B\bigl(u_t-\bar u_t(q_t)\bigr)+(w_t-w_t).
\]

Identity \eqref{eq:error-recursion} explains why
\cref{alg:counterfactual-tracking} is designed as it is.  It says that,
whatever the decisions and inputs, the tracking error obeys the dynamics of
the system \eqref{eq:system}, with $(e_t,v_t,\delta_{t+1})$ in place of
$(x_t,u_t,w_t)$.  If the base learner never switched, every $\delta_{t+1}$
would be zero, and an algorithm whose corrective input vanishes at zero
error would keep $e_t=0$.  It would play the recommended input and pay
exactly $c_t(\bar y_t(q_t))$, the loss passed to the base learner, so its
control regret would equal the base learner's regret.

When the base learner does switch, each discrepancy enters the error system
just as a disturbance enters the original system, and feedback can remove
it.  Because \eqref{eq:error-recursion} does not involve the references, a
single feedback $v_t=\psi(e_t)$, designed once for $(A,B)$, can steer the
error back toward zero after every switch
(\cref{fig:counterfactual-tracking}).  This is the corrective input of
\cref{alg:counterfactual-tracking}, with $\psi(0)=0$ (\cref{sec:recovery}).
Unlike $w_t$, moreover, the discrepancies are generated by the base learner,
which can keep them small.

\begin{figure}[htbp]
\centering
\begin{tikzpicture}[x=1.5cm,y=0.95cm,font=\small,
  >={Latex[length=2mm]}]
  \draw[->,ink] (0.55,0.30) -- (7.25,0.30) node[right] {time};
  \draw[->,ink] (0.55,0.30) -- (0.55,3.42) node[above] {state};

  \coordinate (oldt) at (1.05,0.90);
  \coordinate (oldn) at (2.45,1.25);
  \coordinate (oldtwo) at (3.85,0.95);
  \coordinate (oldthree) at (5.25,1.18);
  \coordinate (oldfour) at (6.65,1.00);

  \coordinate (newt) at (1.05,1.90);
  \coordinate (newn) at (2.45,2.55);
  \coordinate (newtwo) at (3.85,2.18);
  \coordinate (newthree) at (5.25,2.52);
  \coordinate (newfour) at (6.65,2.35);

  \coordinate (phystwo) at (3.85,1.58);
  \coordinate (physthree) at (5.25,2.25);
  \coordinate (physfour) at (6.65,2.32);

  \fill[refred!7]
    (newn) -- (newtwo) -- (newthree) -- (newfour) --
    (physfour) -- (physthree) -- (phystwo) -- (oldn) -- cycle;

  \draw[line width=1.15pt,refblue] (oldt) -- (oldn);
  \draw[densely dashed,line width=0.9pt,refblue!75]
    (oldn) -- (oldtwo) -- (oldthree) -- (oldfour);
  \draw[densely dashed,line width=0.9pt,refred!75]
    (newt) -- (newn);
  \draw[line width=1.15pt,refred]
    (newn) -- (newtwo) -- (newthree) -- (newfour);

  \draw[line width=1.25pt,ink]
    (oldn) -- (phystwo) -- (physthree) -- (physfour);

  \draw[densely dotted,ink!60]
    (2.45,0.30) -- (2.45,3.16);
  \node[above,ink] at (2.45,3.16) {reference switch};

  \draw[<->,ink]
    (oldn) --
    node[right=2pt,inner sep=1pt] {$\delta_{t+1}$}
    (newn);
  \draw[<->,ink!45]
    (phystwo) --
    node[left=2.5pt,inner sep=0.5pt,text=ink!75,font=\footnotesize] {$e_{t+2}$}
    (newtwo);

  \foreach \p in {oldt,oldn,oldtwo,oldthree,oldfour}
    \fill[refblue] (\p) circle[radius=1.25pt];
  \foreach \p in {newt,newn,newtwo,newthree,newfour}
    \fill[refred] (\p) circle[radius=1.25pt];
  \foreach \p in {oldn,phystwo,physthree,physfour}
    \fill[ink] (\p) circle[radius=1.35pt];

  \node[refred] at (5.95,2.92) {new reference};
  \node[ink] at (5.95,1.86) {physical trajectory};
  \node[refblue] at (5.95,0.74) {old reference};
  \node[below=2pt] at (1.05,0.30) {$t$};
  \node[below=2pt] at (2.45,0.30) {$t+1$};
  \node[below=2pt] at (3.85,0.30) {$t+2$};
  \node[below=2pt] at (5.25,0.30) {$t+3$};
  \node[below=2pt] at (6.65,0.30) {$t+4$};
\end{tikzpicture}
\caption{Counterfactual tracking after one switch.  Up to time~$t+1$, the
physical state follows the blue reference.  After round~$t$, the base
learner switches to the red reference, so the tracking error at time~$t+1$
is the switch discrepancy $\delta_{t+1}$.  The corrective law then steers
the state onto the red reference, shrinking the error to $e_{t+2}$ and
beyond; the shaded region is the tracking error whose cost
\cref{sec:recovery} bounds.  Only the state coordinate is shown.}
\label{fig:counterfactual-tracking}
\end{figure}
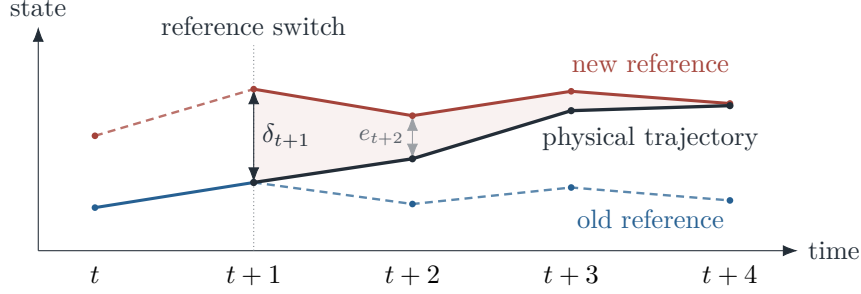

The base learner's losses $h_t$ do not see the extra cost
$c_t(y_t)-c_t(\bar y_t(q_t))$ of this error.  For its regret to transfer, this
cost must be charged to the base learner when it switches, since switches are
what create the error.  This poses the central problem: to price each
switch by a function of its discrepancy $\delta_{t+1}$ alone so that,
whatever the sequence of switches, the prices of all switches together bound
the total extra cost.  \Cref{sec:recovery} constructs such prices.  Under the
assumptions of the converse, \cref{sec:minimax-characterization} shows that
pricing each switch by its least recovery cost over the remaining horizon
loses at most a factor depending only on $A$, $B$, and $\norm{\cdot}$.

\subsection{Pricing a switch: recovery norms}
\label{sec:recovery}

What does a switch cost?  Since $y_t-\bar y_t(q_t)=(e_t,v_t)$, for
$L$-Lipschitz costs the extra cost $c_t(y_t)-c_t(\bar y_t(q_t))$ is at most
$L\norm{(e_t,v_t)}$, so the total extra cost is at most $L$ times the total
deviation $\sum_t\norm{(e_t,v_t)}$.  In the physical run, this
deviation mixes all switches: discrepancies created at different times
overlap in the error and are corrected together.  To charge each switch
separately, we price its discrepancy in an imagined world in which no later
switch occurs.  There the error starts at $\delta$ and evolves under the
corrective law alone, $e_0=\delta$ and $e_{k+1}=Ae_k+B\psi(e_k)$, and the
natural price of $\delta$ is the total deviation it causes,
\[
 V(\delta)=\sum_{k\ge0}\norm{(e_k,\psi(e_k))}.
\]

These stand-alone prices add up in the physical run, for two reasons.
Removing an error means taking one corrective step and then removing what
remains, so $V(e)=\norm{(e,\psi(e))}+V(Ae+B\psi(e))$.  And a new discrepancy
that lands on an existing error should cost at most the sum of the two
prices, $V(e+\delta)\le V(e)+V(\delta)$.  In the physical run with
$v_t=\psi(e_t)$, identity \eqref{eq:error-recursion} writes the next error as
the error left after one corrective step plus the new discrepancy,
$e_{t+1}=\bigl(Ae_t+B\psi(e_t)\bigr)+\delta_{t+1}$, so the two properties give
\[
 \norm{(e_t,\psi(e_t))}+V(e_{t+1})
 \le V(e_t)+V(\delta_{t+1}).
\]
Summing over $t\le T$, with $e_1=0$, $\delta_{T+1}=0$, and $V(e_{T+1})\ge0$,
gives
\begin{equation}
 \sum_{t=1}^T\norm{(e_t,\psi(e_t))}
 \le\sum_{t=1}^{T-1}V(\delta_{t+1}).
 \label{eq:tracking-deviation-sum}
\end{equation}
Each switch is thus charged a price of its discrepancy alone, whatever later
switches do: this is the price sought in \cref{sec:references}.

The formal definition keeps the two properties, relaxing the first to an
inequality and obtaining the second by requiring $V$ to be a norm.  It asks
for $V$ only on the \emph{reachable subspace}
\begin{equation}
 \Reach=\operatorname{im}
 \begin{bmatrix}B&AB&\cdots&A^{d_x-1}B\end{bmatrix},
 \label{eq:reachable}
\end{equation}
because $\Reach$ contains every switch discrepancy and every tracking error
(\cref{lem:relative-reachability}).

\begin{definition}[Recovery norm and corrective law]
\label{def:correction-certificate}
A norm $V:\Reach\to[0,\infty)$ and a Borel measurable map
$\psi:\Reach\to\R^{d_u}$ form a \emph{recovery pair} if
\begin{equation}
 \norm{(e,\psi(e))}+V(Ae+B\psi(e))\le V(e)
 \qquad(e\in\Reach).
 \label{eq:correction-certificate}
\end{equation}
We then call $V$ the \emph{recovery norm} and $\psi$ the \emph{corrective law}.
\end{definition}

For every recovery pair, the argument above gives
\eqref{eq:tracking-deviation-sum} (\cref{lem:certificate-tracking}), and
\eqref{eq:correction-certificate} at $e=0$ forces $\psi(0)=0$.

Restricting $V$ to $\Reach$ is what makes recovery pairs exist for every
system $(A,B)$, including unstabilizable ones (\cref{ex:network}).  The
system is controllable on $\Reach$, so a linear law $\psi(e)=-Ke$ can make
the error decay there, and its stand-alone price is then a norm that
satisfies \eqref{eq:correction-certificate} with equality
(\cref{prop:fixed-feedback-certificate}).  For instance, for
$x_{t+1}=x_t+u_t+w_t$ with the Euclidean norm, $\psi(e)=-e$ removes an error
in one round, and $V(e)=\norm{(e,-e)}_2=\sqrt2\,|e|$.  Finding recovery pairs
is a well-studied problem in control: \eqref{eq:correction-certificate} is
the control Minkowski--Lyapunov inequality for norm costs
\cite{rakovic2021control}.  Better pairs, with smaller recovery norms, give
better counterfactual-tracking algorithms.

\subsection{The reduction theorem}
\label{sec:reduction-theorem}

By \eqref{eq:tracking-deviation-sum}, the switches add at most
$L\sum_{t<T}V(\delta_{t+1})$ to the cost.  \Cref{alg:counterfactual-tracking}
charges this price to the base learner.  Since $\delta_{t+1}$ depends on the
next decision, which the base learner has yet to choose, the algorithm quotes
the price of every possible move: moving from $q$ to $q'$ costs
$Ld_t^V(q,q')$, where
\begin{equation*}
 d_t^V(q,q')=V\bigl(\bar x_{t+1}(q)-\bar x_{t+1}(q')\bigr).
\end{equation*}
The move actually made costs $Ld_t^V(q_t,q_{t+1})=LV(\delta_{t+1})$.  We call
the change in the base learner's decision its \emph{movement}, and
$\sum_{t<T}d_t^V(q_t,q_{t+1})$ its \emph{recovery-weighted movement}.  A base
learner need not penalize movement explicitly if its updates keep movement
small, as exponential weights does (\cref{sec:finite-policy-libraries}).

The loss $h_t(q)=c_t(\bar y_t(q))$ is not the only choice.  For a general
loss $h_t$, write the excess over the cost at the reference as
\begin{equation}
 J_t(q)=h_t(q)-c_t(\bar y_t(q)).
 \label{eq:jensen-gain}
\end{equation}
The default loss has $J_t=0$.  The only alternative we use is the averaged
loss $h_t(q)=\E_{\pi\sim q}c_t(y_t^\pi)$.  It is linear in $q$, so standard
experts algorithms apply, and since $c_t$ is convex, Jensen's inequality
gives $J_t\ge0$.  We therefore call $J_t$ the \emph{Jensen gain}.  It is a
saving: the base learner is charged $h_t(q_t)$, while the algorithm's cost
is compared with $c_t(\bar y_t(q_t))=h_t(q_t)-J_t(q_t)$.

\Needspace{6\baselineskip}
\begin{theorem}[Counterfactual-tracking reduction]
\label{thm:master-reduction}
Let the reference family $(\bar y_t)_{t\le T}$ be feasible.  Run
\cref{alg:counterfactual-tracking} with any base learner and any recovery
pair \((V,\psi)\), with any loss \(h_t:\cQ\to\R\) in Step~2(c).  On every realization
on which each \(c_t\) is \(L\)-Lipschitz with respect to \(\norm{\cdot}\) on
the segment between \(y_t\) and \(\bar y_t(q_t)\), every comparator
\(q\in\cQ\) satisfies
\begin{equation}
 \sum_{t=1}^T\bigl[c_t(y_t)-h_t(q)\bigr]
 \le
 \underbrace{\sum_{t=1}^T
   \bigl[h_t(q_t)-h_t(q)\bigr]}_{\text{online-learning regret}}
 +\underbrace{L\sum_{t=1}^{T-1}
   d_t^V(q_t,q_{t+1})}_{\text{switching cost}}
 -\underbrace{\sum_{t=1}^T J_t(q_t)}_{\text{Jensen gain}}.
 \label{eq:master-certificate}
\end{equation}
\end{theorem}

\begin{proof}
By \eqref{eq:jensen-gain},
\[
 c_t(y_t)-h_t(q)
 =\bigl[c_t(y_t)-c_t(\bar y_t(q_t))\bigr]
  -J_t(q_t)
  +\bigl[h_t(q_t)-h_t(q)\bigr].
\]
Since $y_t-\bar y_t(q_t)=(e_t,\psi(e_t))$, Lipschitzness bounds the first
bracket by $L\norm{(e_t,\psi(e_t))}$, and \eqref{eq:tracking-deviation-sum},
with $V(\delta_{t+1})=d_t^V(q_t,q_{t+1})$, bounds their sum by
$L\sum_{t<T}d_t^V(q_t,q_{t+1})$.  Summing over $t$ proves
\eqref{eq:master-certificate}.
\end{proof}

If $\cQ$ contains the point masses on $\Pi$, both the default and the
averaged loss satisfy $h_t(\pi)=c_t(y_t^\pi)$.  Taking the supremum over
comparators $q=\pi\in\Pi$ in \eqref{eq:master-certificate} then gives the
reduction announced in \cref{sec:contributions}:
\[
 \Reg_T(\Pi)
 \le\sup_{\pi\in\Pi}\sum_{t=1}^T\bigl[h_t(q_t)-h_t(\pi)\bigr]
 +L\sum_{t=1}^{T-1}d_t^V(q_t,q_{t+1})
 -\sum_{t=1}^TJ_t(q_t).
\]

The bound holds on every realization, so it remains to bound the base
learner's regret on the losses and switching costs it actually faces.  In
an adaptive environment, these depend on its past decisions through past
inputs, so a base learner that samples its decisions needs a guarantee
against adaptive adversaries.  In an oblivious environment, $h_t$ and
$d_t^V$ do not depend on the inputs, and an oblivious guarantee suffices.

Two extensions add one term each (\cref{cor:reduction-corrections}).  If
all references start from a common state $\bar x_1\ne x_{\mathrm{ini}}$ with
$e_1=x_{\mathrm{ini}}-\bar x_1\in\Reach$, the bound gains $LV(e_1)$, the
price of the initial error.  If a benchmark policy $\pi$ is only
approximated by a decision $q\in\cQ$, the regret to $\pi$ gains the
\emph{representation error} $\sum_t[h_t(q)-c_t(y_t^\pi)]$.

The base learner may describe $\cQ$ by parameters and exploit their
geometry.  For a finite class, $\cQ$ is the simplex and the base learner
runs exponential weights (\cref{sec:policies}).  For the $\ell_1$ response
ball, a decision is a disturbance response and the base learner runs mirror
descent in its coefficients (\cref{sec:responses}).  For truncated
disturbance--action control (\cref{sec:truncated-dac}) and spectral filters
(\cref{sec:spectral-compression}), $\cQ$ is a convex parameter set that only
approximates $\Pi$, and the gap enters as the representation error.

\section{The minimax converse}
\label{sec:minimax-characterization}
\label{sec:optimal-correction}

A general online-control algorithm need not follow any reference: it may
play inputs that no benchmark policy would choose.  Could such an algorithm
achieve substantially lower regret than counterfactual tracking?  This
section proves a \emph{converse} to the reduction: under the assumptions of
\cref{thm:minimax-completeness}, no algorithm improves on the minimax regret
of counterfactual tracking by more than a factor that depends only on the
system and the norm $\norm{\cdot}$.  The converse prices each switch by its
least recovery cost, which depends on the time left.

\subsection{Remaining-horizon recovery}
\label{sec:finite-horizon-recovery}

What does a switch discrepancy cost when only $h$ rounds remain?  In the
imagined world of \cref{sec:recovery}, where no later switch occurs, its
price is the total deviation over those rounds, minimized over corrective
inputs.  Set $V_0\equiv0$ and
\begin{equation}
 V_h(\delta)
 =\inf_{\substack{e_0=\delta\\ e_{s+1}=Ae_s+Bv_s}}
   \sum_{s=0}^{h-1}\norm{(e_s,v_s)},
 \qquad h\ge1,\quad \delta\in\Reach.
 \label{eq:remaining-horizon-value}
\end{equation}
There is no terminal constraint, because deviations after the horizon cost
nothing.  For $h\ge1$, $V_h$ is a finite norm, and it obeys the Bellman
equation
\[
 V_h(e)=\min_{v\in\R^{d_u}}\bigl\{\norm{(e,v)}+V_{h-1}(Ae+Bv)\bigr\}
\]
with a Borel minimizing law $\psi_h$ (\cref{thm:remaining-horizon-recovery}).

An error present on round $t$ affects the $T-t+1$ costs of rounds
$t,\ldots,T$, so counterfactual tracking can correct it with the
time-varying law $v_t=\psi_{T-t+1}(e_t)$ in place of $\psi(e_t)$.  A
discrepancy $\delta_{t+1}$ first affects round $t+1$, so it is priced by
$V_{T-t}$:
\begin{equation}
 d_{t,T}^\star(q,q')
 =V_{T-t}\bigl(\bar x_{t+1}(q)-\bar x_{t+1}(q')\bigr),\qquad t<T.
 \label{eq:optimal-finite-pullback}
\end{equation}
Replacing \eqref{eq:correction-certificate} by the Bellman equation in the
telescoping argument of \cref{sec:recovery} shows that
\cref{thm:master-reduction} holds with $d_{t,T}^\star$ in place of $d_t^V$
(\cref{cor:finite-horizon-reduction}).

This is the least price that can be charged to each switch separately.
Take any causal correction scheme, which chooses $v_s$ before seeing
$\delta_{s+1}$, and any prices $N_s$ with $N_s(0)=0$ such that, under
\eqref{eq:error-recursion} with $e_1=0$,
\[
 \sum_{s=1}^T\norm{(e_s,v_s)}\le\sum_{s=1}^{T-1}N_s(\delta_{s+1})
 \quad\text{for all }\delta_2,\ldots,\delta_T\in\Reach.
\]
Then $N_t(\delta)\ge V_{T-t}(\delta)$ for all $t<T$ and $\delta\in\Reach$
(\cref{thm:finite-horizon-minimality}): a single discrepancy $\delta$ after
round $t$ must pay alone for its whole recovery over the remaining rounds.

\subsection{Minimax equivalence}
\label{sec:minimax-completeness}

Counterfactual tracking turns online control into a \emph{switching
problem}: an online-learning problem in which the learner chooses a
decision, pays the cost at its reference, and pays a switching cost for
each switch (\cref{thm:master-reduction}).  In the converse, $\cQ$ is the
set of distributions over $\Pi$ whose references \eqref{eq:mixture-reference}
exist.  A \emph{switching learner} is a base learner on $\cQ$ that chooses
$q_t$ using only its private randomness and the \emph{switching history}
$\mathcal H_t^{\mathrm{sw}}=(q_{1:t-1},c_{1:t-1},w_{1:t-1})$; we call such a
choice \emph{predictable}.  This history records decisions where the
history $\mathcal H_t$ of \cref{sec:model} records states and inputs: the
switching problem fixes no corrective law, so it determines no physical
trajectory.
It is charged the default loss $h_t(q)=c_t(\bar y_t(q))$ and the
remaining-horizon switching cost $Ld_{t,T}^\star$ of
\eqref{eq:optimal-finite-pullback}.  Its \emph{switching regret} is the
right side of \eqref{eq:master-certificate} with these choices, maximized
over benchmark policies:
\begin{equation*}
 \operatorname{SReg}_T(q_{1:T})
 =\sum_{t=1}^T h_t(q_t)
 -\inf_{\pi\in\Pi}\sum_{t=1}^T h_t(\pi)
 +L\sum_{t=1}^{T-1}d_{t,T}^\star(q_t,q_{t+1}),
\end{equation*}
where $h_t(\pi)=c_t(y_t^\pi)$ is the loss of the point mass at $\pi$.  By
\cref{cor:finite-horizon-reduction}, tracking a switching learner with the
laws $\psi_{T-t+1}$ gives an online-control algorithm that satisfies, on
every realization on which each $c_t$ is $L$-Lipschitz,
\begin{equation}
 \Reg_T(\Pi)\le\operatorname{SReg}_T(q_{1:T}).
 \label{eq:tracking-switching-bound}
\end{equation}
Conversely, can every online-control algorithm be converted into a
switching learner that loses at most a system-dependent factor?

We compare the two problems through their minimax values.  Let
$\Alg^{\mathrm{ctrl}}$ be the class of online-control algorithms and
$\Alg^{\mathrm{sw}}$ the class of switching learners.  Let
$\Env_T^{\mathrm{ctrl}}$ be the adaptive environments of \cref{sec:model}
with convex globally $L$-Lipschitz costs, and define $\Env_T^{\mathrm{sw}}$
in the same way, with $\mathcal H_t^{\mathrm{sw}}$ in place of $\mathcal H_t$.  With
unconstrained randomized inputs, the algorithm's cost and the best benchmark
cost can both have infinite expectation, so we evaluate regret by its
\emph{upper expectation} $\E^\uparrow Z=\lim_{m\to\infty}\E[\max\{Z,-m\}]$.
It is well defined for every measurable $Z$, since the truncations are
bounded below and decrease in $m$, and it equals $\E Z$ whenever $\E Z$
exists.  Define
\begin{align*}
 \MinReg_T^{\mathrm{ctrl}}
 &=\inf_{\mathcal A\in\Alg^{\mathrm{ctrl}}}
   \sup_{\mathcal E\in\Env_T^{\mathrm{ctrl}}}
   \E^\uparrow_{\mathcal A,\mathcal E}[\Reg_T(\Pi)],
\\
 \MinReg_T^{\mathrm{sw}}
 &=\inf_{\mathcal L\in\Alg^{\mathrm{sw}}}
   \sup_{\mathcal E\in\Env_T^{\mathrm{sw}}}
   \E^\uparrow_{\mathcal L,\mathcal E}[\operatorname{SReg}_T(q_{1:T})].
\end{align*}

The converse bounds $\MinReg_T^{\mathrm{sw}}$ by $\MinReg_T^{\mathrm{ctrl}}$
times a factor built from the least infinite-horizon recovery cost
\begin{equation}
 \Vstar(\delta)
 =\inf_{\substack{e_0=\delta\\ e_{s+1}=Ae_s+Bv_s}}
   \sum_{s=0}^{\infty}\norm{(e_s,v_s)},\qquad \delta\in\Reach,
 \label{eq:correction-value}
\end{equation}
which is finite because the system is controllable on $\Reach$.  The factor
\begin{equation}
 C_{\mathrm{rec}}
 =1+\sup_{\substack{e\in\Reach,\ v\in\R^{d_u}\\(e,v)\ne(0,0)}}
 \frac{\norm{(e,v)}+\Vstar(e)+\Vstar(Ae+Bv)}{\norm{(e,v)}}
 \label{eq:recovery-distortion}
\end{equation}
depends only on $A$, $B$, and $\norm{\cdot}$.  The proof modifies the
revealed costs, so it needs benchmark trajectories that do not depend on
them.

\begin{definition}[Loss-independent policy]
\label{def:loss-independent-policy}
A causal policy is \emph{loss-independent} if its inputs do not depend on
previously revealed costs.  With its seed fixed as in \cref{sec:benchmarks},
it has the form $u_t^\pi=\pi_t(x_t^\pi,w_{1:t-1})$, where the maps $\pi_t$
encode time dependence and any internal memory.
\end{definition}

State-feedback policies with fixed gains, DAC policies with fixed
coefficients, and policies given by a fixed disturbance response are
loss-independent.  The algorithms that learn them may still use past costs.

\Needspace{16\baselineskip}
\begin{theorem}[Minimax equivalence of control and switching]
\label{thm:minimax-completeness}
Under the protocol of \cref{sec:model}, suppose inputs are unconstrained
and costs are convex and globally $L$-Lipschitz.  Let $\Pi$ be a nonempty
loss-independent class of causal policies.  Suppose the points $y_t^\pi$ and
the infimum in \eqref{eq:regret} are measurable, and that, on every round
$t$ and for every point $y$, a decision $q\in\cQ$ minimizing
$\norm{y-\bar y_t(q)}$ exists and can be chosen as a measurable function of
$y$ and $\mathcal H_t^{\mathrm{sw}}$.  Then
$C_{\mathrm{rec}}<\infty$, and for every $T\ge1$ the
minimax values defined above satisfy
\begin{equation*}
 \MinReg_T^{\mathrm{ctrl}}
 \le \MinReg_T^{\mathrm{sw}}
 \le C_{\mathrm{rec}}\,\MinReg_T^{\mathrm{ctrl}}.
\end{equation*}
Consequently, for every fixed
$(A,B,\norm{\cdot},\Pi,L,W,x_{\mathrm{ini}})$,
\begin{equation}
 \MinReg_T^{\mathrm{ctrl}}=o(T)
 \quad\Longleftrightarrow\quad
 \MinReg_T^{\mathrm{sw}}=o(T).
 \label{eq:no-regret-completeness}
\end{equation}
\end{theorem}

These conditions hold for finite classes, where $\cQ$ is a simplex, and for
classes whose points $y_t^\pi$ depend continuously on a parameter in a
compact set (\cref{app:minimax-proof}).  Since $C_{\mathrm{rec}}$ depends on
neither $T$ nor $\Pi$, the equivalence \eqref{eq:no-regret-completeness}
also holds when the class changes with the horizon.

\paragraph{Proof idea.}
For the first inequality, fix a switching learner.  The history
$\mathcal H_t$ of the algorithm that tracks it is a function of
$\mathcal H_t^{\mathrm{sw}}$, so every environment in $\Env_T^{\mathrm{ctrl}}$ induces one in
$\Env_T^{\mathrm{sw}}$ that draws the same costs and disturbances.  Taking
upper expectations in \eqref{eq:tracking-switching-bound} proves the
inequality.  For the second, a switching learner simulates a given
online-control algorithm $\mathcal A$ and, before each round, selects a
decision $q_t$ whose reference point $\bar y_t=\bar y_t(q_t)$ is nearest to
the simulated state--input point $y_t^{\mathcal A}$ in $\norm{\cdot}$.
Identity \eqref{eq:error-recursion} holds for any inputs and decisions, so
the projection errors $(e_t,v_t)=y_t^{\mathcal A}-\bar y_t$ satisfy
\[
 \delta_{t+1}=e_{t+1}-(Ae_t+Bv_t).
\]
Since $V_{T-t}$ is a norm and $V_{T-t}\le\Vstar$, the switch after round $t$
costs at most $L[\Vstar(e_{t+1})+\Vstar(Ae_t+Bv_t)]$.  Shift the term
$L\Vstar(e_{t+1})$ to round $t+1$.  Round $t$ then carries the cost
difference $h_t(q_t)-c_t(y_t^{\mathcal A})\le L\norm{(e_t,v_t)}$ and two
recovery terms, in total at most
$L[\norm{(e_t,v_t)}+\Vstar(e_t)+\Vstar(Ae_t+Bv_t)]
\le L(C_{\mathrm{rec}}-1)\norm{(e_t,v_t)}$.  The learner therefore feeds
$\mathcal A$ the costs
\[
 \widetilde c_t(y)
 =\frac{c_t(y)+L(C_{\mathrm{rec}}-1)\inf_{q\in\cQ}\norm{y-\bar y_t(q)}}{C_{\mathrm{rec}}}.
\]
Since $\bar y_t$ is linear in $q$, the reference points form a convex set,
whose distance function is convex and $1$-Lipschitz.  The costs
$\widetilde c_t$ are therefore again convex and $L$-Lipschitz.  The penalty
vanishes at every $y_t^\pi$, and loss independence keeps these points fixed
when the costs change.  Hence the switching regret is at most
$C_{\mathrm{rec}}$ times the regret of $\mathcal A$ on $\widetilde c_{1:T}$.
The projection and the optimal recovery may require optimization oracles, so
the argument is information-theoretic (\cref{app:minimax-proof}).

When is the switching problem just online learning?  When all decisions
share the next state $\bar x_{t+1}(q)$, so that no switch creates a
discrepancy (\cref{cor:oco-consistency}).

\subsection{Infinite-horizon recovery}
\label{sec:computable-correction}

A corrective law used on every round, as in
\cref{alg:counterfactual-tracking}, needs a price that does not depend on
the time left.  Letting $h\to\infty$ gives one: $V_h$ increases to $\Vstar$,
the value of an optimal-control problem with a norm as its stage cost.  The
limit is a finite continuous norm on $\Reach$ and obeys the Bellman equation
\begin{equation}
 \Vstar(e)
 =\min_{v\in\R^{d_u}}
   \{\norm{(e,v)}+\Vstar(Ae+Bv)\},
 \label{eq:stationary-bellman}
\end{equation}
with a Borel minimizing law $\psi_\star$ of linear growth, so
$(\Vstar,\psi_\star)$ is a recovery pair
(\cref{thm:optimal-correction-certificate}).  Equation
\eqref{eq:stationary-bellman} is an instance of the Minkowski--Bellman
equation \cite{rakovic2020bellman}.

The minimality of \cref{sec:finite-horizon-recovery} carries over.  A price
$N$ that does not depend on the time left, and satisfies the bound there
with $N_s=N$ at every horizon $T$, is at least $V_{T-t}$ for every $t<T$,
hence at least $\Vstar$
(\cref{thm:finite-horizon-minimality,thm:minimality}).  In particular,
$\Vstar\le V$ for every recovery pair $(V,\psi)$.  Computable pairs come from
stable linear corrective laws (\cref{app:fixed-feedback-correction}) and
from finite-horizon optimization with a terminal recovery norm, which
approaches $\Vstar$ from above (\cref{app:mpc-certificates}).  Outside
$\Reach$, recovery can have infinite cost: extended to the full state space,
$\Vstar$ is finite if and only if $(A,B)$ is stabilizable
(\cref{app:outside-reach}), a property that prior work commonly assumes
\cite{agarwal2019online,simchowitz2020improper,kumar2023} but our framework
does not require.

For quantitative bounds, summarize a recovery norm by its
\emph{worst-direction constant}
\begin{equation}
 M_V=\sup_{\substack{z\in\Reach\\\norm{z}_2=1}}V(z),
 \label{eq:euclidean-correction-factor}
\end{equation}
with $M_V=0$ if $\Reach=\{0\}$, and write $M_\star$ for $V=\Vstar$.  In the
algorithmic bounds below, any known upper bound on $M_V$ may replace it.
For the Euclidean norm,
$C_{\mathrm{rec}}\le2+M_\star(1+\norm{[A\ B]}_{\mathrm{op}})$.

\paragraph{An example of a non-stabilizable system.}
\label{sec:network-example}
The following system is not stabilizable, so analyses that assume
stabilizability exclude it.  Its switch discrepancies nevertheless lie in
$\Reach$, where $\Vstar$ is finite, and counterfactual tracking achieves
sublinear regret on it, for example over finite policy classes with bounded
diameters (\cref{sec:policies}).

\begin{example}[Conserved network]
\label{ex:network}
Let $D\in\R^{n\times m}$ be an oriented incidence matrix of a connected
unweighted graph with $n\ge2$ nodes and $m$ edges, and consider
\begin{equation*}
 x_{t+1}=x_t+Du_t+w_t.
\end{equation*}
Here $x_i$ is the quantity at node $i$, and each input coordinate is a signed
transfer along one edge, so inputs change the allocation but not its total.
The direction $\mathbf 1$ is uncontrollable with eigenvalue one, so the
system is not stabilizable.  Yet
$\Reach(I,D)=\operatorname{im}D=\mathbf 1^\perp$: every difference between
feasible references is a reallocation.
With $D^\dagger$ the Moore--Penrose pseudoinverse, $v=-D^\dagger e$ removes
such a difference $e$ in one step.  For
$\norm{(e,v)}=\sqrt{\norm{e}_2^2+\norm{v}_2^2}$ and $z\in\mathbf 1^\perp$,
\begin{equation}
 \Vstar(z)
 =\sqrt{\norm{z}_2^2+\norm{D^\dagger z}_2^2},
 \qquad
 M_\star=\sqrt{1+\frac{1}{\lambda_2(DD^\top)}},
 \label{eq:network-vstar}
\end{equation}
where $\lambda_2(DD^\top)$ is the algebraic connectivity of the graph
(proof in \cref{app:network-formula}).  Connectivity thus governs the
switching cost.
\end{example}

\section{Application I: learning general policy classes}
\label{sec:policies}

We now apply the reduction to a general policy class $\Pi$, which may mix
policies designed by different methods, such as linear controllers and
nonlinear learned policies.  The decisions are weights over $\Pi$, and their references are weighted
averages of the policies' counterfactual trajectories, so the policies need
not share a parameterization.  Counterfactual tracking attains regret
$O(\min\{T,\sqrt{T\log N}\})$ to $N$ policies, which is worst-case optimal in
$T$ and $N$ (\cref{sec:finite-policy-libraries}).  Under strong convexity,
the Jensen gain improves it to $O(\log N)$ (\cref{sec:strong-policies}).  A
prior extends both bounds to infinite classes
(\cref{sec:distributional-policies}).

Throughout this section, fix a recovery pair $(V,\psi)$ and $L>0$.  Costs
are convex and $L$-Lipschitz on the \emph{counterfactual hull}, the closed
convex hull of $\{y_t^\pi:\pi\in\Pi\}$, and $L$-Lipschitz on each segment
joining the actual point $y_t$ to its reference.  Unlike in the converse,
the benchmark policies may use the revealed cost history.

\subsection{Finite-class aggregation}
\label{sec:finite-policy-libraries}

Let $\Pi=\{\pi_1,\ldots,\pi_N\}$, with $N\ge2$, and write $y_t^i=y_t^{\pi_i}$.
A decision is a weight vector $p\in\simplex N$, with reference
$\bar y_t(p)=\sum_ip_iy_t^i$.  The base learner starts from uniform weights,
$p_{1,i}=1/N$.  After observing $c_t$, it evaluates the policy costs
$\mathsf c_{t,i}=c_t(y_t^i)$ and applies the exponential-weights update
\cite{cesabianchi2006,shalevshwartz2012}
\begin{equation}
 p_{t+1,i}
 =\frac{p_{t,i}\exp(-\eta\mathsf c_{t,i})}
        {\sum_{j=1}^N p_{t,j}\exp(-\eta\mathsf c_{t,j})}.
 \label{eq:finite-policy-ew}
\end{equation}
The update uses the averaged loss $h_t(p)=\sum_ip_i\mathsf c_{t,i}$ of
\cref{sec:reduction-theorem} and ignores the switching cost: the learning
rate alone keeps movement small.  The Jensen gain is
$J_t(p)=h_t(p)-c_t(\bar y_t(p))\ge0$.  Besides simulating the $N$ policies
and evaluating $\psi$, each round uses $N$ cost evaluations and
$O(N(d_x+d_u))$ scalar operations.

Two diameters enter the bound.  Define the \emph{trajectory norm}
\begin{equation}
 \norm{(e,v)}_V=\norm{(e,v)}+V(Ae+Bv)
 \label{eq:induced-trajectory-norm}
\end{equation}
on $\Reach\times\R^{d_u}$, which contains every difference of policy points
(\cref{lem:relative-reachability}), and set
\begin{align}
 D_t&=\max_{i,j}\norm{y_t^i-y_t^j},
 \label{eq:loss-range}\\
 D_{t,V}&=\max_{i,j}\norm{y_t^i-y_t^j}_V.
 \label{eq:correction-diameter}
\end{align}
For $(e,v)=y_t^i-y_t^j$, the disturbance cancels in the next-state gap,
$x_{t+1}^i-x_{t+1}^j=Ae+Bv$, so $\norm{y_t^i-y_t^j}_V$ adds the recovery cost
of this gap to the distance between the two points.  Thus $LD_t$ bounds the
range of the policy costs, and $LD_{t,V}\ge LD_t$ bounds the switching cost
per unit of movement in total variation.

\begin{theorem}[Finite-class aggregation]
\label{thm:finite-aggregation-main}
Under the assumptions of this section, let $0<D\le D_V$ be constants known
before play.  On every realization on which $D_t\le D$ and $D_{t,V}\le D_V$
for all $t$, counterfactual tracking with exponential weights
\eqref{eq:finite-policy-ew} as the base learner and
$\eta=\sqrt{2\log N}/(L\sqrt{TDD_V})$ satisfies
\begin{equation*}
 \Reg_T(\Pi)\le L\sqrt{2TDD_V\log N}-\sum_{t=1}^T J_t(p_t).
\end{equation*}
\end{theorem}

Exponential weights has online-learning regret at most
$(\log N)/\eta+(\eta L^2/8)\sum_tD_t^2$ and moves at most $\eta LD_t/4$ in
total variation on round $t$.  The switching costs therefore sum to at most
$(\eta L^2/4)\sum_{t<T}D_tD_{t,V}$, and balancing the terms in $\eta$ proves
the theorem.  \Cref{thm:finite-policy-library} allows any known bound on
$\sum_tD_t^2+\sum_{t<T}D_tD_{t,V}$ in place of $2TDD_V$ (proof in
\cref{app:finite-policy-proof}).

Is this rate optimal?  For short horizons, following one policy throughout
does better: its regret is at most $LDT$.  Running whichever of the two
algorithms has the smaller bound gives regret
\begin{equation}
 O\!\left(L\min\left\{TD,\sqrt{TDD_V\log N}\right\}\right).
 \label{eq:finite-policy-sharp-upper}
\end{equation}

\begin{proposition}[Finite-class lower bound]
\label{prop:finite-lower-main}
For every $T\ge1$ and $N\ge2$, consider the scalar system $x_{t+1}=u_t+w_t$
with $x_{\mathrm{ini}}=0$ and $w_t=0$, under the Euclidean norm and
$V=\Vstar$.  There is a class of $N$ open-loop policies, each a fixed input
sequence, with $D_t\le2\sqrt2$ and $D_{t,V}\le2(1+\sqrt2)$, such that every
online-control algorithm has expected regret at least
$cL\min\{T,\sqrt{T\log N}\}$ on some oblivious sequence of convex, globally
$L$-Lipschitz costs.  Here $c>0$ is a universal constant.
\end{proposition}

The costs are $c_t(x,u)=L\abs{u-\sigma_t}$ with independent random signs
$\sigma_t$, so every algorithm pays at least $L$ per round in expectation.
Split the rounds into $\min\{T,\lfloor\log_2N\rfloor\}$ blocks.  Each policy
plays a fixed input $\pm1$ on each block, and the best one follows the
majority sign of $\sigma_t$ on every block (\cref{prop:finite-policy-lower}).  With
$D=2\sqrt2$ and $D_V=2(1+\sqrt2)$, the upper bound
\eqref{eq:finite-policy-sharp-upper} therefore matches the lower bound, and
the dependence on $T$ and $N$ is worst-case optimal.  Individual classes may
be easier.

\subsection{Strong convexity pays for learning and recovery}
\label{sec:strong-policies}

The bound of \cref{thm:finite-aggregation-main} subtracts the Jensen gain,
which the algorithm earns by tracking the average trajectory rather than a
sampled policy.  Under strong convexity, this gain grows with the spread of
the policies' points under the current weights.  The learning and switching
terms also admit bounds linear in this spread, so the gain can pay for both.
Use the Euclidean norm $\norm{\cdot}=\norm{\cdot}_2$, fix $\mu>0$, and assume
in addition that costs are $\mu$-strongly convex on the counterfactual hull.
Measure the spread by
\begin{equation*}
 \mathsf V_t^y=\sum_i p_{t,i}\norm{y_t^i-\bar y_t(p_t)}_2^2,
 \qquad\text{so that}\qquad
 J_t(p_t)\ge\frac\mu2\mathsf V_t^y
\end{equation*}
by strong convexity.  Lipschitzness bounds the variance of the policy costs
under $p_t$ by $L^2\mathsf V_t^y$.  To compare this spread with recovery
cost, suppose a constant $\alpha\ge1$, known before play, bounds the
\emph{recovery distortion}:
\begin{equation}
 \norm{y_t^i-y_t^j}_V\le\alpha\norm{y_t^i-y_t^j}_2
 \qquad(i,j\le N,\ t<T).
 \label{eq:finite-policy-correction-ratio}
\end{equation}
Since $V(Ae+Bv)\le M_V\norm{[A\ B]}_{\mathrm{op}}\norm{(e,v)}_2$, with $M_V$
the worst-direction constant \eqref{eq:euclidean-correction-factor}, every
class may take
\begin{equation}
 \alpha=1+M_V\norm{[A\ B]}_{\mathrm{op}}.
 \label{eq:global-trajectory-distortion}
\end{equation}

\begin{theorem}[Horizon-independent finite-class aggregation]
\label{cor:horizon-independent-finite-library}
Under these assumptions, there are universal constants $c_0,C_0>0$ such
that counterfactual tracking with exponential weights
\eqref{eq:finite-policy-ew} as the base learner and
$\eta=c_0\mu/(L^2\alpha)$ satisfies
\begin{equation*}
 \Reg_T(\Pi)\le C_0\frac{L^2}{\mu}\alpha\log N,
\end{equation*}
\end{theorem}

By Lipschitzness, \eqref{eq:finite-policy-correction-ratio}, and the
regret--variance inequality for exponential weights \cite{vanderhoeven2022},
the online-learning and switching terms of round $t$ are at most
$C_1\eta L^2\alpha\mathsf V_t^y$, with $C_1$ universal
(\cref{lem:variance-aggregation}).  The chosen learning rate makes this at
most $(\mu/2)\mathsf V_t^y\le J_t(p_t)$, so the Jensen gain cancels these
terms on every round, and only $(\log N)/\eta$ remains.  No diameter bound is
needed: strong convexity and Lipschitzness on the counterfactual hull force
$D_t\le4L/\mu$, which also keeps $\eta$ small enough for
\cref{lem:variance-aggregation} (\cref{app:strong-policy-proof}).

Strongly convex costs cannot be globally Lipschitz (\cref{sec:cost-classes}),
so \cref{cor:horizon-independent-finite-library} assumes Lipschitzness only
on the counterfactual hull and on the segments joining $y_t$ to its
reference.  These segments depend on
the run, but a region containing them is known before play: with a linear
corrective law and a known bound on every $\norm{y_t^\pi}_2$,
\cref{lem:policy-tracking-region} gives an explicit ball.

\subsection{Infinite policy classes}
\label{sec:distributional-policies}

What if the class is infinite?  Exponential weights then starts from a
\emph{prior} $\nu$, a probability measure on $\Pi$, and updates to
$p_t(d\pi)\propto\exp(-\eta\sum_{s<t}c_s(y_s^\pi))\,\nu(d\pi)$.  Both bounds
carry over when $\pi\mapsto y_t^\pi$ is measurable and the references and
normalizing integrals exist (\cref{app:measurable-policy-proof}).  For a
comparison distribution $\rho\ll\nu$, define
$\Reg_T(\rho)=\sum_t c_t(y_t)-\E_{\pi\sim\rho}\sum_t c_t(y_t^\pi)$.  The
relative entropy $\KL(\rho\Vert\nu)$ then replaces $\log N$, as in
PAC-Bayesian bounds \cite{mcallester1999pac,alquier2024pacbayes}.  For example,
under the assumptions of \cref{sec:strong-policies}, with
\eqref{eq:finite-policy-correction-ratio} for all pairs of policies, the
learning rate of \cref{cor:horizon-independent-finite-library} gives
\[
 \Reg_T(\rho)\le C_0\frac{L^2}{\mu}\alpha\KL(\rho\Vert\nu).
\]
\Cref{thm:pac-bayes-policies,thm:horizon-independent-aggregation} give the
complete bounds.  The prior is part of the algorithm, not a probabilistic
model of the costs or disturbances.  With a nonatomic prior, a point mass has
infinite relative entropy.  To compete with a single policy $\pi$, compare
with a distribution $\rho$ concentrated near $\pi$; by Lipschitzness on the
counterfactual hull, this adds
$L\sum_t\E_{\pi'\sim\rho}\norm{y_t^{\pi'}-y_t^\pi}$ to the bound.  Related
approximations of a single predictor by a distribution are widely used in
statistical learning \cite{langford2002pacbayes,li2026pointwise}.  Exact
implementation requires the references $\bar y_t(p_t)$ and the normalizing
integrals.

\section{Application II: learning bounded disturbance responses}
\label{sec:responses}

Limiting how much bounded disturbances are amplified is a classical control
objective \cite{dahleh1987l1}.  System-level synthesis (SLS) expresses such
limits as constraints on a controller's disturbance response
\cite{anderson2019sls,wang2019sls}.  Even with full state observation,
optimal linear controllers for this objective can require internal memory
\cite{diazbobillo1992dynamic}.  This section takes disturbance responses as
decisions and shows that a bound on total response magnitude alone defines a
learnable class: it prescribes neither how quickly disturbance effects decay
nor how much internal memory a policy uses.

\subsection{Disturbance responses as decisions}
\label{sec:response-learning-update}

Throughout this section, $T\ge2$, $\norm{w_t}_2\le W$, and
$\norm{\cdot}=\norm{\cdot}_2$ is the Euclidean norm.  The system starts at
$x_{\mathrm{ini}}=0$, and any internal memory of a policy also starts at zero.
The \emph{disturbance response} of a linear time-invariant policy is the
sequence of stacked matrices $\Phi[i]=(\Phi_x[i],\Phi_u[i])$ such that a
disturbance $w_s=w$ contributes $\Phi_x[i]w$ to the state and $\Phi_u[i]w$ to
the input on round $s+i$ (\cref{fig:response-coefficients}).  We call $i\ge1$
the \emph{disturbance age} and $\Phi[i]$ the \emph{response coefficient} at
age $i$.  A disturbance enters the next state unchanged, and its later
contributions obey the dynamics, so every response lies in the set
$\cS_\infty$ of sequences that satisfy the SLS feasibility constraints
\cite{anderson2019sls,wang2019sls}
\begin{equation}
 \Phi_x[1]=I,\qquad\Phi_x[i+1]=A\Phi_x[i]+B\Phi_u[i]\quad(i\ge1).
 \label{eq:infinite-response-space}
\end{equation}
Every $\Phi\in\cS_\infty$ serves both as a decision and as a benchmark; we
write $\Reg_T(\mathfrak S)$ for the regret \eqref{eq:regret} to a set
$\mathfrak S\subseteq\cS_\infty$.  Since contributions of different
disturbances add, the reference of $\Phi$, which is also its benchmark
trajectory, is
\begin{equation}
 y_t^\Phi=\bar y_t(\Phi)=\sum_{i=1}^{t-1}\Phi[i]w_{t-i}.
 \label{eq:response-reference-infinite}
\end{equation}
This reference depends only on past disturbances, and
\eqref{eq:infinite-response-space} is exactly the condition under which it
obeys the dynamics \eqref{eq:reference-feasible} for every horizon and
disturbance sequence.  It is the counterfactual trajectory of the causal
policy $u_t=\sum_{i<t}\Phi_u[i]w_{t-i}$, even when no finite-dimensional
controller realizes $\Phi$.

\begin{figure}[H]
\centering
\begin{tikzpicture}[x=1.5cm,y=0.85cm,font=\small,>={Latex[length=2mm]}]
  \draw[->,ink] (-0.25,0) -- (6.75,0) node[right] {time};
  \foreach \x/\lab in {0/$s$,2/$s+1$,4/$s+2$,6/$s+3$}{
    \draw[ink] (\x,0.08) -- (\x,-0.08);
    \node[below] at (\x,-0.08) {\lab};
  }
  \node[draw=refblue,rounded corners=1.5pt,inner sep=3pt] (w) at (0,1.55)
    {$w_s=w$};
  \node[draw=ink,rounded corners=1.5pt,inner sep=3pt] (p1) at (2,1.1)
    {$\Phi[1]w$};
  \node[draw=ink,rounded corners=1.5pt,inner sep=3pt] (p2) at (4,1.1)
    {$\Phi[2]w$};
  \node[draw=ink,rounded corners=1.5pt,inner sep=3pt] (p3) at (6,1.1)
    {$\Phi[3]w$};
  \draw[->,refblue] (w) to[bend left=12] (p1);
  \draw[->,refblue] (w) to[bend left=18] (p2);
  \draw[->,refblue] (w) to[bend left=22] (p3);
  \node[ink] at (4,3.15) {contribution to the state--input point};
\end{tikzpicture}
\caption{The meaning of a response coefficient.  A one-time disturbance
$w_s=w$ contributes $\Phi[i]w$ to the state and input $i$ rounds later.  With
many disturbances, the contributions add as in
\eqref{eq:response-reference-infinite}.}
\label{fig:response-coefficients}
\end{figure}
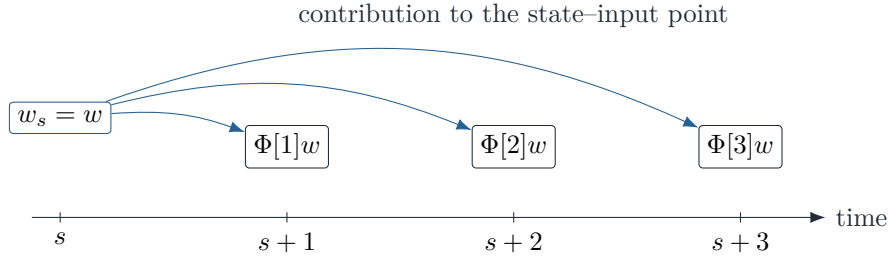

\subsection{The full \texorpdfstring{$\ell_1$}{l1} response ball and its learner}
\label{sec:full-response-ball}\label{sec:response-sharpness}

Fix a baseline $\Phi^0\in\cS_\infty$ and a radius $R\ge0$.  The \emph{full
$\ell_1$ response ball} is
\begin{equation}
 \mathfrak S^{(1)}(R)=\left\{\Phi\in\cS_\infty:
 \sum_{i\ge1}\norm{\Phi[i]-\Phi^0[i]}_{\mathrm F}\le R\right\}.
 \label{eq:infinite-l1-response-ball}
\end{equation}
Its responses satisfy $\norm{y_t^\Phi-y_t^{\Phi^0}}_2\le WR$, so the radius
limits disturbance amplification relative to the baseline.  Up to
dimension-dependent factors, the sum in \eqref{eq:infinite-l1-response-ball}
is comparable to the \emph{peak-to-peak gain} of $\Phi-\Phi^0$, the $\ell_\infty$-induced
gain that $\ell_1$ control bounds \cite{dahleh1987l1}
(\cref{app:response-peak-gain}).

The baseline need not be \emph{absolutely summable}, that is, satisfy
$\sum_i\norm{\Phi^0[i]}_{\mathrm F}<\infty$.  For example, the zero-input
response $\Phi_x^0[i]=A^{i-1}$, $\Phi_u^0=0$ always lies in $\cS_\infty$, so
the ball requires no stabilizable system.  A finite-dimensional linear
time-invariant controller is \emph{internally stabilizing} if, without
disturbances, the states of the system and of the controller converge to
zero from every initial condition.  Around an absolutely summable baseline,
which exists only for a stabilizable system, the ball contains the response
of every such controller for some finite, controller-dependent radius
(\cref{prop:full-response-scope}), and also responses that no
finite-dimensional controller realizes.

Which decisions does the base learner need?  Only the first $H=T-1$
coefficients affect the costs over the horizon.  The base learner keeps all
of them, since a response in the ball can place its entire deviation from
the baseline beyond any cutoff (\cref{prop:no-common-response-tail}).  Its
decisions are therefore prefixes in the affine set $\cS_H$ below, with
references
\begin{gather}
 \cS_H=\left\{\Phi=(\Phi[i])_{i\le H}:
 \Phi_x[1]=I,\ \Phi_x[i+1]=A\Phi_x[i]+B\Phi_u[i]\ (1\le i<H)\right\},
 \label{eq:response-prefix}\\
 \bar y_t(\Phi)=\sum_{i=1}^{\min\{H,t-1\}}\Phi[i]w_{t-i},
 \label{eq:response-reference}
\end{gather}
which are feasible (\cref{app:direct-response-proof}).  The decision set is
$\cS_H^{(1)}(R)=\{\Phi\in\cS_H:\sum_{i\le H}\norm{\Phi[i]-\Phi^0[i]}_{\mathrm F}
\le R\}$.  It contains the first $H$ coefficients of every response in the
ball, which determine its trajectory over the horizon.

The default loss $h_t(\Phi)=c_t(\bar y_t(\Phi))$ is convex, because the
reference is affine in $\Phi$, so on a closed convex set $\cC\subseteq\cS_H$
the base learner can run mirror descent (\cref{app:convex-reference-proof}):
\begin{equation}
 \Phi_{t+1}\in\argmin_{\Phi\in\cC}
 \left\{\eta\ip{g_t}{\Phi-\Phi_t}+D_\Omega(\Phi,\Phi_t)\right\},
 \qquad g_t\in\partial h_t(\Phi_t),
 \label{eq:response-mirror-update}
\end{equation}
where $D_\Omega(\Phi,\Phi')=\Omega(\Phi)-\Omega(\Phi')
-\ip{\nabla\Omega(\Phi')}{\Phi-\Phi'}$ is the Bregman divergence of a
regularizer $\Omega$.  For the ball, we use the \emph{near-$\ell_1$
regularizer} $\Omega(\Phi)=\frac12\norm{\Phi-\Phi^0}_{2,p}^2$, where
$\norm{Z}_{2,p}=(\sum_{i=1}^H\norm{Z[i]}_{\mathrm F}^p)^{1/p}$ and
$p=1+1/\log(eH)$.  This mixed norm is within a factor $e$ of the $\ell_1$ sum,
and $\Omega$ is $1/\log(eH)$-strongly convex in it, which gives the logarithm
below.  A fixed quadratic norm would lose a factor $\sqrt H$ in the tradeoff
between radius and dual radius (\cref{prop:hilbert-separation}).

\begin{theorem}[Learning the full response ball]
\label{thm:l1-response-main}
Let $H=T-1$, suppose the costs are convex and globally $L$-Lipschitz, and fix
a recovery pair $(V,\psi)$ with worst-direction constant $M_V$.  Run
counterfactual tracking with this pair and with the mirror-descent update
\eqref{eq:response-mirror-update} as the base learner, on
$\cC=\cS_H^{(1)}(R)$ with the near-$\ell_1$ regularizer, starting from the
length-$H$ prefix $\Phi_1$ of $\Phi^0$.  If $LWR>0$, use the step size
\begin{equation}
 \eta=\frac{R}{eLW\sqrt{(1+2M_V)T\log(eH)}};
 \label{eq:l1-response-step}
\end{equation}
if $LWR=0$, keep $\Phi_t=\Phi_1$.  Then
$\Reg_T(\mathfrak S^{(1)}(R))\le eLWR\sqrt{(1+2M_V)T\log(eH)}$.
\end{theorem}

The bound has no explicit dependence on $d_x$, $d_u$, or a decay rate: the
system enters only through $M_V$, and the class only through $R$.  The proof
applies \cref{thm:convex-reference-omd}, and the step size balances learning
against switching (\cref{app:near-l1-response-proof}).  The base learner
stores $T-1$ coefficient blocks and solves a convex problem
whose size grows with $T$.

\subsection{The statistical price of unrestricted disturbance ages}

Is the factor $\sqrt{\log T}$ necessary?  It is, even on a system that
recovers in one step.  The scalar \emph{one-step system} and its least
recovery cost \eqref{eq:correction-value} in the Euclidean norm are
\begin{gather}
 x_{t+1}=u_t+w_t,
 \label{eq:one-step-system}\\
 \Vstar(z)=\abs{z}.
 \label{eq:one-step-correction}
\end{gather}
Zero corrective input removes a switch discrepancy $z$ after one step, and
the first step alone costs at least $\abs{z}$
(\cref{app:one-step-correction}).

\Needspace{6\baselineskip}
\begin{theorem}[Lower bounds for the full response ball]
\label{thm:response-lower}
Fix $W>0$, and center $\mathfrak S^{(1)}(R)$ at the zero-input response of the
system \eqref{eq:one-step-system} with $x_{\mathrm{ini}}=0$.  There are
universal constants $c,c_0>0$ such that the following holds for $T$
sufficiently large.
\begin{enumerate}[label=(\roman*),leftmargin=2.2em]
\item For every $L,R>0$ and every online-control algorithm, there is an
oblivious sequence with $\abs{w_t}\le W$ and convex globally $L$-Lipschitz
costs such that
\begin{equation}
 \E\Reg_T\bigl(\mathfrak S^{(1)}(R)\bigr)\ge cLWR\sqrt{T\log(eT)}.
 \label{eq:convex-response-lower}
\end{equation}
\item Let $0<R\le c_0$, $\mu>0$, $G\ge4\mu$, and $D=4W(1+R)$.  There is an
oblivious distribution over disturbance--cost sequences, with $\abs{w_t}\le W$
almost surely, whose costs are globally $\mu$-strongly convex and
differentiable with gradients bounded by $GD$ on the Euclidean ball of radius
$D$, such that every online-control algorithm obeys
\begin{equation}
 \E\Reg_T\bigl(\mathfrak S^{(1)}(R)\bigr)\ge cGDWR\sqrt{T\log(eT)}.
 \label{eq:strong-response-lower}
\end{equation}
\end{enumerate}
Both bounds hold already against the policies
\begin{equation}
 u_t^\theta=\sum_{j=1}^H\theta_jw_{t-j},\qquad\norm{\theta}_1\le R/2,
 \label{eq:delayed-policies}
\end{equation}
with $\theta_j=0$ for $j>\lfloor T/2\rfloor$.  Their responses lie in the
ball, since the sum in \eqref{eq:infinite-l1-response-ball} is at most
$2\norm{\theta}_1\le R$.  In
part~(i), the expectation is over the algorithm's private randomization; in
part~(ii), it is also over the stated distribution.
\end{theorem}

\Cref{app:delayed-signs,app:convex-response-lower,app:strong-response-lower}
give the proofs.  The recovery pair $(\Vstar,\psi_\star)$ has
$M_\star=1$ by \eqref{eq:one-step-correction}, so
\cref{thm:l1-response-main} and part~(i) give minimax regret
$\Theta(LWR\sqrt{T\log(eT)})$ on this system for convex Lipschitz costs.
Since recovery takes one step, the logarithm is the statistical cost of learning a response over all
disturbance ages.  Part~(ii) shows that strong convexity does not remove this
price, unlike for finite classes (\cref{sec:strong-policies}).

\section{Further consequences for linear state-feedback policies}
\label{sec:familiar-benchmarks}

Linear state-feedback policies $u_t=-Kx_t$ are the benchmark class of most
prior work on online control \cite{cohen2018linear,agarwal2019online}.  Their
structure helps in two ways.  Their disturbance responses decay at a common
rate, which removes the $\sqrt{\log T}$ factor of \cref{thm:response-lower}
for convex costs (\cref{sec:stable-responses}).  Their gains form a
finite-dimensional class, which gives logarithmic regret under strong
convexity (\cref{sec:limits}).  Truncated DAC and spectral filters then
reduce the computation (\cref{sec:compact-responses}).

We keep the conventions of \cref{sec:response-learning-update}: zero initial
state, $T\ge2$, $\norm{w_t}_2\le W$, and the Euclidean norm.  For $\kappa\ge1$
and $\gamma\in(0,1)$, let $\cK_{\kappa,\gamma}$ be the class of gains $K$ with
\begin{equation}
 \norm{K}_{\mathrm{op}}\le\kappa,\qquad
 \norm{(A-BK)^r}_{\mathrm{op}}\le\kappa^2(1-\gamma)^r\quad(r\ge0).
 \label{eq:strong-stability}
\end{equation}
assumed nonempty.  This condition follows from $(\kappa,\gamma)$-strong
stability \cite{cohen2018linear}, the standard benchmark assumption of
\cite{agarwal2019online}.  The \emph{decay margin} $\gamma$ sets how fast the
effect of a disturbance fades with its age, and $\kappa$ bounds the gain and
the transient amplification.

\subsection{Regret for convex costs through age-weighted responses}
\label{sec:stable-responses}
\label{sec:full-stable-response}

Every benchmark response decays at rate $1-\gamma$, so coefficients at
larger ages can be penalized more heavily when they move without excluding
any benchmark.  We keep all $H=T-1$ ages.  The response of
$K\in\cK_{\kappa,\gamma}$ is
\begin{equation}
 \Phi_x^K[i]=(A-BK)^{i-1},\qquad
 \Phi_u^K[i]=-K(A-BK)^{i-1}.
 \label{eq:feedback-response}
\end{equation}
With $\beta^2=1-\gamma$, use the weighted norm
\begin{equation}
 \norm{\Delta\Phi}_{\beta,H,\circ}^2
 =\norm{\Delta\Phi_u[1]}_{\mathrm F}^2
 +\sum_{i=2}^H\beta^{-2(i-1)}
 \bigl(\norm{\Delta\Phi_x[i]}_{\mathrm F}^2+\norm{\Delta\Phi_u[i]}_{\mathrm F}^2\bigr),
 \label{eq:weighted-response-norm}
\end{equation}
which omits the fixed coefficient $\Phi_x[1]=I$.  The weights grow with age,
but more slowly than the benchmark responses decay, which keeps both the
radius of the class and the size of the gradients small.  Let
\[
 R_{\kappa,\gamma}^2
 =d_x\kappa^2+
 \frac{d_x\kappa^4(1+\kappa^2)}{\gamma},
 \qquad
 \cC_{\kappa,\gamma,H}
 =\left\{\Phi\in\cS_H:
 \norm{\Phi}_{\beta,H,\circ}\le R_{\kappa,\gamma}\right\}.
\]
This weighted ball contains the length-$H$ prefix of every response from
$\cK_{\kappa,\gamma}$.  Its radius is $O(\sqrt{d_x/\gamma})$ at fixed
$\kappa$, and the reference depends on the response with sensitivity
$O(W/\sqrt\gamma)$; their product gives the factor $\gamma^{-1}$ below.

\begin{theorem}[Regret for convex costs and stable linear state-feedback policies]
\label{cor:stable-feedback}
Let $H=T-1$.  Suppose the costs are convex and globally $L$-Lipschitz, and
fix a recovery pair $(V,\psi)$.
Counterfactual tracking with this pair and with the mirror-descent update
\eqref{eq:response-mirror-update} over $\cC_{\kappa,\gamma,H}$ as the base
learner, with
$D_\Omega(\Phi,\Phi')=\frac12\norm{\Phi-\Phi'}_{\beta,H,\circ}^2$, any
$\Phi_1\in\cC_{\kappa,\gamma,H}$, and step size
$\eta=2R_{\kappa,\gamma}\sqrt\gamma/(LW\sqrt{3(1+M_V)T})$
($\Phi_t=\Phi_1$ if $LW=0$) satisfies
\begin{equation}
 \Reg_T(\cK_{\kappa,\gamma})
 \le C\frac{LW\kappa^2\sqrt{d_x(1+\kappa^2)}}{\gamma}
 \sqrt{(1+M_V)T},
 \label{eq:stable-feedback-rate}
\end{equation}
where $C$ is universal.
\end{theorem}

\Cref{app:stable-response-proof} gives the radius calculation.  At fixed
dimensions, each round takes $O(T)$ arithmetic operations after $O(T)$
preprocessing (\cref{prop:response-runtime}).  The corrective law need not
decay as fast as the benchmarks: a corrective gain with decay margin
$\gamma_{\mathrm{cor}}$ gives $M_V=O(\gamma_{\mathrm{cor}}^{-1})$ at fixed
$\kappa$.  For quadratic costs, which are not globally Lipschitz, a gradient
bound on the visited region suffices.  That region has radius of order
$W\sqrt{d_x}/\gamma$ (\cref{cor:regular-response}), so the gradient bound
adds a factor $\gamma^{-1}$; this is why \cref{tab:stable-comparison} shows
$\gamma^{-2}$.

The factor $\gamma^{-1}$ is necessary.  On the scalar system
$x_{t+1}=u_t+W$, started at zero, the two policies $K=0$ and $K=-(1-\gamma)$
already force
\[
 \E\Reg_T(\cK_{1,\gamma})\ge cLW\gamma^{-1}\sqrt T
 \qquad(0<\gamma\le1/2,\ T\ge C/\gamma)
\]
under oblivious linear costs that are globally $L$-Lipschitz
(\cref{thm:convex-lower}).  Recovery on this system costs $|z|$ for every
$\gamma$, so the factor comes from the benchmark: under the constant
disturbance, the slowly decaying policy drifts far from the other one.  The
optimal dependence on dimension and recovery is open.

\subsection{Regret for strongly convex costs through a finite cover}
\label{sec:limits}

Under strong convexity, the Jensen gain pays for learning and recovery on a
finite class (\cref{sec:strong-policies}).  We therefore run exponential
weights over a finite cover of the gains.  The cover approximates the
policies' full trajectories over the horizon.  Fix
a recovery pair $(V,\psi)$ and $L,\mu>0$.  The costs are $\mu$-strongly convex
and $L$-Lipschitz on the counterfactual hull, and $L$-Lipschitz on the
segments joining the actual and reference points.  A known $\alpha\ge1$
bounds the recovery distortion:
$\norm{y_t^K-y_t^{K'}}_V\le\alpha\norm{y_t^K-y_t^{K'}}_2$ for all
$K,K'\in\cK_{\kappa,\gamma}$ and $t<T$.

\begin{corollary}[Regret for strongly convex costs and stable linear state-feedback policies]
\label{cor:strong-feedback-main}
Under these assumptions, let $\mathcal N_\varepsilon$ be a maximal
$\varepsilon$-separated subset of $\cK_{\kappa,\gamma}$ in operator norm, with
$\varepsilon$ as in \cref{cor:strong-stable-cover}.  Counterfactual tracking
with exponential weights \eqref{eq:finite-policy-ew} over the policies
$u_t=-Kx_t$ ($K\in\mathcal N_\varepsilon$) as the base learner and
$\eta=c_0\mu/(L^2\alpha)$ satisfies
\begin{equation*}
 \Reg_T(\cK_{\kappa,\gamma})
 \le C\frac{L^2}{\mu}\alpha d_xd_u\log(eT).
\end{equation*}
Here $c_0$ is the universal constant of
\cref{cor:horizon-independent-finite-library}, and $C$ depends only on
$\kappa,W,\norm{B}_{\mathrm{op}}$ and $\mu/(L\alpha)$, not on $d_x,d_u,T$ or
$\gamma$.
\end{corollary}

The rate is logarithmic because the class is finite-dimensional.  The
separated set $\mathcal N_\varepsilon$ has at most
$((1+2\kappa)/\varepsilon)^{d_xd_u}$ elements, and by maximality it is an
$\varepsilon$-cover.  The trajectories depend on the gain with sensitivity
$O(T^2)$, so a cover that is polynomially fine in $T$ suffices.  The Jensen
gain pays for learning and recovery on the cover, and Lipschitzness pays for
the approximation (\cref{cor:strong-stable-cover}).

The decay margin enters only through the costs.  On a fixed system, one
corrective law serves every $\gamma$, so $\alpha$ in
\eqref{eq:global-trajectory-distortion} does not depend on $\gamma$.  The
Lipschitz constant can: on $x_{t+1}=x_t+u_{1,t}+u_{2,t}+W$, quadratic costs
with $L\asymp\mu W/\gamma$ force
$\E\Reg_T(\cK_{1,\gamma})\ge c\mu W^2\gamma^{-2}\log T$, which matches the
prefactor $L^2/\mu$ of the corollary
(\cref{thm:strong-lower,app:strong-feedback-normalization}).  Efficient
algorithms with fast rates retain polynomial decay factors
\cite[Section~D.2]{simchowitz2020almost}, in a setting that also allows
partial observation.  Our cover need not be efficient, and whether an
efficient algorithm can avoid the decay factor is open.

\subsection{Cheaper implementations: DAC and spectral filters}
\label{sec:compact-responses}

The algorithms above can be costly to run.  The response learner stores a
coefficient for every disturbance age, and the cover can contain very many
gains.  Truncated DAC and spectral filters reduce this cost under additional
assumptions.  \Cref{tab:stable-comparison} compares their regret and time per
round with prior guarantees.  Following prior work, it assumes
$\norm{\nabla c_t(y)}_2\le GD$ on a trajectory ball of radius $D=O(W/\gamma)$,
with $\kappa$ and the dimensions fixed.  Three decay margins appear: $\gamma$
for the benchmarks, $\gamma_0$ for a known baseline gain $K_0$ that builds
feasible references, and $\gamma_{\mathrm{cor}}$ for the corrective gain.
DAC uses $K_0$ for both purposes, while spectral learning may use a different
corrective gain.  The prior rows keep their published normalizations.

\begin{table}[htbp]
\centering
\caption{Regret for convex costs and computation for linear state feedback.}
\label{tab:stable-comparison}
\vspace{0.3em}
\small
\renewcommand{\arraystretch}{1.35}
\begin{tabularx}{\textwidth}{@{}
 >{\raggedright\arraybackslash}X
 >{\centering\arraybackslash}p{0.32\textwidth}
 >{\centering\arraybackslash}p{0.28\textwidth}@{}}
\toprule
Method & Regret $/\,(GW^2)$ & Time per round \\
\midrule
\multicolumn{3}{@{}l}{\textit{Prior guarantees}}\\[-0.15em]
infinite-memory DAC \cite{kumar2023}
& $O(\gamma^{-9/2}\sqrt T)$
& $O(T)+\mathsf C_{\mathrm{FTRL}}$ \\
truncated DAC \cite{agarwal2019online}
& $\widetilde O(\gamma^{-11/2}\sqrt T)$
& $\widetilde O(\gamma^{-1})$ \\
spectral control \cite{brahmbhatt2025}
& $\widetilde O(\gamma^{-4}\sqrt T)$
& $\widetilde O(1)+\mathsf C_{\mathrm{proj}}$ \\
\midrule
\multicolumn{3}{@{}l}{\textit{This work}}\\[-0.15em]
untruncated weighted responses
& $O(\gamma^{-2}\gamma_{\mathrm{cor}}^{-1/2}\sqrt T)$
& $O(T)$ \\
weighted truncated DAC
& $O(\gamma^{-2}\gamma_0^{-1/2}\sqrt T)$
& $O(H)$ \\
spectral responses
& $\widetilde O(\gamma^{-2}\gamma_0^{-1}\gamma_{\mathrm{cor}}^{-1/2}\sqrt T)$
& $\widetilde O(1)+\mathsf C_{\mathrm{proj}}$ \\
\bottomrule
\end{tabularx}
\tablenote{The constant $\kappa$ and the dimensions are fixed.  Time is amortized per
round, with unit-cost gradient evaluations and preprocessing excluded.
$\mathsf C_{\mathrm{FTRL}}$ is the cost of cumulative optimization in
follow-the-regularized-leader (FTRL).  $\mathsf C_{\mathrm{proj}}$ is the
cost of projection onto the reference-bounded coefficient set.  Neither is
assumed constant.  The notation $\widetilde O$ suppresses logarithms in $T$,
inverse decay margins, and other problem scales.  The DAC prefix length
satisfies $H\le T-1$ and $H=\widetilde O(\gamma^{-1})$.  The spectral rows
require the more restricted class of \cref{sec:spectral-compression}.}
\end{table}

\subsubsection{Truncated disturbance--action control}
\label{sec:truncated-dac}

DAC learns the response to recent disturbances and lets a stabilizing gain
handle earlier ones.  Fix a known gain $K_0$ that stabilizes the full system.
The base learner chooses a feasible response prefix of length $H$, and at
larger ages the response continues under $A-BK_0$.  Correcting with the
same gain gives the DAC action \cite{agarwal2019online}
\begin{equation}
 u_t=-K_0x_t+
 \sum_{i=1}^{\min\{H,t-1\}}M_{\Phi_t}^{[i]}w_{t-i},
 \qquad M_{\Phi_t}^{[i]}=\Phi_{u,t}[i]+K_0\Phi_{x,t}[i].
 \label{eq:literal-truncated-dac}
\end{equation}
The sum covers only ages up to $H$, but the response itself is not
truncated: disturbances of larger age still act through $x_t$ in $-K_0x_t$.  With
$H=\min\{T-1,\widetilde O(\gamma^{-1})\}$, the regret keeps the decay
dependence of \cref{sec:stable-responses}, and each round takes $O(H)$
operations at fixed dimensions.  \Cref{cor:compact-dac-main} gives the prefix
length and the bound, and \cref{cor:regular-truncated-dac} the version for
locally bounded gradients.

\subsubsection{Spectral filters}
\label{sec:spectral-compression}

Spectral filters \cite{brahmbhatt2025} replace one coefficient per
disturbance age by one coefficient per temporal pattern.  A closed-loop
eigenvalue $\lambda$ generates the pattern $(1,\lambda,\ldots,\lambda^{m-1})$,
and the leading eigenvectors of a Hankel matrix built from these patterns
span a low-dimensional space that approximates all of them.  This requires
real nonnegative closed-loop eigenvalues and well-conditioned eigenvectors,
assumptions needed only for spectral filters.  A stable baseline gain turns the filtered input
into a feasible reference, and an independent recovery pair tracks it.  With
$h=\operatorname{polylog}(T/\gamma)$ filters, each round takes
$O(h\log^2(eT))$ amortized operations, plus a projection when gradients are
only locally bounded (\cref{cor:main-spectral-response,prop:spectral-runtime,cor:spectral-local-gradients}).
\Cref{app:truncated-dac-proof,app:spectral-construction} give the
constructions and proofs.  

\section{Conclusion and open directions}
\label{sec:discussion}

Counterfactual tracking reduces online control to online learning with
switching costs: the base learner chooses references, and a corrective law
pays for each switch at its recovery cost.  The reduction needs neither a
stabilizable system nor a common parameterization of the benchmark class.
For convex globally Lipschitz costs, the minimax converse shows that
tracking references loses at most a factor that depends only on the system.
We conclude with the main gaps.

\paragraph{Gaps in the minimax equivalence.}
The converse covers only convex costs that are globally Lipschitz, and
benchmark policies that ignore past costs.  The reduction itself is not
restricted in this way: it needs Lipschitzness only on the regions that the
trajectories visit, which allows the quadratic and strongly convex costs of
our fast rates.  Whether the equivalence extends to such costs is open; its
proof adds a distance penalty to the costs, which does not preserve
quadratic structure.  The factor $C_{\mathrm{rec}}$ may also be loose.  It
depends on the system $(A,B)$ and the norm $\norm{\cdot}$ only through one
worst-direction recovery constant for all discrepancies.  It also charges each
switch separately, ignoring cancellation between switches.  The sharp system
dependence of both the equivalence and our strongly convex state-feedback
bounds remains open.

\paragraph{Beyond the present model.}
The reduction uses three features of the model.  Known dynamics let us
evaluate counterfactual trajectories, linearity makes averaged trajectories
feasible, and unconstrained inputs let us correct errors in the reachable
subspace.  Unknown or nonlinear dynamics, partial observation, and hard
constraints each need a replacement for one of these steps.  In stochastic
networks, for example, service constraints are intrinsic, and finite-time
guarantees rest instead on Lyapunov drift arguments
\cite{liu2025queueing,liang2026queue}.

\paragraph{Efficient algorithms.}
The full response learner needs no common decay bound, but its
representation grows with the horizon.  Can it be compressed without
imposing a common decay rate?  The strongly convex state-feedback bound uses
a potentially enormous cover.  Can an efficient algorithm attain it without
an additional decay factor?  DAC and spectral filters give cheaper
implementations under additional assumptions but do not settle either
question.  Approximate evaluation and projection must also be analyzed
together with the tracking error, and the approximation bounds of
\cref{cor:reduction-corrections} are a starting point.

\section*{Statement on AI use}

The author conceived the main ideas of this work, including counterfactual
tracking and the minimax equivalence, and posed most of the research
questions.  Many of the results began as the author's conjectures, and the
author guided their proofs and provided their key ideas.  ChatGPT (OpenAI)
helped complete proof details, identify extensions, and sharpen the
theoretical analysis through many rounds of iteration with the author.  The
author wrote the paper with assistance from ChatGPT and Claude (Anthropic).  The author is
solely responsible for the results and their presentation.

\clearpage
\appendix

\section*{Guide to the appendices}
\pdfbookmark[1]{Guide to the appendices}{appendix-contents}

\begingroup
\makeatletter
\newcommand{\appendixcontentsline}[2]{%
  \@dottedtocline{1}{0em}{1.8em}%
    {\numberline{\ref*{#1}}\hyperref[#1]{#2}}{\pageref{#1}}}
\appendixcontentsline{app:reduction-proofs}{Proofs for the reduction}
\appendixcontentsline{app:correction-proofs}{Proofs for the minimax converse and recovery norms}
\appendixcontentsline{app:policy-proofs}{Proofs for general policy classes}
\appendixcontentsline{app:response-proofs}{Proofs for disturbance responses and linear state feedback}
\appendixcontentsline{app:lower-proofs}{Proofs of the lower bounds}
\appendixcontentsline{app:prior-general-frameworks}{Comparison with two general frameworks}
\makeatother
\endgroup

\bigskip

\section{Proofs for the reduction}
\label{app:reduction-proofs}

\subsection{Reachable differences and the telescoping bound}

\begin{lemma}[Feasible trajectories differ only in reachable directions]
\label{lem:relative-reachability}
Let $\bar x(q)$ and $\bar x(q')$ obey the recursion in
\eqref{eq:reference-feasible} from a common initial state, not necessarily
$x_{\mathrm{ini}}$.  Then
\[
 \bar x_t(q)-\bar x_t(q')\in\Reach
\]
for every $t\le T+1$.  The same holds for vector-valued averages of such
references under distributions over them, whenever these averages exist, and
for differences between such averages.  If $x_t$ is any trajectory of
\eqref{eq:system} with the same initial state and disturbances, then also
$x_t-\bar x_t(q)\in\Reach$.
\end{lemma}

\begin{proof}[Proof of \cref{lem:relative-reachability}]
Subtracting the two recursions and iterating gives
\[
 \bar x_t(q)-\bar x_t(q')
 =\sum_{s<t}A^{t-1-s}B
   \bigl(\bar u_s(q)-\bar u_s(q')\bigr).
\]
By Cayley--Hamilton, every power of $A$ is a linear combination of
$I,A,\ldots,A^{d_x-1}$.  Every term therefore lies in $\Reach$, and so does
their sum.  Linearity extends this to averages and their differences.  For
the final assertion, replace the inputs of one reference by those of the
actual trajectory.
\end{proof}

\begin{lemma}[Telescoping bound for the tracking error]
\label{lem:certificate-tracking}
Let $(V,\psi)$ be a recovery pair and set $\delta_{T+1}=0$.  If
\begin{equation}
 e_{t+1}=Ae_t+B\psi(e_t)+\delta_{t+1}\quad(1\le t\le T),
 \qquad e_1,\delta_2,\ldots,\delta_T\in\Reach,
 \label{eq:certificate-mismatch-dynamics}
\end{equation}
then
\begin{equation}
 \sum_{t=1}^T\norm{(e_t,\psi(e_t))}
 \le V(e_1)+\sum_{t=1}^{T-1}V(\delta_{t+1}).
 \label{eq:certificate-tracking-bound}
\end{equation}
\end{lemma}

\begin{proof}[Proof of \cref{lem:certificate-tracking}]
Since $\Reach$ is $A$-invariant and contains $\im B$, every $e_t$ lies in
$\Reach$, so every term below is defined.  By
\eqref{eq:correction-certificate}, subadditivity of $V$, and
\eqref{eq:certificate-mismatch-dynamics},
\[
 \norm{(e_t,\psi(e_t))}+V(e_{t+1})
 \le V(e_t)+V(\delta_{t+1}).
\]
Summing over $t$, the $V(e_t)$ terms telescope.  Discarding
$V(e_{T+1})\ge0$ and using $\delta_{T+1}=0$ gives
\eqref{eq:certificate-tracking-bound}.
\end{proof}

Two facts relate the trajectory norm \eqref{eq:induced-trajectory-norm} to
recovery.  First, \eqref{eq:correction-certificate} states exactly that
\begin{equation*}
 \norm{(e,\psi(e))}_V\le V(e).
\end{equation*}
Second, for two feasible references, the disturbance cancels in the
next-state difference, so
$A(\bar x_t(q)-\bar x_t(q'))+B(\bar u_t(q)-\bar u_t(q'))
=\bar x_{t+1}(q)-\bar x_{t+1}(q')$ and
\begin{equation*}
 \norm{\bar y_t(q)-\bar y_t(q')}_V
 =\norm{\bar y_t(q)-\bar y_t(q')}+d_t^V(q,q').
\end{equation*}

\subsection{Mirror descent as a base learner}
\label{app:convex-reference-proof}

When $\cQ$ is convex and the losses are convex in $q$, the base learner can
run mirror descent.  Its step size controls both its regret and its
movement, and hence its switching cost.

Let $\cQ$ be a nonempty closed convex set in an affine space
$q_0+\mathsf H$, where $\mathsf H=\Span(\cQ-\cQ)$ carries a norm
$\norm{\cdot}_{\cQ}$.  For a linear functional $g$ on $\mathsf H$, write
\begin{equation*}
 \norm{g}_{\cQ,*}
 =\sup_{\substack{z\in\mathsf H\\\norm{z}_{\cQ}\le1}}\ip{g}{z}.
\end{equation*}
Let $\Omega$ be differentiable at the iterates, with Bregman divergence
\begin{equation*}
 D_\Omega(q,q')
 =\Omega(q)-\Omega(q')-\ip{\nabla\Omega(q')}{q-q'},
\end{equation*}
and assume
\begin{equation}
 D_\Omega(q,q')
 \ge \frac{\sigma}{2}\norm{q-q'}_{\cQ}^2.
 \label{eq:reference-regularizer-convexity}
\end{equation}
For convex losses $h_t$, subgradients $g_t\in\partial h_t(q_t)$, and a step
size $\eta>0$, the update is
\begin{equation}
 q_{t+1}\in\argmin_{q\in\cQ}
 \left\{\eta\ip{g_t}{q-q_t}+D_\Omega(q,q_t)\right\}.
 \label{eq:reference-mirror-descent}
\end{equation}
Assume the minimum is attained, as it is on the compact finite-dimensional
sets used below.  Define the \emph{loss sensitivity} $G_t$ and the
\emph{recovery sensitivity} $\Gamma_t^V$, the recovery cost per unit of
movement, by
\begin{align}
 G_t&=\norm{g_t}_{\cQ,*},
 \notag\\
 \Gamma_t^V&=
 \sup_{\substack{q,q'\in\cQ\\q\ne q'}}
 \frac{d_t^V(q,q')}{\norm{q-q'}_{\cQ}}\quad(t<T),
 \qquad \Gamma_T^V=0.
 \label{eq:reference-recovery-sensitivity}
\end{align}
If $\cQ$ is a singleton, both are zero; the supremum over an empty set is
zero.  In this subsection, $q_{T+1}$ is the mirror-descent output after
round $T$.  The algorithm never plays it, and $\Gamma_T^V=0$ charges it no
recovery cost, so the convention $q_{T+1}=q_T$ of \cref{sec:reduction} is
not needed.

\begin{theorem}[Mirror-descent base learner]
\label{thm:convex-reference-omd}
Run \cref{alg:counterfactual-tracking} with mirror descent
\eqref{eq:reference-mirror-descent} as the base learner.  Under the
conditions above and those of \cref{thm:master-reduction}, every $q\in\cQ$
satisfies
\begin{equation}
 \sum_{t=1}^T[c_t(y_t)-h_t(q)]
 \le \frac{D_\Omega(q,q_1)}{\eta}
 +\frac{\eta}{2\sigma}\sum_{t=1}^T G_t^2
 +\frac{L\eta}{\sigma}\sum_{t=1}^{T-1}G_t\Gamma_t^V
 -\sum_{t=1}^T J_t(q_t).
 \label{eq:convex-reference-omd}
\end{equation}
\end{theorem}

The first two terms are the usual mirror-descent bound.  The third is the
switching cost, and the last is the Jensen gain.

\begin{proof}[Proof of \cref{thm:convex-reference-omd}]
The proof bounds the regret and the movement of mirror descent separately
and inserts both into \cref{thm:master-reduction}.

\emph{Regret.}  Convexity of $h_t$ gives
\[
 h_t(q_t)-h_t(q)\le\ip{g_t}{q_t-q}.
\]
The optimality condition for \eqref{eq:reference-mirror-descent} states that
\[
 \ip{\eta g_t+\nabla\Omega(q_{t+1})-\nabla\Omega(q_t)}{q-q_{t+1}}\ge0
 \qquad(q\in\cQ).
\]
Write $q_t-q=(q_t-q_{t+1})+(q_{t+1}-q)$.  Apply the optimality condition to
the second difference; by the three-point identity for Bregman divergences
\cite{shalevshwartz2012}, it contributes
$[D_\Omega(q,q_t)-D_\Omega(q,q_{t+1})-D_\Omega(q_{t+1},q_t)]/\eta$.  Bound the
first difference by $G_t\norm{q_t-q_{t+1}}_{\cQ}$, use
\eqref{eq:reference-regularizer-convexity} for $D_\Omega(q_{t+1},q_t)$, and
use
\[
 \eta G_t\norm{q_t-q_{t+1}}_{\cQ}
 -\frac{\sigma}{2}\norm{q_t-q_{t+1}}_{\cQ}^2
 \le\frac{\eta^2G_t^2}{2\sigma}.
\]
The remaining Bregman divergences telescope, and
\begin{equation}
 \sum_{t=1}^T\bigl[h_t(q_t)-h_t(q)\bigr]
 \le
 \frac{D_\Omega(q,q_1)}{\eta}
 +\frac{\eta}{2\sigma}\sum_{t=1}^T G_t^2.
 \label{eq:convex-reference-learning-proof}
\end{equation}

\emph{Movement.}  The optimality condition tested at $q=q_t$ gives
\[
 \ip{\nabla\Omega(q_{t+1})-\nabla\Omega(q_t)}
     {q_{t+1}-q_t}
 \le \eta\ip{g_t}{q_t-q_{t+1}}.
\]
The left side equals
$D_\Omega(q_{t+1},q_t)+D_\Omega(q_t,q_{t+1})$, which is at least
$\sigma\norm{q_{t+1}-q_t}_{\cQ}^2$.  The right side is at most
$\eta G_t\norm{q_{t+1}-q_t}_{\cQ}$.  Hence
\begin{equation}
 \norm{q_{t+1}-q_t}_{\cQ}
 \le\frac{\eta}{\sigma}G_t.
 \label{eq:convex-reference-movement-proof}
\end{equation}
For every played update $t<T$, the definition of $\Gamma_t^V$ therefore
gives
\[
 Ld_t^V(q_t,q_{t+1})
 \le\frac{L\eta}{\sigma}\Gamma_t^VG_t.
\]
Inserting this bound and \eqref{eq:convex-reference-learning-proof} into
\cref{thm:master-reduction} proves \eqref{eq:convex-reference-omd}.
\end{proof}

\subsection{Representation error and initial tracking error}

The reduction still applies when a decision only approximates the target
policy, or when the references start from a state other than
$x_{\mathrm{ini}}$.  For a decision $q$ representing a target policy $\pi$,
define the \emph{representation error}
\begin{equation*}
 \operatorname{Disc}_T(q,\pi)
 =\sum_{t=1}^T\bigl[h_t(q)-c_t(y_t^\pi)\bigr].
\end{equation*}

\begin{corollary}[Representation error and initial tracking error]
\label{cor:reduction-corrections}
Keep the assumptions of \cref{thm:master-reduction} on the costs and the
recovery pair.  Relax only the initial condition in
\eqref{eq:reference-feasible}: the references are predictable and obey its
recursion, but start from a common state $\bar x_1$, possibly different from
$x_{\mathrm{ini}}$, with $e_1=x_{\mathrm{ini}}-\bar x_1\in\Reach$.  Then
every target policy $\pi$ and every $q\in\cQ$ satisfy
\begin{align}
 \sum_{t=1}^T\bigl[c_t(y_t)-c_t(y_t^\pi)\bigr]
 \le{}&
 \operatorname{Disc}_T(q,\pi)
 +\sum_{t=1}^T\bigl[h_t(q_t)-h_t(q)\bigr]
 +L V(e_1)
 \notag\\
 &+L\sum_{t=1}^{T-1}d_t^V(q_t,q_{t+1})
 -\sum_{t=1}^T J_t(q_t).
 \label{eq:master-certificate-corrected}
\end{align}
Consequently, if the base learner guarantees
\[
 \sum_{t=1}^T[h_t(q_t)-h_t(q)]
 +L\sum_{t=1}^{T-1}d_t^V(q_t,q_{t+1})
 \le\mathcal B_T(q),
\]
then the control regret to $\pi$ is at most
\begin{equation}
 \operatorname{Disc}_T(q,\pi)+\mathcal B_T(q)+LV(e_1)
 -\sum_{t=1}^T J_t(q_t).
 \label{eq:black-box-transfer}
\end{equation}
\end{corollary}

\begin{proof}[Proof of \cref{cor:reduction-corrections}]
Repeat the proof of \cref{thm:master-reduction} without setting $e_1=0$.
The references share their initial state and disturbances, so the switch
discrepancies lie in $\Reach$ by \cref{lem:relative-reachability}, and
$e_1\in\Reach$ by assumption.  \Cref{lem:certificate-tracking} then adds
$V(e_1)$ to the tracking bound.  To compare with $\pi$ rather than $q$, write
$c_t(y_t)-c_t(y_t^\pi)=[c_t(y_t)-h_t(q)]+[h_t(q)-c_t(y_t^\pi)]$; the second
brackets sum to $\operatorname{Disc}_T(q,\pi)$.  This proves
\eqref{eq:master-certificate-corrected}, and the assumed guarantee of the
base learner gives \eqref{eq:black-box-transfer}.
\end{proof}

\section{Proofs for the minimax converse and recovery norms}
\label{app:correction-proofs}

\Cref{tab:recovery-quantities} summarizes the recovery quantities used in
the reduction and the converse.

\begin{table}[htbp]
\centering
\caption{Recovery quantities and their roles.}
\label{tab:recovery-quantities}
\vspace{0.3em}
\small
\renewcommand{\arraystretch}{1.12}
\begin{tabularx}{\linewidth}{@{}lXl@{}}
\toprule
Quantity & Role & Definition \\
\midrule
$V_h(z)$ & Minimum cumulative deviation from $z$ over the next $h$ rounds.
 & \eqref{eq:remaining-horizon-value} \\
$\Vstar(z)$ & Minimum cumulative deviation from $z$ over the infinite future.
 & \eqref{eq:correction-value} \\
$V(z)$ & Chosen recovery norm, paired with a corrective law.
 & \eqref{eq:correction-certificate} \\
$V_K(z)$ & Total future deviation under the linear corrective law $-Ke$.
 & \eqref{eq:linear-correction-value} \\
$M_V$ & Worst-direction bound: $V(z)\le M_V\norm{z}_2$.
 & \eqref{eq:euclidean-correction-factor} \\
$P_K$ & Bounds peak deviation under repeated unit-size switch discrepancies.
 & \eqref{eq:peak-tracking-gain} \\
$C_{\mathrm{rec}}$ & Factor comparing minimax switching and control regret.
 & \eqref{eq:recovery-distortion} \\
\bottomrule
\end{tabularx}
\end{table}

\subsection{Controllability on the reachable subspace}

Why are the recovery values finite and comparable to the Euclidean norm?
Because the system is controllable on $\Reach$.  Recall the
remaining-horizon values $V_h$ \eqref{eq:remaining-horizon-value} and the
least infinite-horizon recovery cost $\Vstar$ \eqref{eq:correction-value}.
Let $r=\dim\Reach$.  Since $\norm{\cdot}$ is a norm on
$\R^{d_x}\times\R^{d_u}$, there are constants $0<m\le M<\infty$ such
that
\begin{equation}
 m(\norm{e}_2+\norm{v}_2)
 \le\norm{(e,v)}
 \le M(\norm{e}_2+\norm{v}_2).
 \label{eq:norm-equivalence}
\end{equation}
Since $(A,B)$ restricted to $\Reach$ is controllable, there are linear maps
$L_0,\ldots,L_{r-1}$ such that the inputs $v_s=L_sz$ drive every
$z\in\Reach$ to zero in at most $r$ steps.  The intermediate states are also
linear in $z$.  By \eqref{eq:norm-equivalence}, this finite-step trajectory
therefore has total deviation at most $C_{\mathrm{db}}\norm{z}_2$, for a
constant $C_{\mathrm{db}}<\infty$ depending only on $(A,B,\norm{\cdot})$.
Since the first stage alone costs at least $m\norm{z}_2$,
\begin{equation}
 m\norm{z}_2
 \le V_h(z)
 \le\Vstar(z)
 \le C_{\mathrm{db}}\norm{z}_2,
 \qquad h\ge1.
 \label{eq:value-equivalence}
\end{equation}
In particular, all these values are finite, and $\Vstar$ is bounded above and
below by multiples of the Euclidean norm on $\Reach$.

\subsection{Remaining-horizon recovery}
\label{app:finite-correction}

Recall that $V_0\equiv0$ and, for $h\ge1$,
\begin{equation*}
 V_h(z)
 =\inf_{\substack{e_0=z\\ e_{s+1}=Ae_s+Bv_s}}
   \sum_{s=0}^{h-1}\norm{(e_s,v_s)},
\end{equation*}
with no terminal constraint.

\begin{theorem}[Remaining-horizon recovery]
\label{thm:remaining-horizon-recovery}
For every $h\ge1$, $V_h$ is a finite continuous norm on $\Reach$ and obeys
\begin{equation*}
 V_h(e)
 =\min_{v\in\R^{d_u}}
   \{\norm{(e,v)}+V_{h-1}(Ae+Bv)\}.
\end{equation*}
A Borel measurable minimizing selector $\psi_h$ exists.  If
$e_1,\delta_2,\ldots,\delta_T\in\Reach$, $\delta_{T+1}=0$,
$v_t=\psi_{T-t+1}(e_t)$, and
$e_{t+1}=Ae_t+Bv_t+\delta_{t+1}$, then
\begin{equation}
 \sum_{t=1}^T\norm{(e_t,v_t)}
 \le V_T(e_1)+
 \sum_{t=1}^{T-1}V_{T-t}(\delta_{t+1}).
 \label{eq:remaining-horizon-tracking}
\end{equation}
\end{theorem}

\begin{theorem}[Minimality of remaining-horizon prices]
\label{thm:finite-horizon-minimality}
Let maps $N_{s,T}:\Reach\to[0,\infty)$ satisfy $N_{s,T}(0)=0$.  Consider a
possibly randomized causal correction scheme: its input $v_s$ depends only
on $(\delta_2,\ldots,\delta_s)$ and its internal randomness, so it is chosen
before $\delta_{s+1}$ is revealed.  Suppose that, under
$e_{s+1}=Ae_s+Bv_s+\delta_{s+1}$ with $e_1=0$ and $\delta_{T+1}=0$,
\[
 \E\!\left[\sum_{s=1}^T\norm{(e_s,v_s)}\right]
 \le\sum_{s=1}^{T-1}N_{s,T}(\delta_{s+1})
\]
for every choice of $\delta_2,\ldots,\delta_T\in\Reach$.  Then
$N_{t,T}(z)\ge V_{T-t}(z)$ for every $1\le t<T$ and $z\in\Reach$.
\end{theorem}

\begin{proof}[Proof of \cref{thm:remaining-horizon-recovery} and
\cref{thm:finite-horizon-minimality}]
\emph{Norm.}  Scaling and superposing feasible finite trajectories show that
$V_h$ is absolutely homogeneous and subadditive.  By
\eqref{eq:value-equivalence},
$m\norm{z}_2\le V_h(z)\le C_{\mathrm{db}}\norm{z}_2$.  Thus every $V_h$,
$h\ge1$, is a finite norm, hence continuous on the finite-dimensional space
$\Reach$.

\emph{Bellman equation.}  Separating the first input gives
\begin{equation*}
 V_h(e)
 =\min_{v\in\R^{d_u}}
   \{\norm{(e,v)}+V_{h-1}(Ae+Bv)\}.
\end{equation*}
The objective is continuous and, by \eqref{eq:norm-equivalence}, coercive in
$v$, so the minimum is attained.  Its argmin correspondence is nonempty,
compact valued, and has a closed graph.  Inf-compactness gives upper
hemicontinuity, hence Borel measurability, and the Kuratowski--Ryll-Nardzewski
selection theorem \cite{kuratowski1965} gives a Borel minimizing selector
$\psi_h$.

\emph{Tracking bound.}  With $v_t=\psi_{T-t+1}(e_t)$, the Bellman equation and
the triangle inequality give
\[
 \norm{(e_t,v_t)}+V_{T-t}(e_{t+1})
 \le V_{T-t+1}(e_t)+V_{T-t}(\delta_{t+1}).
\]
Summing over $t$, the values $V_{T-t+1}(e_t)$ telescope, and $V_0\equiv0$
gives
\begin{equation*}
 \sum_{t=1}^{T}\norm{(e_t,v_t)}
 \le V_T(e_1)+
 \sum_{t=1}^{T-1}V_{T-t}(\delta_{t+1}).
\end{equation*}

\emph{Minimality.}  Take $N_{s,T}$ and a correction scheme as in
\cref{thm:finite-horizon-minimality}.  On the all-zero sequence, the bound
is zero, so the deviation is zero almost surely.  Now introduce a single
discrepancy $z$ after round $t$.  By causality, the scheme plays as on the
zero sequence up to round $t$, so it pays nothing before the switch.  Every
sample-path continuation is then a feasible length-$(T-t)$ correction from
$z$, so its cost is at least $V_{T-t}(z)$.  Taking expectations proves
\begin{equation*}
 N_{t,T}(z)\ge V_{T-t}(z)
 \qquad(1\le t<T).
\end{equation*}
\end{proof}

\begin{corollary}[Counterfactual tracking with remaining-horizon recovery]
\label{cor:finite-horizon-reduction}
Keep the assumptions of \cref{thm:master-reduction}, but on round \(t\) use
the minimizing law \(\psi_{T-t+1}\) as the corrective law.  Then, for every
\(q\in\cQ\), with \(d_{t,T}^\star\) as in \eqref{eq:optimal-finite-pullback},
\begin{align*}
 \sum_{t=1}^T\bigl[c_t(y_t)-h_t(q)\bigr]
 \le{}&
 \sum_{t=1}^T\bigl[h_t(q_t)-h_t(q)\bigr]
 +L\sum_{t=1}^{T-1}d_{t,T}^\star(q_t,q_{t+1})
 -\sum_{t=1}^T J_t(q_t).
\end{align*}
\end{corollary}

\begin{proof}
In the proof of \cref{thm:master-reduction}, replace
\eqref{eq:tracking-deviation-sum} by \eqref{eq:remaining-horizon-tracking}
with $e_1=0$.
\end{proof}

\subsection{Minimax equivalence}
\label{app:minimax-proof}

The proof couples each play of one game with a play of the other on the
same disturbances, and compares the two regrets pathwise.  We
first fix the two games and the expectation convention.

Recall $\Alg^{\mathrm{ctrl}}$ and $\Alg^{\mathrm{sw}}$ from
\cref{sec:minimax-completeness}.  The two games have different public
transcripts: the history $\mathcal H_t$ of \cref{sec:model} in the control
game, and the switching history $\mathcal H_t^{\mathrm{sw}}$ of
\cref{sec:minimax-completeness} in the switching game.  An environment in
$\Env_T^{\mathrm{ctrl}}$ or $\Env_T^{\mathrm{sw}}$ is a stochastic kernel
from the corresponding public transcript to bounded disturbances and convex,
globally $L$-Lipschitz costs.  It may depend on the learner's realized past
outputs but not on its fresh private randomization in the current round.

The reference points of round $t$ form the convex set
\begin{equation*}
 \cY_t=\{\bar y_t(q):q\in\cQ\},
\end{equation*}
which contains every benchmark point $y_t^\pi=\bar y_t(\pi)$.  By
\eqref{eq:reference-feasible}, each switch discrepancy is
\begin{equation}
 \delta_{t+1}
 =\bar x_{t+1}(q_t)-\bar x_{t+1}(q_{t+1})
 =A\bar x_t(q_t)+B\bar u_t(q_t)+w_t-\bar x_{t+1}(q_{t+1}),
 \label{eq:minimax-switch-discrepancy}
\end{equation}
and it lies in $\Reach$.  Fix any $\pi_0\in\Pi$.  By
\cref{lem:relative-reachability}, $A\bar x_t(q)+B\bar u_t(q)+w_t$ and
$\bar x_{t+1}(q')$ lie in $x_{t+1}^{\pi_0}+\Reach$ for all $q,q'\in\cQ$, so
the discrepancy \eqref{eq:minimax-switch-discrepancy} lies in $\Reach$.

\begin{remark}[Scope of the converse]
The argument uses loss independence only in
\eqref{eq:penalty-regret-identity}.  The reduction in
\cref{thm:master-reduction} continues to hold for policies that react to past
costs.  For finite $\Pi$, the measurability and selection conditions of
\cref{thm:minimax-completeness} are automatic.  For compact continuously
parameterized classes, the reference points form a compact convex set,
Carath\'eodory's theorem makes a nearest point a finite mixture, and a
measurable selection exists.  For a general measurable class, the theorem
keeps the selection assumption explicit.
Approximate projections may be used when their errors are summable and added
to the regret bound.
\end{remark}

\paragraph{Expectation convention.}
Unrestricted randomized inputs can make the learner's cost and the
comparator's cost both have infinite expectation.  We therefore form the whole
pathwise regret first, never separate expected learner and comparator
costs, and take its upper expectation
\begin{equation}
 \E^\uparrow Z:=\lim_{m\to\infty}\E[\max\{Z,-m\}].
 \label{eq:minimax-upper-expectation}
\end{equation}
We write $\E$ when there is no risk of confusion.  For a measurable $Z$, each
expectation in \eqref{eq:minimax-upper-expectation} is well defined, since
$\max\{Z,-m\}\ge-m$, and the expectations decrease in $m$ to an
extended-real limit.  With $Z^+=\max\{Z,0\}$ and $Z^-=\max\{-Z,0\}$, the
limit is $\E Z^+-\E Z^-$ when $\E Z^+<\infty$, and $+\infty$ otherwise.  It
therefore agrees with ordinary expectation whenever the latter is defined, and
it assigns $+\infty$ when both tails have infinite expectation.  Lower
truncation gives monotonicity.  The identity
$\max\{aZ,-m\}=a\max\{Z,-m/a\}$ gives positive homogeneity for $a>0$.

The comparator infimum is never $+\infty$: the class is nonempty, and every
policy has finite cost on each finite realized path.  If the infimum is
$-\infty$, the regret is $+\infty$, and the pathwise identities below remain
valid with both sides equal to $+\infty$.  Zero costs give zero control
regret and nonnegative switching regret.  Hence the worst-case value of every
strategy, and both minimax values, lie in $[0,+\infty]$.

\begin{proof}[Proof of \cref{thm:minimax-completeness}]
\emph{Finiteness of $C_{\mathrm{rec}}$.}
Let $e\in\Reach$ and $v\in\R^{d_u}$.  By \eqref{eq:norm-equivalence} and
\eqref{eq:value-equivalence},
$\Vstar(e)\le C_{\mathrm{db}}\norm{e}_2\le(C_{\mathrm{db}}/m)\norm{(e,v)}$.
Since $Ae+Bv\in\Reach$, likewise
$\Vstar(Ae+Bv)\le(C_{\mathrm{db}}\norm{[A\ B]}_{\mathrm{op}}/m)\norm{(e,v)}$.

\emph{From switching to control.}
Take any switching learner, and track its decisions with the
remaining-horizon minimizing laws of \cref{thm:remaining-horizon-recovery}:
round $t$ plays $u_t=\bar u_t(q_t)+\psi_{T-t+1}(e_t)$, with
$e_t=x_t-\bar x_t(q_t)$.  Since $\bar x_1(q)=x_{\mathrm{ini}}$, the tracking
error obeys
\[
 e_{t+1}=Ae_t+Bv_t+\delta_{t+1},
 \qquad e_1=0.
\]
Lipschitzness and \eqref{eq:remaining-horizon-tracking} give, pathwise,
\[
 \sum_{t=1}^T[c_t(y_t)-h_t(q_t)]
 \le L\sum_{t=1}^{T-1}d_{t,T}^\star(q_t,q_{t+1}).
\]
Adding $\sum_th_t(q_t)-\inf_{\pi\in\Pi}\sum_th_t(\pi)$ to both sides gives
$\Reg_T(\Pi)\le\operatorname{SReg}_T(q_{1:T})$ pathwise.

To pass to adaptive environments, fix
$\mathcal E^{\mathrm{ctrl}}\in\Env_T^{\mathrm{ctrl}}$.  From
$\mathcal H_t^{\mathrm{sw}}$, a switching environment can reconstruct the
history of this tracking algorithm.  On each round, it applies the
kernel of $\mathcal E^{\mathrm{ctrl}}$ to this reconstructed control
transcript.  The resulting member of $\Env_T^{\mathrm{sw}}$ couples the two
plays with the same costs and disturbances.  Monotonicity of $\E^\uparrow$,
the supremum over environments in each game, and then the infimum over
switching learners prove
$\MinReg_T^{\mathrm{ctrl}}\le\MinReg_T^{\mathrm{sw}}$.

\emph{From control to switching.}
Fix an arbitrary causal control algorithm $\mathcal A$.  We construct a
switching learner that runs $\mathcal A$ internally on modified costs.  Let
$y_t^{\mathcal A}=(x_t^{\mathcal A},u_t^{\mathcal A})$ be the simulated
state--input point.  Before the round-$t$ cost and disturbance are revealed,
the learner selects a decision whose point is nearest,
\[
 q_t\in\argmin_{q\in\cQ}\norm{y_t^{\mathcal A}-\bar y_t(q)},
\]
and we set $\bar y_t=\bar y_t(q_t)$ and
\[
 r_t=\norm{y_t^{\mathcal A}-\bar y_t},
 \qquad
 \phi_t(y)=\inf_{q\in\cQ}\norm{y-\bar y_t(q)}.
\]
As the distance to the convex set $\cY_t$, $\phi_t$ is convex and
$1$-Lipschitz in $\norm{\cdot}$.  After observing $c_t$, the learner feeds $\mathcal A$ the
modified cost
\begin{equation}
 \widetilde c_t(y)
 =\frac{
   c_t(y)+L(C_{\mathrm{rec}}-1)\phi_t(y)
 }{C_{\mathrm{rec}}}.
 \label{eq:distance-penalized-cost}
\end{equation}
This cost is convex and $L$-Lipschitz.  It is built from predictable
quantities and $c_t$, so the adaptive protocol is preserved.

Every benchmark point is a reference point, $y_t^\pi=\bar y_t(\pi)$, so
$\phi_t(y_t^\pi)=0$.  By loss
independence, replacing $c_{1:t-1}$ by $\widetilde c_{1:t-1}$ changes no
counterfactual comparator trajectory.  Since $\phi_t(y_t^{\mathcal A})=r_t$,
\begin{equation}
\begin{aligned}
 C_{\mathrm{rec}}\!\left[
   \sum_{t=1}^T\widetilde c_t(y_t^{\mathcal A})
   -\inf_{\pi\in\Pi}\sum_{t=1}^T\widetilde c_t(y_t^\pi)
 \right]
 ={}&
 \sum_{t=1}^T c_t(y_t^{\mathcal A})
 -\inf_{\pi\in\Pi}\sum_{t=1}^T c_t(y_t^\pi)\\
 &+L(C_{\mathrm{rec}}-1)\sum_{t=1}^T r_t.
\end{aligned}
 \label{eq:penalty-regret-identity}
\end{equation}

It remains to bound the switching cost by the penalty.  Write
\[
 e_t=x_t^{\mathcal A}-\bar x_t(q_t),
 \qquad
 v_t=u_t^{\mathcal A}-\bar u_t(q_t),
\]
so that $r_t=\norm{(e_t,v_t)}$.  The simulated state and $\bar x_t(q_t)$ lie
in the same translate of $\Reach$, by the relative-reachability argument
above, so $e_t\in\Reach$ and $e_1=0$.
Subtracting the dynamics gives
\begin{equation*}
 \delta_{t+1}=e_{t+1}-(Ae_t+Bv_t).
\end{equation*}
Since $V_{T-t}$ is a norm,
\[
 V_{T-t}(\delta_{t+1})
 \le V_{T-t}(e_{t+1})+V_{T-t}(Ae_t+Bv_t).
\]
Combine this with $h_t(q_t)-c_t(y_t^{\mathcal A})\le Lr_t$, reindex the
first recovery term, and use $e_1=0$ and $V_0=0$:
\begin{align*}
 \operatorname{SReg}_T(q_{1:T})
 \le{}&
 \sum_{t=1}^T c_t(y_t^{\mathcal A})
 -\inf_{\pi\in\Pi}\sum_{t=1}^T c_t(y_t^\pi)\\
 &+L\sum_{t=1}^T
 \left[
   \norm{(e_t,v_t)}+V_{T-t+1}(e_t)+V_{T-t}(Ae_t+Bv_t)
 \right].
\end{align*}
Because $V_h\le\Vstar$, the definition of $C_{\mathrm{rec}}$ bounds the
second line by $L(C_{\mathrm{rec}}-1)\sum_t r_t$.  By
\eqref{eq:penalty-regret-identity}, pathwise,
\[
 \operatorname{SReg}_T(q_{1:T})
 \le C_{\mathrm{rec}}\,
 \Reg_T^{\mathcal A}(\widetilde c_{1:T};\Pi),
\]
where the right side is the regret of $\mathcal A$ to $\Pi$ on the modified
costs.

For adaptive environments, fix
$\mathcal E^{\mathrm{sw}}\in\Env_T^{\mathrm{sw}}$.  From the public control
transcript under the modified costs, a control environment can recover each
past original cost,
\[
 c_s=C_{\mathrm{rec}}\widetilde c_s
     -L(C_{\mathrm{rec}}-1)\phi_s,
\]
form $\phi_s$, and apply the fixed measurable nearest-decision selection
to $y_s^{\mathcal A}$.  It thus reconstructs the entire past switching
transcript.  On round $t$, it queries the kernel of $\mathcal E^{\mathrm{sw}}$
on this transcript, obtains $(c_t,w_t)$, and passes $(\widetilde c_t,w_t)$ to
$\mathcal A$.  This defines a member of $\Env_T^{\mathrm{ctrl}}$: $\phi_t$ depends
only on the past, and \eqref{eq:distance-penalized-cost} is
convex and globally $L$-Lipschitz.  Every adaptive switching play is thus
coupled to an admissible control play on which the pathwise inequality above
holds.  Monotonicity and positive homogeneity of $\E^\uparrow$, the supremum
over environments in each game, and then the infimum over $\mathcal A$ prove
$\MinReg_T^{\mathrm{sw}}\le C_{\mathrm{rec}}\MinReg_T^{\mathrm{ctrl}}$.
Since $C_{\mathrm{rec}}$ does not depend on $T$, the equivalence
\eqref{eq:no-regret-completeness} follows.
\end{proof}

\paragraph{Norm rescaling.}
The norm $\norm{\cdot}$ specifies cost regularity, so it must be read
together with $L$.  Replacing $\norm{\cdot}$ by $a\norm{\cdot}$ and $L$ by
$L/a$, for $a>0$, leaves the cost class, the switching costs $LV_h$, and
$C_{\mathrm{rec}}$ unchanged.  Changing the relative weights of state and
input can change the recovery costs and the comparison factor.  At fixed $L$,
it also changes the admissible cost class.  The theorem holds for every fixed
norm; which geometry gives the best quantitative bound depends on the
application.

\begin{corollary}[When dynamics create no switching cost]
\label{cor:oco-consistency}
If, on every round, all decisions $q\in\cQ$ share the next state
$\bar x_{t+1}(q)$, then every switch discrepancy in
\eqref{eq:minimax-switch-discrepancy} is zero.  The switching problem is then
ordinary online learning over counterfactual trajectories, with no switching
cost.  More generally, a state component that no input can change is the
same along every reference, so it creates no switching cost even when it is
unstable.  Its effect on costs belongs to the ordinary learning term.
\end{corollary}

\subsection{Infinite-horizon recovery}
\label{app:optimal-correction-norm}

\begin{theorem}[Infinite-horizon recovery norm]
\label{thm:optimal-correction-certificate}
The remaining-horizon values satisfy $V_h\uparrow\Vstar$ pointwise on
$\Reach$.  The limit $\Vstar$ is a finite continuous norm and obeys the
Bellman equation \eqref{eq:stationary-bellman}.  The minimum is attained by a Borel measurable corrective law
$\psi_\star:\Reach\to\R^{d_u}$ with $\norm{\psi_\star(e)}_2\le C\norm{e}_2$
on $\Reach$ for some constant $C<\infty$.  In particular,
$(\Vstar,\psi_\star)$ is a recovery pair.
\end{theorem}

\begin{proof}[Proof of \cref{thm:optimal-correction-certificate}]
\emph{Convergence.}  Dropping a nonnegative terminal stage gives
$V_h\le V_{h+1}\le\Vstar$.  Fix $z$ and choose an $h$-stage minimizer for
each $h$.  By \eqref{eq:value-equivalence} and \eqref{eq:norm-equivalence},
every stage of every minimizer is bounded by a multiple of $\norm{z}_2$.  A
diagonal subsequence therefore converges on every finite prefix to an
infinite feasible trajectory.  Fatou's lemma gives
\[
 \Vstar(z)\le\lim_{h\to\infty}V_h(z),
\]
and $V_h\le\Vstar$ gives the reverse inequality.  Hence
$V_h\uparrow\Vstar$.  Since the $V_h$ are norms and
\eqref{eq:value-equivalence} holds uniformly in $h$, their limit is a finite
norm equivalent to the Euclidean norm on $\Reach$, hence continuous.

\emph{Bellman equation.}  For any first input $v$, concatenating it with an
$\varepsilon$-optimal continuation and letting $\varepsilon\downarrow0$ gives
\[
 \Vstar(e)
 \le\norm{(e,v)}+\Vstar(Ae+Bv).
\]
Conversely, split any feasible sequence after its first stage and take the
infimum.  This proves \eqref{eq:stationary-bellman}.  The minimized objective
is continuous and, by \eqref{eq:norm-equivalence}, coercive in $v$.  Its
argmin correspondence is nonempty, compact valued, and has a closed graph.
Inf-compactness gives upper hemicontinuity, hence Borel measurability, and
the Kuratowski--Ryll-Nardzewski selection theorem \cite{kuratowski1965}
gives a Borel selector $\psi_\star$.

\emph{Growth and recovery pair.}  By \eqref{eq:norm-equivalence},
\eqref{eq:stationary-bellman}, and \eqref{eq:value-equivalence},
$m\norm{\psi_\star(e)}_2\le\norm{(e,\psi_\star(e))}\le\Vstar(e)\le
C_{\mathrm{db}}\norm{e}_2$, so $\norm{\psi_\star(e)}_2\le C\norm{e}_2$ with
$C=C_{\mathrm{db}}/m$.  At $v=\psi_\star(e)$, \eqref{eq:stationary-bellman}
is \eqref{eq:correction-certificate} with equality, so
$(\Vstar,\psi_\star)$ is a recovery pair.
\end{proof}

\Needspace{14\baselineskip}

\subsection{Minimality of the infinite-horizon price}
\label{app:minimal-switch-bound}

\begin{theorem}[Minimality of the infinite-horizon price]
\label{thm:minimality}
Let $N:\Reach\to[0,\infty)$, with $N(0)=0$, assign a time-independent
price to each switch discrepancy.  Consider a possibly randomized causal
correction scheme: its input $v_t$ depends only on
$(\delta_2,\ldots,\delta_t)$ and its internal randomness, so it is chosen
before $\delta_{t+1}$ is observed.  Suppose that
\begin{equation}
 \E\!\left[\sum_{t\ge1}\norm{(e_t,v_t)}\right]
 \le \sum_{t\ge1}N(\delta_{t+1})
 \label{eq:generic-additive-certificate}
\end{equation}
under
\[
 e_{t+1}=Ae_t+Bv_t+\delta_{t+1},
 \qquad e_1=0,
\]
for every finitely supported sequence $(\delta_{t+1})_{t\ge1}$ in $\Reach$.
Then $N(z)\ge\Vstar(z)$ for every $z\in\Reach$.  Consequently, $\Vstar$ is
the pointwise smallest time-independent, horizon-uniform separable switch
charge, and every recovery pair $(V,\psi)$ satisfies $V\ge\Vstar$ on
$\Reach$.
\end{theorem}

\begin{proof}[Proof of \cref{thm:minimality}]
On the all-zero sequence, the bound is zero, so the deviation is zero almost
surely.  Now apply \eqref{eq:generic-additive-certificate} to a sequence with
a single nonzero discrepancy $z$.  By causality, the scheme plays as on the
zero sequence before the switch, so it pays nothing there.  Every
sample-path continuation is then a feasible correction from $z$ and costs at
least $\Vstar(z)$.  Taking expectations proves $N(z)\ge\Vstar(z)$.

If $(V,\psi)$ is a recovery pair, \cref{lem:certificate-tracking} with
$e_1=0$ applies to $v_t=\psi(e_t)$ on every horizon $T$ beyond the support
of $(\delta_{t+1})$.  Letting $T\to\infty$ gives
\eqref{eq:generic-additive-certificate} with $N=V$.  Hence $V\ge\Vstar$ on
$\Reach$.
\end{proof}

\subsection{Stable linear corrective laws}
\label{app:fixed-feedback-correction}

For a gain $K$, let $F_K=A-BK$, and price a discrepancy by its total future
deviation under the corrective law $e\mapsto-Ke$:
\begin{equation}
 V_K(z)=\sum_{r\ge0}\norm{(F_K^rz,-KF_K^rz)}.
 \label{eq:linear-correction-value}
\end{equation}

\begin{proposition}[Recovery norm from a stable linear corrective law]
\label{prop:fixed-feedback-certificate}
Let $F_K=A-BK$ and suppose all eigenvalues of its restriction to $\Reach$ lie
strictly inside the unit disk.  Then $V_K$ in
\eqref{eq:linear-correction-value} is a norm on $\Reach$, $(V_K,e\mapsto-Ke)$ is a recovery pair, and
\begin{equation*}
 \Vstar(z)\le V_K(z)
 \quad(z\in\Reach).
\end{equation*}
\end{proposition}

The norm $V_K$ records transient amplification as well as control effort.
For the Euclidean norm, write $M_K$ for its factor $M_{V_K}$ in
\eqref{eq:euclidean-correction-factor}, and define the peak tracking factor
\begin{equation}
 P_K=\sum_{r\ge0}
 \left\|\begin{bmatrix}I\\-K\end{bmatrix}
 (A-BK)^r\big|_{\Reach}\right\|_{\mathrm{op}}.
 \label{eq:peak-tracking-gain}
\end{equation}
Always $M_K\le P_K$.  The factor $M_K$ bounds the total deviation caused by
one initial error.  The factor $P_K$ bounds the peak deviation under repeated
bounded switch discrepancies.

When a particular gain is installed as the corrective law, we write
$K_{\mathrm{cor}}$ to distinguish it from benchmark gains, and abbreviate
$P_{\mathrm{cor}}=P_{K_{\mathrm{cor}}}$.  Suppose that, for some
$\kappa_{\mathrm{cor}}\ge1$ and $\gamma_{\mathrm{cor}}\in(0,1)$,
\[
 \norm{K_{\mathrm{cor}}}_{\mathrm{op}}\le\kappa_{\mathrm{cor}},
 \qquad
 \norm{(A-BK_{\mathrm{cor}})^r|_{\Reach}}_{\mathrm{op}}
 \le\kappa_{\mathrm{cor}}^2(1-\gamma_{\mathrm{cor}})^r
 \quad(r\ge0).
\]
Here $\kappa_{\mathrm{cor}}$ bounds the gain size and transient
amplification, and $\gamma_{\mathrm{cor}}$ is the decay rate.  Summing the
geometric bound gives
\begin{equation}
 M_\star\le M_{K_{\mathrm{cor}}}\le P_{K_{\mathrm{cor}}}
 \le\frac{\kappa_{\mathrm{cor}}^2\sqrt{1+\kappa_{\mathrm{cor}}^2}}
           {\gamma_{\mathrm{cor}}}.
 \label{eq:linear-gamma-certificate}
\end{equation}
A known gain that is stable on $\Reach$ therefore gives a computable recovery
pair.  If some gain lies in $\cK_{\kappa,\gamma}$ (\cref{eq:strong-stability}),
the same estimate gives the existential bound
\begin{equation}
 M_\star\le\frac{\kappa^2\sqrt{1+\kappa^2}}{\gamma}.
 \label{eq:comparator-bounds-correction}
\end{equation}
Finite-horizon model predictive control (MPC) with a terminal recovery norm
gives computable recovery norms that decrease to $\Vstar$ from above, while
$V_h$ increases to it from below (\cref{app:mpc-certificates}).

\begin{proof}[Proof of \cref{prop:fixed-feedback-certificate}]
On the finite-dimensional space $\Reach$, the eigenvalue condition makes the
series in \eqref{eq:linear-correction-value} absolutely convergent.  Scaling
and superposing closed-loop trajectories, together with the corresponding
properties of $\norm{\cdot}$, give absolute homogeneity and the triangle
inequality.  The $r=0$ term gives definiteness.  Thus $V_K$ is a norm.

Separating the first term of the series gives
\[
 V_K(z)=\norm{(z,-Kz)}+V_K(F_Kz),
\]
so $(V_K,z\mapsto-Kz)$ is a recovery pair.  The closed-loop trajectory of this
law is feasible for \eqref{eq:correction-value}, so $\Vstar(z)\le V_K(z)$.

For the Euclidean norm, suppose
$\norm{K}_{\mathrm{op}}\le\kappa$ and
$\norm{F_K^r|_{\Reach}}_{\mathrm{op}}\le\kappa^2(1-\gamma)^r$.
Then
\[
 \norm{(F_K^rz,-KF_K^rz)}
 \le \sqrt{1+\kappa^2}\,\kappa^2(1-\gamma)^r\norm{z}_2.
\]
Summing the geometric series proves
\eqref{eq:linear-gamma-certificate}.  Every member of $\cK_{\kappa,\gamma}$
satisfies these two bounds, because $F_K$ maps $\Reach$ into itself and
restriction does not increase the operator norm.  Applying the estimate to
one member proves \eqref{eq:comparator-bounds-correction}.
\end{proof}

\subsection{Recovery outside the reachable subspace}
\label{app:outside-reach}

What if the recovery cost is extended to all initial errors?  In Kalman
controllability coordinates, in which $\Reach$ is $\{e_\perp=0\}$, the error
system reads
\[
 \begin{bmatrix}e_{r,s+1}\\e_{\perp,s+1}\end{bmatrix}
 =
 \begin{bmatrix}A_r&A_{r\perp}\\0&A_\perp\end{bmatrix}
 \begin{bmatrix}e_{r,s}\\e_{\perp,s}\end{bmatrix}
 +\begin{bmatrix}B_r\\0\end{bmatrix}v_s.
\]
The quotient state evolves autonomously, $e_{\perp,s}=A_\perp^se_{\perp,0}$.
If every eigenvalue of $A_\perp$ lies strictly inside the unit disk, a stabilizing state-feedback gain on the reachable
block makes the full state and input summable, so the extended $\Vstar$ is
finite everywhere.  If $A_\perp$ has an eigenvalue of modulus at least one,
some initial quotient state is not absolutely summable under any input, so
its recovery cost is infinite.  This is exactly the Kalman stabilizability
criterion.

\subsection{Recovery norms from model predictive control}
\label{app:mpc-certificates}

\begin{proposition}[Recovery norms from model predictive control]
\label{prop:mpc-certificates}
Let $P$ be a norm on $\Reach$, and let $\psi_P:\Reach\to\R^{d_u}$ be Borel
measurable and satisfy the control-Lyapunov inequality
\begin{equation}
 \norm{(z,\psi_P(z))}
 +P(Az+B\psi_P(z))\le P(z).
 \label{eq:terminal-certificate}
\end{equation}
For $N\ge0$, define
\begin{equation}
 \widehat V_N(z)
 =\min\left\{
   \sum_{r=0}^{N-1}\norm{(e_r,v_r)}+P(e_N):
   e_0=z,\ e_{r+1}=Ae_r+Bv_r
 \right\}.
 \label{eq:mpc-value}
\end{equation}
Then
\begin{equation*}
 \Vstar\le\widehat V_{N+1}\le\widehat V_N\le P,
 \qquad
 \widehat V_N\downarrow\Vstar.
\end{equation*}
Each $\widehat V_N$ is a norm on $\Reach$.  For $N\ge1$, some Borel
measurable receding-horizon first action $\widehat\psi_N$ makes
$(\widehat V_N,\widehat\psi_N)$ a recovery pair; for $N=0$, the pair is
$(P,\psi_P)$.  When the stage and terminal norms are polyhedral,
\eqref{eq:mpc-value} is a linear program.  For Euclidean or ellipsoidal
norms, it is a second-order cone program with $O(N(d_x+d_u))$ variables.
Fixed state-feedback gains, finite-horizon MPC, and the optimal value
therefore trade computational effort against tightness of the recovery norm.
\end{proposition}

\begin{proof}
\emph{Monotonicity.}
By \eqref{eq:terminal-certificate}, appending one step of $\psi_P$ to any
$N$-stage optimizer does not increase its objective.  Hence
\[
 \widehat V_{N+1}\le\widehat V_N\le\widehat V_0=P.
\]
Continuing with $\psi_P$ forever gives a feasible infinite correction whose
cost is at most the finite-horizon objective, so $\Vstar\le\widehat V_N$.

\emph{Convergence.}
Evaluate the objective of $\widehat V_N$ on an $\varepsilon$-optimal infinite
correction $(e_r,v_r)_{r\ge0}$.  Finite cost implies $e_N\to0$, and
continuity of the norm $P$ gives
\[
 \widehat V_N(z)
 \le\sum_{r=0}^{N-1}\norm{(e_r,v_r)}+P(e_N)
 \longrightarrow
 \sum_{r\ge0}\norm{(e_r,v_r)}
 \le\Vstar(z)+\varepsilon.
\]
Letting $\varepsilon\downarrow0$ proves $\widehat V_N\downarrow\Vstar$.

\emph{Norm.}
Homogeneity and the triangle inequality follow by scaling and superposing
feasible trajectories.  Definiteness comes from $P$ when $N=0$ and from the
first-stage state term when $N\ge1$.

\emph{Recovery pair.}
For $N\ge1$, the finite-horizon objective is continuous and coercive in the
input sequence.  Its argmin correspondence is nonempty, compact valued, and
has a Borel graph, so a Borel measurable optimizer can be selected.  Let
$\widehat\psi_N(e)$ be its first action.  Dynamic programming gives
\[
 \widehat V_N(e)
 =
 \norm{(e,\widehat\psi_N(e))}
 +\widehat V_{N-1}
   (Ae+B\widehat\psi_N(e)).
\]
Since $\widehat V_N\le\widehat V_{N-1}$,
\[
 \norm{(e,\widehat\psi_N(e))}
 +\widehat V_N(Ae+B\widehat\psi_N(e))
 \le\widehat V_N(e).
\]
This is \eqref{eq:correction-certificate} with $V=\widehat V_N$ and
$\psi=\widehat\psi_N$, so $(\widehat V_N,\widehat\psi_N)$ is a recovery pair.
\end{proof}

\subsection{The conserved network}
\label{app:network-formula}

\begin{proof}[Proof of \eqref{eq:network-vstar}]
Since $\mathbf 1^\top D=0$ and the graph is connected,
$\im D=\mathbf 1^\perp$.  Because $A=I$, \eqref{eq:reachable} gives
$\Reach(I,D)=\im D$.  Each input leaves the total, and hence the average,
unchanged, although disturbances may move it.

For $z\in\mathbf 1^\perp$, the one-step input $v_0=-D^\dagger z$ drives $z$
to zero, since $DD^\dagger$ is the orthogonal projection onto
$\im D=\mathbf 1^\perp$.  Its cost is
\[
 \sqrt{\norm{z}_2^2+\norm{D^\dagger z}_2^2},
\]
which gives the upper bound in \eqref{eq:network-vstar}.

For the lower bound, take a finite-cost correction in \cref{ex:network}.
Then $\sum_s\norm{v_s}_2<\infty$ and $e_s\to0$.  Summing
$e_{s+1}-e_s=Dv_s$ gives
\[
 D\sum_{s\ge0}v_s=-z.
\]
Since $-D^\dagger z$ is the minimum-norm solution of this equation,
\[
 \sum_s\norm{v_s}_2
 \ge\norm{\textstyle\sum_s v_s}_2
 \ge\norm{D^\dagger z}_2,
 \qquad
 \sum_s\norm{e_s}_2\ge\norm{z}_2.
\]
Minkowski's inequality in $\R^2$ yields the lower bound in
\eqref{eq:network-vstar}.  Finally, the nonzero singular values of $D$ are
the square roots of the nonzero eigenvalues of $DD^\top$.  Connectivity makes
zero a simple eigenvalue of this Laplacian, so the smallest nonzero one is
$\lambda_2(DD^\top)$, which proves the formula for $M_\star$.
\end{proof}

\section{Proofs for general policy classes}
\label{app:policy-proofs}

We write the arguments in finite-index notation.  Each identity extends to a
measurable policy class by replacing sums with integrals, and the stated
essential bounds justify domination (\cref{app:measurable-policy-proof}).

\subsection{Finite-class aggregation}
\label{app:finite-policy-proof}

\Cref{thm:finite-aggregation-main} uses uniform diameter bounds.  The
following version lets the diameters vary across rounds.

\begin{theorem}[Finite-class aggregation with time-varying diameters]
\label{thm:finite-policy-library}
Suppose $N\ge2$ and $L>0$.  Let each cost be convex and $L$-Lipschitz on
$\conv\{y_t^1,\ldots,y_t^N\}$ and $L$-Lipschitz on the segment joining the
actual point $y_t$ to the reference $\bar y_t(p_t)$.  Let $\mathsf B_T>0$ be a
deterministic bound, known before play, such that
\begin{equation}
 \sum_{t=1}^T D_t^2+\sum_{t=1}^{T-1}D_tD_{t,V}
 \le\mathsf B_T
 \label{eq:aggregation-known-budget}
\end{equation}
for every cost and disturbance sequence allowed by the model.
Counterfactual tracking with exponential weights \eqref{eq:finite-policy-ew}
as the base learner and $\eta=2\sqrt{\log N}/(L\sqrt{\mathsf B_T})$
satisfies
\begin{equation}
 \Reg_T(\Pi)
 \le L\sqrt{\mathsf B_T\log N}-\sum_{t=1}^T J_t(p_t).
 \label{eq:finite-policy-main-bound}
\end{equation}
Known bounds $D_t\le D\le D_V$ and $D_{t,V}\le D_V$ permit
$\mathsf B_T=TD^2+(T-1)DD_V\le2TDD_V$, giving
$O(L\sqrt{TDD_V\log N})$ regret to the best policy.
\end{theorem}

The theorem is the uniform-prior, point-mass case of a bound for any prior
and any comparison distribution.  This is the finite-index form of
\cref{thm:pac-bayes-policies}.

Let $\Pi=\{\pi_1,\ldots,\pi_N\}$ be a finite causal policy class, and set
\[
 y_t^i=(x_t^i,u_t^i)=y_t^{\pi_i},
 \qquad
 \mathsf c_t=(\mathsf c_{t,i})_{i=1}^N
 =(c_t(y_t^i))_{i=1}^N.
\]
For $p\in\simplex N$, use the averaged reference and the averaged loss
\begin{equation}
 \bar y_t(p)=\sum_i p_i y_t^i,
 \qquad
 h_t(p)=\ip{p}{\mathsf c_t}.
 \label{eq:policy-mixture}
\end{equation}
For a comparison distribution $\rho$, write
\begin{equation*}
 \Reg_T(\rho)
 =\sum_{t=1}^T c_t(y_t)
 -\E_{i\sim\rho}\sum_{t=1}^T c_t(y_t^i).
\end{equation*}
The reference is feasible because the dynamics are linear.  Convexity gives
the Jensen gain
\begin{equation*}
 J_t(p)=\sum_i p_i c_t(y_t^i)
 -c_t\!\left(\sum_i p_i y_t^i\right)\ge0.
\end{equation*}

Start the recursion \eqref{eq:finite-policy-ew} from a full-support prior
$p_1=\nu\in\simplex N$ instead of the uniform one.  For a learning rate
$\eta>0$, it then reads
\[
 p_{t+1,i}
 =\frac{\nu_i\exp(-\eta\sum_{s=1}^t\mathsf c_{s,i})}
        {\sum_{j=1}^N\nu_j\exp(-\eta\sum_{s=1}^t\mathsf c_{s,j})}.
\]
The algorithm runs \cref{alg:counterfactual-tracking} with $q_t=p_t$ and the
maps in \eqref{eq:policy-mixture}.  The bounds use the roundwise diameters
$D_t\le D_{t,V}$ of \cref{eq:loss-range,eq:correction-diameter} and the
loss range
\begin{equation*}
 r_t=\max_i\mathsf c_{t,i}-\min_i\mathsf c_{t,i}\le LD_t.
\end{equation*}

\begin{corollary}[Finite-class aggregation with a prior]
\label{cor:finite-library}
Suppose each cost $c_t$ is convex and $L$-Lipschitz on the convex hull of the
policy points $y_t^1,\ldots,y_t^N$, and $L$-Lipschitz on the segment
joining $y_t$ to $\bar y_t(p_t)$.  Then counterfactual tracking with recovery
pair $(V,\psi)$ and exponential weights \eqref{eq:finite-policy-ew}, started
from $p_1=\nu$, as the base learner and any $\eta>0$ satisfies, for every
distribution $\rho\in\simplex N$,
\begin{equation}
 \Reg_T(\rho)
 \le
 \frac{\KL(\rho\Vert\nu)}{\eta}
 +\frac{\eta L^2}{8}\sum_{t=1}^T D_t^2
 +\frac{\eta L^2}{4}\sum_{t=1}^{T-1}D_tD_{t,V}
 -\sum_{t=1}^T J_t(p_t).
 \label{eq:finite-library-bound}
\end{equation}
\end{corollary}

To tune $\eta$ before play when $L>0$, fix an entropy budget $k>0$ and a
bound $\mathsf B_T>0$ satisfying \eqref{eq:aggregation-known-budget} for every
cost and disturbance sequence allowed by the model.  Setting
$\eta=2\sqrt{k}/(L\sqrt{\mathsf B_T})$ gives
\begin{equation}
 \Reg_T(\rho)
 \le L\sqrt{k\mathsf B_T}-\sum_{t=1}^T J_t(p_t)
 \qquad\bigl(\KL(\rho\Vert\nu)\le k\bigr).
 \label{eq:finite-library-uniform}
\end{equation}
A single learner attains this bound for every $\rho$ in the entropy budget.
The tuning uses only bounds known before play.  A uniform prior,
$k=\log N$, and point-mass comparators give \cref{thm:finite-policy-library}.

The degenerate cases need no learning.  If $k=0$, the only comparator is
$\rho=\nu$.  Holding the prior fixed then creates no switches and gives
regret $-\sum_tJ_t(\nu)\le0$.  If $\mathsf B_T=0$, all policy points coincide
on each round, and holding the prior fixed gives zero regret.  The same holds
if $L=0$, whatever the diameters.  A singleton class likewise requires no
learning.  For a measurable class, these statements hold up to $\nu$-null
sets.

The per-round cost is that of \cref{sec:finite-policy-libraries}.

\begin{proof}[Proof of \cref{cor:finite-library}]
The proof bounds the learning regret and the movement of exponential weights,
and converts the movement into a switching cost through a diameter
inequality.

\emph{Diameter inequality.}  Let $N_0$ be a norm and let $z_1,\ldots,z_N$ have
diameter $D=\max_{i,j}N_0(z_i-z_j)$.  Couple $I\sim p$ and $J\sim p'$ so that
$\Pr(I\ne J)=\TV(p,p')$.  Then
\[
 N_0\!\left(\E z_I-\E z_J\right)
 \le\E N_0(z_I-z_J)
 \le D\,\TV(p,p').
\]

\emph{Movement and learning regret.}  Exponential weights is mirror descent
with the negative-entropy mirror map on the simplex.  By Pinsker's
inequality, this map is $1$-strongly convex in $\norm{\cdot}_1$, so
$\sigma=1$.  The dual norm of the loss vector on the affine hull is $r_t/2$.
The mirror-descent movement bound \eqref{eq:convex-reference-movement-proof}
therefore gives
\[
 \norm{p_{t+1}-p_t}_1\le\frac{\eta r_t}{2},
 \qquad
 \TV(p_t,p_{t+1})\le\frac{\eta r_t}{4}.
\]
The standard Hoeffding bound for exponential weights gives
\[
 \sum_t\ip{p_t-\rho}{\mathsf c_t}
 \le\frac{\KL(\rho\Vert\nu)}{\eta}
   +\frac\eta8\sum_t r_t^2.
\]

\emph{Switching cost.}  Apply the diameter inequality to $z_i=x_{t+1}^i$ in
the norm $V$.  Since $V(x_{t+1}^i-x_{t+1}^j)\le D_{t,V}$,
\[
 V\!\left(\sum_i(p_{t,i}-p_{t+1,i})x_{t+1}^i\right)
 \le D_{t,V}\TV(p_t,p_{t+1}).
\]
With $r_t\le LD_t$, the learning remainder on round $t$ is at most
$(\eta L^2/8)D_t^2$, and the switching cost is at most
$(\eta L^2/4)D_tD_{t,V}$ for $t<T$.  Applying \cref{thm:master-reduction}
and retaining the Jensen gain proves \eqref{eq:finite-library-bound}.

\emph{Tuning.}  Under \eqref{eq:aggregation-known-budget}, the two positive
remainder sums are at most $\eta L^2\mathsf B_T/4$.  Substituting the
prescribed learning rate and $\KL(\rho\Vert\nu)\le k$ proves
\eqref{eq:finite-library-uniform}, with the Jensen gain of the same learner.
\end{proof}

\subsection{The variance form of exponential weights}
\label{app:variance-aggregation-proof}

The diameter bounds above ignore how concentrated the weights are.  The fast
rate under strong convexity needs bounds in terms of the spread of the current
mixture instead.  Define
\begin{align}
 \mathsf V_t^y
 &=\E_{i\sim p_t}\norm{y_t^i-\bar y_t(p_t)}_2^2,
 \label{eq:trajectory-variance}\\
 \mathsf V_t^c
 &=\Var_{i\sim p_t}(\mathsf c_{t,i}),
 \label{eq:loss-variance}\\
 \mathsf V_t^V
 &=\frac12\E_{i,j\sim p_t}\norm{y_t^i-y_t^j}_V^2.
 \label{eq:correction-dispersion}
\end{align}
The last quantity is a pairwise squared dispersion in the trajectory norm
$\norm{\cdot}_V$ of \eqref{eq:induced-trajectory-norm}.  It need not be a
Hilbert-space variance.

\begin{lemma}[Variance form of finite-class aggregation]
\label{lem:variance-aggregation}
Let each $c_t$ be $L$-Lipschitz on the segment joining $y_t$ to
$\bar y_t(p_t)$, as in \cref{thm:master-reduction}.  There is a universal
constant $C$ with the following property.  If $\eta r_t\le1/2$ on every
round, then counterfactual tracking with exponential weights
\eqref{eq:finite-policy-ew}, started from $p_1=\nu$, as the base learner
satisfies, for every $\rho\in\simplex N$,
\begin{align}
 \Reg_T(\rho)
 \le{}&
 \frac{\KL(\rho\Vert\nu)}{\eta}
 +C\eta\sum_{t=1}^T\mathsf V_t^c
 +4L\eta\sum_{t=1}^{T-1}
   \sqrt{\mathsf V_t^V\mathsf V_t^c}
 -\sum_{t=1}^T J_t(p_t).
 \label{eq:localized-gibbs}
\end{align}
\end{lemma}

\begin{proof}[Proof of \cref{lem:variance-aggregation}]
We bound the movement by a covariance and the learning regret by a
variance, and then apply \cref{thm:master-reduction}.

\emph{A covariance inequality for arbitrary norms.}  For a distribution $p$
over vectors $z_i$ and scalars $b_i$, and a norm $N_0$, define
\[
 A_p=\frac12\E_{i,j\sim p}N_0(z_i-z_j)^2,
 \qquad
 S_p=\Var_{i\sim p}(b_i).
\]
Symmetrization and Cauchy--Schwarz give
\begin{align}
 N_0\!\left(
 \E_p[(z_i-\E_pz_i)(b_i-\E_pb_i)]
 \right)
 &\le
 \frac12\E_{i,j}N_0(z_i-z_j)\abs{b_i-b_j}
 \notag\\
 &\le\sqrt{A_pS_p}.
 \label{eq:norm-covariance}
\end{align}

\emph{Movement.}  Interpolate between the old and new weights by
\[
 p_i(s)=
 \frac{p_{t,i}e^{-s\mathsf c_{t,i}}}
      {\sum_j p_{t,j}e^{-s\mathsf c_{t,j}}},
 \qquad 0\le s\le\eta.
\]
Then $p(0)=p_t$ and $p(\eta)=p_{t+1}$.  The derivative of the mean
$\sum_i p_i(s)y_t^i$ is minus the covariance of $y_t^i$ and
$\mathsf c_{t,i}$ under $p(s)$.  Apply \eqref{eq:norm-covariance} at each
$s$, with $z_i=y_t^i$, $b_i=\mathsf c_{t,i}$, and $N_0=\norm{\cdot}_V$.  The
next-state map is linear on trajectory differences, and
$V(Ae+Bv)\le\norm{(e,v)}_V$.  If $\eta r_t\le1/2$, center the losses at the
midpoint of their range.  The density ratios $p_i(s)/p_{t,i}$ along the tilt
then lie in $[e^{-1/2},e^{1/2}]$, so a product of two such ratios is at most
$e<4$.  Integrating the covariance bound over $s\in[0,\eta]$ gives
\begin{equation}
 V\bigl(\bar x_{t+1}(p_t)-\bar x_{t+1}(p_{t+1})\bigr)
 \le4\eta\sqrt{\mathsf V_t^V\mathsf V_t^c}.
 \label{eq:covariance-movement}
\end{equation}

\emph{Second-order Gibbs regret.}  Let $X=\mathsf c_{t,i}-\E_{p_t}\mathsf c_t$
with $i\sim p_t$.  Under $\eta r_t\le1/2$, Taylor's inequality gives
\[
 \log\E_{p_t}e^{-\eta X}
 \le C\eta^2\mathsf V_t^c.
\]
Telescoping the log partition functions and using the Gibbs variational
formula yields
\begin{equation}
 \sum_t\ip{p_t-\rho}{\mathsf c_t}
 \le
 \frac{\KL(\rho\Vert\nu)}{\eta}
 +C\eta\sum_t\mathsf V_t^c.
 \label{eq:second-order-gibbs}
\end{equation}
Combining \eqref{eq:second-order-gibbs} and \eqref{eq:covariance-movement}
with \cref{thm:master-reduction} proves \cref{lem:variance-aggregation}.
\end{proof}

\subsection{Absorption by the Jensen gain}
\label{app:strong-policy-proof}

Under strong convexity, the Jensen gain grows with the spread
$\mathsf V_t^y$, while every positive term in \eqref{eq:localized-gibbs}
grows at most linearly in it.  A small enough learning rate lets the gain
absorb them.  For a finite class with a full-support prior $\nu$, all
integrability conditions hold automatically.  The assumptions of
\cref{thm:horizon-independent-aggregation} then require the distortion bound
for every pair $i,j$, and its conclusion reads
\begin{equation*}
 \Reg_T(\rho)\le C_0\frac{L^2}{\mu}\alpha\KL(\rho\Vert\nu)
 \qquad(\rho\in\simplex N).
\end{equation*}
A uniform prior and a point-mass comparator give
\cref{cor:horizon-independent-finite-library}.

\begin{proof}[Finite-class case of \cref{thm:horizon-independent-aggregation}]
Strong convexity and Jensen's inequality give
\begin{equation*}
 J_t(p_t)
 \ge\frac\mu2\mathsf V_t^y.
\end{equation*}
Lipschitzness gives $\mathsf V_t^c\le L^2\mathsf V_t^y$.  The pairwise
identity $\frac12\E_{i,j\sim p_t}\norm{y_t^i-y_t^j}_2^2=\mathsf V_t^y$ and the
distortion assumption give $\mathsf V_t^V\le \alpha^2\mathsf V_t^y$ for
$t<T$.  Since $\alpha\ge1$, every positive variance term in
\eqref{eq:localized-gibbs}, including the learning term on the last round,
is at most
\[
 C_1\eta L^2\alpha\mathsf V_t^y,
\]
where $C_1=C+4$ for the constant $C$ of \cref{lem:variance-aggregation}.
The learning rate of \cref{thm:horizon-independent-aggregation}, with a
sufficiently small universal constant, makes this at most
$(\mu/2)\mathsf V_t^y$.

It remains to verify the tilt condition $\eta r_t\le1/2$.  For two points
$y,y'$ in the counterfactual hull with midpoint $m$, strong convexity and
Lipschitzness give
\[
 \frac{\mu}{8}\norm{y-y'}_2^2
 \le\frac{c_t(y)+c_t(y')}{2}-c_t(m)
 \le\frac{L}{2}\norm{y-y'}_2.
\]
Thus every pair distance is at most $4L/\mu$, and the loss range is at most
$4L^2/\mu$.  The same universal choice of learning rate makes
$\eta r_t\le1/2$.

Finally, for every $e\in\Reach$ and $v\in\R^{d_u}$,
\[
 \norm{(e,v)}_V
 \le
 \bigl(1+M_V\norm{[A\ B]}_{\mathrm{op}}\bigr)\norm{(e,v)}_2,
\]
which justifies the class-independent choice
\eqref{eq:global-trajectory-distortion}.  Since $\Vstar\le V$ by
\cref{thm:minimality}, the optimal recovery cost gives the smallest trajectory
norm, and hence the smallest distortion, whether global or local to the
comparators.
\end{proof}

\subsection{Infinite policy classes}
\label{app:measurable-policy-proof}

\paragraph{Distributional bounds and algorithms.}
\label{app:distributional-statements}

The construction extends from a finite class to a measurable one.  A
probability measure $\nu$ assigns the initial weights and sets the scale of
complexity.  It is not a Bayesian model for the costs or disturbances.  The
learner's weights are a distribution $p_t$, and $\rho$ denotes any
distribution used for comparison.  Relative entropy $\KL(\rho\Vert\nu)$
replaces the finite-class complexity $\log N$.

Let $\Pi$ be a measurable policy space with initial distribution $\nu$.  We
assume standing measurability and integrability conditions, stated
explicitly before the proof.  They make the averaged references and the
normalizing integrals in \eqref{eq:pac-bayes-gibbs} well defined.  Write
$\mathsf c_t(\pi)=c_t(y_t^\pi)$ and set
\begin{align}
 \bar y_t(p)&=\E_{\pi\sim p}y_t^\pi,
 &h_t(p)&=\E_{\pi\sim p}\mathsf c_t(\pi),
 \notag\\
 J_t(p)&=h_t(p)-c_t(\bar y_t(p))\ge0,
 \label{eq:pac-bayes-jensen}\\
 \Reg_T(\rho)&=\sum_{t=1}^T c_t(y_t)
 -\E_{\pi\sim\rho}\sum_{t=1}^T c_t(y_t^\pi).
 \label{eq:pac-bayes-regret}
\end{align}
Regret \eqref{eq:pac-bayes-regret} compares with the average cumulative cost
under $\rho$.  Relative entropy is
\[
 \KL(\rho\Vert\nu)=\int\log(d\rho/d\nu)\,d\rho,
\]
with value $+\infty$ unless $\rho\ll\nu$.  For $\eta>0$, exponential weights
starts at $p_1=\nu$ and uses the Gibbs distribution
\begin{equation}
 p_t(d\pi)=
 \frac{\exp(-\eta\sum_{s<t}\mathsf c_s(\pi))}
 {\int_\Pi\exp(-\eta\sum_{s<t}\mathsf c_s(\pi'))\,\nu(d\pi')}
 \,\nu(d\pi).
 \label{eq:pac-bayes-gibbs}
\end{equation}
For a finite class, this is exactly \eqref{eq:finite-policy-ew}.  The
algorithm again uses $q_t=p_t$ in \cref{alg:counterfactual-tracking}.  The
construction is deterministic: ``prior'' and ``Gibbs'' introduce no
stochastic model of the data.

Let $D_t$ be the $\nu\otimes\nu$-essential pairwise diameter of $y_t^\pi$ in
$\norm{\cdot}$, and $D_{t,V}$ the one in $\norm{\cdot}_V$.  For a finite class
with a full-support prior, these are the maxima of
\cref{eq:loss-range,eq:correction-diameter}.  Since
$\norm{\cdot}\le\norm{\cdot}_V$, also $D_t\le D_{t,V}$.

\Needspace{12\baselineskip}
\begin{theorem}[Aggregation over infinite policy classes]
\label{thm:pac-bayes-policies}
Assume the standing measurability and integrability conditions and
$D_t,D_{t,V}<\infty$.  Suppose $c_t$ is convex and $L$-Lipschitz on the closed
convex hull of the essential range of $\pi\mapsto y_t^\pi$ under $\nu$, and
$L$-Lipschitz on the segment joining $y_t$ to $\bar y_t(p_t)$.  Then
counterfactual tracking with exponential weights \eqref{eq:pac-bayes-gibbs}
as the base learner and any $\eta>0$ satisfies, for every $\rho\ll\nu$ with
finite relative entropy,
\begin{equation}
 \Reg_T(\rho)
 \le \frac{\KL(\rho\Vert\nu)}{\eta}
 +\frac{\eta L^2}{8}\sum_{t=1}^T D_t^2
 +\frac{\eta L^2}{4}\sum_{t=1}^{T-1}D_tD_{t,V}
 -\sum_{t=1}^T J_t(p_t).
 \label{eq:pac-bayes-bound}
\end{equation}
\end{theorem}

The bound holds for every comparison distribution at a fixed rate.  To tune
the rate before play when $L>0$, fix an entropy budget $k>0$ and a
deterministic bound $\mathsf B_T>0$ satisfying
\eqref{eq:aggregation-known-budget}.  Taking
$\eta=2\sqrt{k}/(L\sqrt{\mathsf B_T})$ gives, for the same learner and all
$\rho$ in the budget,
\begin{equation}
 \Reg_T(\rho)
 \le L\sqrt{k\mathsf B_T}-\sum_{t=1}^T J_t(p_t)
 \qquad\bigl(\KL(\rho\Vert\nu)\le k\bigr).
 \label{eq:pac-bayes-uniform}
\end{equation}
A uniform prior on a finite class permits $k=\log N$ and recovers
\eqref{eq:finite-policy-main-bound}.  \Cref{app:finite-policy-proof} treats
the zero-budget cases.

For a nonatomic prior, a point mass has infinite relative entropy.
Competing with a single policy then requires a localized distribution
$\rho\ll\nu$ and a bound on its approximation error.  In a continuously
parameterized class, for example, one may compare with $\nu$ conditioned on
a small neighborhood of a target parameter.  Its prior mass determines the
relative-entropy cost, and continuity of the trajectories determines the
approximation error.  Exact implementation also requires the Gibbs
normalizer and the trajectory integrals.  Sampling or variational
approximations would need a separate error analysis.

\paragraph{Strongly convex costs.}
The finite-class absorption argument has the following relative-entropy form.

\Needspace{10\baselineskip}
\begin{theorem}[Horizon-independent aggregation over infinite policy classes]
\label{thm:horizon-independent-aggregation}
Assume the standing measurability and integrability conditions, and use the
Euclidean norm $\norm{\cdot}=\norm{\cdot}_2$.  Let $\mu,L>0$.  Suppose each
$c_t$ is $\mu$-strongly convex and $L$-Lipschitz on the closed convex hull of
the essential range of $\pi\mapsto y_t^\pi$ under $\nu$, and $L$-Lipschitz
on the segment joining $y_t$ to $\bar y_t(p_t)$.  Suppose a known
$\alpha\ge1$ satisfies
\begin{equation*}
 \norm{y_t^\pi-y_t^{\pi'}}_V
 \le\alpha\norm{y_t^\pi-y_t^{\pi'}}_2
 \quad(t<T)
\end{equation*}
for $\nu\otimes\nu$-almost every pair.  For a sufficiently small universal
$c_0>0$, counterfactual tracking with exponential weights
\eqref{eq:pac-bayes-gibbs} as the base learner and $\eta=c_0\mu/(L^2\alpha)$
satisfies, for every $\rho\ll\nu$ of finite relative entropy,
\begin{equation}
 \Reg_T(\rho)\le C_0\frac{L^2}{\mu}\alpha\KL(\rho\Vert\nu),
 \label{eq:pac-bayes-fast}
\end{equation}
where $C_0$ is universal.
\end{theorem}

A uniform prior on a finite class and a point-mass comparator recover
\cref{cor:horizon-independent-finite-library}.  The factor $\alpha$ compares
recovery and Euclidean geometry only on differences of policy trajectories.
The class-independent choice \eqref{eq:global-trajectory-distortion} also
applies, with $M_V$ the worst-direction constant
\eqref{eq:euclidean-correction-factor}.  Replacing $V$ by the pointwise-minimal
recovery norm $\Vstar$ can only decrease the trajectory norm and the resulting
distortion.  A computable recovery norm gives an implementable upper bound.

The standing conditions are as follows.  For every round $t$, the maps
$\pi\mapsto y_t^\pi$ and $\pi\mapsto x_{t+1}^\pi$ are measurable and satisfy
\[
 \E_{\pi\sim p}\norm{y_t^\pi}_2<\infty,
 \qquad
 \E_{\pi\sim p}\norm{x_{t+1}^\pi}_2<\infty
\]
under every Gibbs distribution $p$ used by the update.  The normalizing
integrals in \eqref{eq:pac-bayes-gibbs} are finite.  Since states and inputs
are finite-dimensional, the expectations are ordinary vector expectations,
taken coordinate by coordinate.

\begin{proof}[Proofs of \cref{thm:pac-bayes-policies,thm:horizon-independent-aggregation} for measurable classes]
The first-moment conditions make the averaged trajectories well defined.  We
redo the three steps of the finite proofs without coupling or
differentiating entropy on a space of measures.

\emph{Diameter inequality.}  For probabilities $p,q\ll\nu$, let $\lambda$ be
their common part, whose density with respect to $p+q$ is the pointwise
minimum of the two densities.  The residual measures $p-\lambda$ and
$q-\lambda$ both have mass $\TV(p,q)$.  Normalize them.  Fubini's theorem and
the triangle inequality then give, for every integrable finite-dimensional
map $z$ whose $\nu\otimes\nu$-essential pairwise diameter in a norm $N_0$ is
at most $D$,
\[
 N_0\!\left(\int z\,dp-\int z\,dq\right)
 \le D\,\TV(p,q).
\]
The residual measures are absolutely continuous with respect to $\nu$, so the
essential diameter applies.

\emph{Movement and learning regret.}  The Gibbs update gives the identity
\[
 \KL(p_{t+1}\Vert p_t)+\KL(p_t\Vert p_{t+1})
 =\eta\int\mathsf c_t\,(dp_t-dp_{t+1}).
\]
The left side is at least $4\TV(p_t,p_{t+1})^2$ by Pinsker's inequality.
With $r_t$ the essential loss range, the right side is at most
$\eta r_t\TV(p_t,p_{t+1})$.  Thus $\TV(p_t,p_{t+1})\le\eta r_t/4$.  By the
Gibbs variational formula and Hoeffding's lemma, every $\rho\ll\nu$ of
finite relative entropy satisfies
\[
 \sum_t\left(\int\mathsf c_t\,dp_t-\int\mathsf c_t\,d\rho\right)
 \le \frac{\KL(\rho\Vert\nu)}{\eta}
     +\frac{\eta}{8}\sum_t r_t^2.
\]

\emph{Convex costs.}  Lipschitzness gives $r_t\le LD_t$.  The diameter
inequality for the next state $x_{t+1}^\pi$ in the norm $V$ bounds the
round-$t$ switch by $D_{t,V}\TV(p_t,p_{t+1})$.  Substituting into
\cref{thm:master-reduction}, with the Jensen gain \eqref{eq:pac-bayes-jensen},
proves \eqref{eq:pac-bayes-bound}.  For \eqref{eq:pac-bayes-uniform}, the
known bound \eqref{eq:aggregation-known-budget} makes the positive remainder
at most $\eta L^2\mathsf B_T/4$.  Substituting
$\eta=2\sqrt{k}/(L\sqrt{\mathsf B_T})$ and $\KL(\rho\Vert\nu)\le k$ gives the
claim.

\emph{Strongly convex costs.}  Define the integral analogues of
\cref{eq:trajectory-variance,eq:loss-variance,eq:correction-dispersion}.
The symmetrization in \eqref{eq:norm-covariance} uses only product
integration, so it holds under the stated essential bounds.  The
exponential-tilt and log-partition calculations extend by domination.  Thus
\eqref{eq:localized-gibbs} holds with integrals.  Strong convexity gives the
same negative term $(\mu/2)\mathsf V_t^y$.  Lipschitzness and the
almost-everywhere distortion hypothesis bound every positive variance term by
$C_1\eta L^2\alpha\mathsf V_t^y$, as in \cref{app:strong-policy-proof}.  The
learning rate of the theorem makes the negative term absorb them, which
proves \eqref{eq:pac-bayes-fast}.
\end{proof}

\subsection{A trajectory region known before play}
\label{app:policy-region-proof}

\paragraph{An explicit region.}
The Lipschitz assumptions above concern segments joining the actual point to
its reference, which depend on the run.  Can one bound these segments before
play?  With a known linear corrective gain, yes.  Suppose
$\norm{y_t^\pi}_2\le D_{\mathrm{ref}}$ for all rounds and all policies under
consideration, and track their averages with $-K_{\mathrm{cor}}e_t$.  With the
peak tracking factor $P_{\mathrm{cor}}$ of \eqref{eq:peak-tracking-gain},
every sequence of averaging weights gives
\begin{equation}
 \norm{y_t}_2\le D_{\mathrm{tr}}
 :=(1+2P_{\mathrm{cor}})D_{\mathrm{ref}}
 \qquad(t\le T).
 \label{eq:policy-tracking-region}
\end{equation}
An $L$-Lipschitz bound on the ball of radius $D_{\mathrm{tr}}$ then covers
every actual--reference segment.  Strong convexity is still needed only on
the counterfactual hull.  The ball is a sufficient region and can be larger than
the region actually visited.

\begin{lemma}[Trajectory region for policy aggregation]
\label{lem:policy-tracking-region}
Use the Euclidean norm $\norm{\cdot}=\norm{\cdot}_2$.  Let $K_{\mathrm{cor}}$ stabilize the
reachable subspace, with finite peak tracking factor $P_{\mathrm{cor}}$.
Suppose all counterfactual paths start at the actual initial state and
$\norm{y_t^\pi}_2\le D_{\mathrm{ref}}$ for $t\le T$ and every policy in a
finite class.  For a measurable class, assume the bound holds
$\nu$-almost surely and use averaging distributions $p_t\ll\nu$.
Then counterfactual tracking with corrective input $-K_{\mathrm{cor}}e_t$
satisfies \eqref{eq:policy-tracking-region} for every predictable sequence
of averaging distributions, and all actual--reference segments lie in the
ball of radius $D_{\mathrm{tr}}$.
\end{lemma}

\begin{proof}
Averages of the reference points have norm at most $D_{\mathrm{ref}}$.
For $t<T$, the switch discrepancy is the difference of two averages of the
counterfactual states at time $t+1$, so
\[
 \norm{\delta_{t+1}}_2
 =\left\|\int x_{t+1}^\pi\,p_t(d\pi)
         -\int x_{t+1}^\pi\,p_{t+1}(d\pi)\right\|_2
 \le2D_{\mathrm{ref}}.
\]
These discrepancies lie in $\Reach$.  Set
$F_{\mathrm{cor}}=A-BK_{\mathrm{cor}}$.  Since $e_1=0$, the error recursion
gives
\[
 \begin{bmatrix}e_t\\-K_{\mathrm{cor}}e_t\end{bmatrix}
 =\sum_{s=1}^{t-1}
   \begin{bmatrix}I\\-K_{\mathrm{cor}}\end{bmatrix}
   F_{\mathrm{cor}}^{t-1-s}\delta_{s+1}.
\]
The definition of $P_{\mathrm{cor}}$ therefore gives
\[
 \norm{y_t-\bar y_t(p_t)}_2
 \le2P_{\mathrm{cor}}D_{\mathrm{ref}}.
\]
Adding the bound on the reference proves \eqref{eq:policy-tracking-region}.
The ball is convex, so it contains every actual--reference segment.  The
proof uses no bound on the learning rate or on the movement of the weights.
Under an essential reference bound, absolute continuity of the averaging
measures justifies the same estimates.
\end{proof}

In particular, a gradient bound $L$ on this known ball verifies the Lipschitz
assumptions of \cref{thm:horizon-independent-aggregation}.  With the recovery
norm $V_{K_{\mathrm{cor}}}$, one may take
$\alpha=1+P_{\mathrm{cor}}\norm{[A\ B]}_{\mathrm{op}}$ in its distortion
condition, since $M_{K_{\mathrm{cor}}}\le P_{\mathrm{cor}}$.  The radius, and
hence the Lipschitz constant, can depend on the corrective gain.  A bound on
the comparator region alone need not suffice.

\subsection{A finite cover of stable state-feedback gains}
\label{app:strong-stable-cover-proof}

How fine must a finite net of $\cK_{\kappa,\gamma}$ be, and how large is it?
We cover the feedback matrices in operator norm, so the covering radius does
not grow with dimension.  A finite-horizon trajectory estimate then makes the
logarithmic factor uniform in the decay margin $\gamma$ at fixed cost and
recovery scales.  The geometric estimate remains available for a sharper
bound.  Set $C_K=1+2\kappa$ and define
\begin{equation}
\begin{aligned}
 \mathcal L_T^{\mathrm{traj}}
 =W\biggl[{}
 &\kappa^2\min\{T-1,\gamma^{-1}\}\\
 &+(1+\kappa)\kappa^4\norm{B}_{\mathrm{op}}
   \min\!\left\{\frac{(T-1)(T-2)}2,\gamma^{-2}\right\}
 \biggr],\\
 \mathcal L_{\kappa,\gamma}^{\mathrm{traj}}
 =W\biggl[{}
 &\frac{\kappa^2}{\gamma}
 +\frac{(1+\kappa)\kappa^4\norm{B}_{\mathrm{op}}}{\gamma^2}
 \biggr].
\end{aligned}
 \label{eq:stable-cover-scales}
\end{equation}
In particular, $\mathcal L_T^{\mathrm{traj}}
\le\mathcal L_{\kappa,\gamma}^{\mathrm{traj}}$.

Write $d_K=d_xd_u$.  The following quantitative form keeps the
finite-horizon trajectory modulus and implies
\cref{cor:strong-feedback-main}.

\begin{corollary}[Exponential weights over a cover of stable state-feedback gains]
\label{cor:strong-stable-cover}
Suppose $\cK_{\kappa,\gamma}$ is nonempty, fix a recovery pair $(V,\psi)$, and
let $\mu,L>0$.  Suppose each cost is $\mu$-strongly convex on the hull of the
class's counterfactual points, and $L$-Lipschitz both there and on the
segments joining the actual and reference points generated by the algorithm.
Suppose a known $\alpha\ge1$ satisfies
\begin{equation*}
 \norm{y_t^K-y_t^{K'}}_V\le\alpha\norm{y_t^K-y_t^{K'}}_2
 \qquad(K,K'\in\cK_{\kappa,\gamma},\ t<T).
\end{equation*}
Let $c_0$ be the universal constant of
\cref{cor:horizon-independent-finite-library}, and let $\mathcal N_\varepsilon$
be a maximal $\varepsilon$-separated subset of $\cK_{\kappa,\gamma}$ in
operator norm, with
\[
 \varepsilon
 =\min\!\left\{1,
   \frac{L\alpha}{\mu T\mathcal L_T^{\mathrm{traj}}}
  \right\}
\]
($\varepsilon=1$ if $\mathcal L_T^{\mathrm{traj}}=0$).  Then counterfactual
tracking with this recovery pair, exponential weights
\eqref{eq:finite-policy-ew} over the policies $u_t=-Kx_t$
($K\in\mathcal N_\varepsilon$) as the base learner, and
$\eta=c_0\mu/(L^2\alpha)$ satisfies
\begin{equation}
 \Reg_T(\cK_{\kappa,\gamma})
 \le C\frac{L^2}{\mu}\alpha d_K\log\!\left(
 C_K\max\!\left\{1,
 \frac{\mu T\mathcal L_T^{\mathrm{traj}}}{L\alpha}
 \right\}\right),
 \label{eq:strong-stable-cover}
\end{equation}
where $C$ is universal.  In particular, with
$L,\mu,\alpha,W,\kappa$, and $\norm{B}_{\mathrm{op}}$ fixed,
\begin{equation}
 \Reg_T(\cK_{\kappa,\gamma})=O(d_xd_u\log T),
 \label{eq:strong-stable-dimension-rate}
\end{equation}
where the implied constant is independent of $d_x,d_u,T$, and $\gamma$.
\end{corollary}

\begin{proof}[Proof of \cref{cor:strong-stable-cover}]
We bound the size of the net, bound how far a policy's trajectory is from
that of its net point, and balance the two.

\emph{Size of the net.}  By definition, $0<\varepsilon\le1$.  The open
operator-norm balls of radius $\varepsilon/2$ around the gains in
$\mathcal N_\varepsilon$ are disjoint and lie inside the
ball of radius $\kappa+\varepsilon/2$ in the $d_K$-dimensional matrix space.
Comparing volumes gives
\[
 |\mathcal N_\varepsilon|
 \le\left(1+\frac{2\kappa}{\varepsilon}\right)^{d_K}
 \le(C_K/\varepsilon)^{d_K}.
\]
By maximality, $\mathcal N_\varepsilon$ is an $\varepsilon$-net, and every
net point is itself an admissible gain.

\emph{Trajectory sensitivity.}  Write $F_K=A-BK$.  For
$K,\widehat K\in\cK_{\kappa,\gamma}$ with
$\norm{K-\widehat K}_{\mathrm{op}}\le\varepsilon$, the telescoping identity
for matrix powers and \eqref{eq:strong-stability} give
\[
 \norm{F_K^r-F_{\widehat K}^r}_{\mathrm{op}}
 \le r\kappa^4\norm{B}_{\mathrm{op}}\varepsilon(1-\gamma)^{r-1}
 \qquad(r\ge1).
\]
Since $x_{\mathrm{ini}}=0$, convolving with the disturbances gives, for
$t\ge2$,
\begin{align*}
 \norm{x_t^K-x_t^{\widehat K}}_2
 &\le W\kappa^4\norm{B}_{\mathrm{op}}\varepsilon
       \sum_{r=1}^{t-2}r(1-\gamma)^{r-1},\\
 \norm{x_t^{\widehat K}}_2
 &\le W\kappa^2\sum_{r=0}^{t-2}(1-\gamma)^r.
\end{align*}
The input identity
$u_t^K-u_t^{\widehat K}
=-K(x_t^K-x_t^{\widehat K})-(K-\widehat K)x_t^{\widehat K}$ implies
\[
 \norm{y_t^K-y_t^{\widehat K}}_2
 \le(1+\kappa)\norm{x_t^K-x_t^{\widehat K}}_2
      +\varepsilon\norm{x_t^{\widehat K}}_2.
\]
Each geometric sum is at most its infinite sum and at most the sum with all
decay factors replaced by one.  The first-round difference is zero.  Thus
\begin{equation*}
 \sup_{t\le T}\norm{y_t^K-y_t^{\widehat K}}_2
 \le\mathcal L_T^{\mathrm{traj}}\varepsilon.
\end{equation*}

\emph{Balancing.}  Apply \cref{thm:horizon-independent-aggregation} with the
prescribed recovery pair, a uniform prior on $\mathcal N_\varepsilon$, and
the distortion bound $\alpha$.  If the net has a single element, hold that
policy fixed instead.  Replacing a comparator by its net point costs at most
$LT\mathcal L_T^{\mathrm{traj}}\varepsilon$ by Lipschitzness.  Hence
\[
 \Reg_T(\cK_{\kappa,\gamma})
 \le C_0\frac{L^2}{\mu}\alpha d_K\log(C_K/\varepsilon)
      +LT\mathcal L_T^{\mathrm{traj}}\varepsilon.
\]
For the $\varepsilon$ of the statement, the approximation term is at most
$(L^2/\mu)\alpha$.  Since $d_K\ge1$ and $C_K\ge3$, a larger universal
constant absorbs it into the logarithmic term.  This proves
\eqref{eq:strong-stable-cover}.

\emph{Dimension-explicit rate.}  Put
\[
 C_{\mathrm{dyn}}
 =W\bigl[\kappa^2+(1+\kappa)\kappa^4\norm{B}_{\mathrm{op}}\bigr].
\]
The finite-horizon bounds in \eqref{eq:stable-cover-scales} give
$\mathcal L_T^{\mathrm{traj}}\le C_{\mathrm{dyn}}T^2$.  The logarithm in
\eqref{eq:strong-stable-cover} is therefore at most
\[
 \log C_K+3\log T+
 \log\max\!\left\{1,\frac{\mu C_{\mathrm{dyn}}}{L\alpha}\right\}.
\]
For $T\ge2$, this is $O(\log T)$ at the fixed scales of the corollary, with a
constant independent of $d_x,d_u,T$, and $\gamma$.  Since $d_K=d_xd_u$, this
proves \eqref{eq:strong-stable-dimension-rate}.
\end{proof}

\section{Proofs for disturbance responses and linear state feedback}
\label{app:response-proofs}

Unless stated otherwise, we keep the conventions of
\cref{sec:response-learning-update}: zero initial state, disturbance bound
$W$, $T\ge2$, and the Euclidean norm.  The subsections follow
\cref{sec:responses,sec:familiar-benchmarks}: the weighted response bound,
the full response ball, stable state feedback, and the cheaper
implementations.

\paragraph{Zero scales.}
The learning rates below assume that the displayed scales are positive.  If
the loss or disturbance scale vanishes, or if the response or local-region
radius is zero, holding a feasible reference fixed gives zero regret.  If
the spectral coefficient radius is zero, the algorithm follows the
representation baseline.  These conventions avoid dividing by zero.

\subsection{Response learning: feasibility and regret}
\label{app:direct-response-proof}

For a nonempty closed convex prefix set $\cC\subseteq\cS_H$, write
\begin{equation*}
 \Reg_T(\cC)=\sum_{t=1}^Tc_t(y_t)
 -\inf_{\Phi\in\cC}\sum_{t=1}^Tc_t(\bar y_t(\Phi)).
\end{equation*}
For a class of infinite responses, $\bar y_t(\Phi)$ is given by
\eqref{eq:response-reference-infinite}.

\paragraph{General weighted-response bound.}

How much does a change of coefficients at a given disturbance age move the
trajectory?  To measure this, fix $\beta\in(0,1)$ and equip response
differences with the weighted Euclidean norm
\eqref{eq:weighted-response-norm} on the free coefficient pairs, omitting the
fixed coefficient $\Phi_x[1]=I$.  For $\Phi\in\cS_H$, write
$\norm{\Phi}_{\beta,H,\circ}$ for the same expression on the free
coefficients.  The inverse weights normalize responses whose effects decay
like $\beta^{i-1}$ with age $i$.  The dual geometric sum
\begin{equation*}
 S_{\beta,H}
 =\left(\sum_{i=1}^H\beta^{2(i-1)}\right)^{1/2}
\end{equation*}
measures how the normalized coefficients combine in a state--input point.

Let $\cC\subseteq\cS_H$ be a nonempty closed convex set of diameter at most $R$ in
\eqref{eq:weighted-response-norm}, and let
\begin{equation*}
 \mathsf H_{\cC}=\Span(\cC-\cC)
\end{equation*}
be the direction space of its affine hull.  A change of coefficients affects
both the current loss and the state that must be recovered after a switch.
Two sensitivities measure these effects per unit change:
\begin{align*}
 \Lambda_t
 &=\sup_{\substack{\Phi,\Phi'\in\cC\\\Phi\ne\Phi'}}
 \frac{\norm{\bar y_t(\Phi)-\bar y_t(\Phi')}_2}
      {\norm{\Phi-\Phi'}_{\beta,H,\circ}},
 \qquad t\le T,
\\
 \Lambda_{t,V}
 &=\sup_{\substack{\Phi,\Phi'\in\cC\\\Phi\ne\Phi'}}
 \frac{\norm{\bar y_t(\Phi)-\bar y_t(\Phi')}_V}
      {\norm{\Phi-\Phi'}_{\beta,H,\circ}},
 \qquad t<T.
\end{align*}
Both are zero when $\cC$ is a singleton.  For $t<T$, the second uses the
recovery-augmented norm of the reduction:
\[
 \norm{\bar y_t(\Phi)-\bar y_t(\Phi')}_V
 =\norm{\bar y_t(\Phi)-\bar y_t(\Phi')}_2
  +V\bigl(\bar x_{t+1}(\Phi)-\bar x_{t+1}(\Phi')\bigr).
\]
The final update is never installed and creates no switching cost, so we set
$\Lambda_{T,V}=\Lambda_T$.  Hence $\Lambda_t\le\Lambda_{t,V}$ for every
$t\le T$, with equality at $t=T$.

The learner is projected subgradient descent in the weighted geometry.  Let
$\operatorname{proj}_{\cC}^{\beta}(Z)$ be a nearest point to $Z$ in $\cC$
under \eqref{eq:weighted-response-norm}, and let $D\bar y_t[Z]$ be the
derivative of \eqref{eq:response-reference} with respect to $\Phi$ in a
direction $Z\in\mathsf H_{\cC}$.  For a general convex class $\cC$, the
projection is an oracle assumption; its cost depends on the constraints
that define the class.  Taking
$D_\Omega(\Phi,\Phi')=\tfrac12\norm{\Phi-\Phi'}_{\beta,H,\circ}^2$ in
\eqref{eq:response-mirror-update} gives projected subgradient descent after
rescaling the coefficients by the weights (\cref{alg:response-omd}).

For $H=T-1$, the reference needs a state at time $T+1$ to close the error
recursion.  Define the unpenalized terminal state
\begin{equation}
 \bar x_{T+1}^{\mathrm{term}}(\Phi)
 =A\bar x_T(\Phi)+B\bar u_T(\Phi)+w_T.
 \label{eq:response-terminal-extension}
\end{equation}
The final update is never installed, so this convention gives
$\delta_{T+1}=0$ without one more response coefficient.

\begin{algorithm}[H]
\caption{Weighted projected subgradient descent for response learning}
\label{alg:response-omd}
\small
\textbf{Input:} horizons $T,H$; nonempty closed convex response class
$\cC\subseteq\cS_H$; weight $\beta$; learning rate $\eta>0$; initial response
$\Phi_1\in\cC$.
\begin{algsteps}
\item \textbf{Initialize:} return $\Phi_1$.
\item \textbf{Update after round $t$:} receive
$h_t(\Phi)=c_t(\bar y_t(\Phi))$ and:\par
  \begin{algsteps}
  \item Choose $s_t\in\partial c_t(\bar y_t(\Phi_t))$.
  \item Set
  $g_t\gets\bigl(D\bar y_t|_{\mathsf H_{\cC}}\bigr)^*s_t$, where the
  adjoint uses \eqref{eq:weighted-response-norm}.
  \item Update
  $\displaystyle
   \Phi_{t+1}\gets
   \operatorname{proj}_{\cC}^{\beta}(\Phi_t-\eta g_t)$.
  \item Return $\Phi_{t+1}$.
  \end{algsteps}
\end{algsteps}
\end{algorithm}

As for finite classes, the bound combines the learner's regret with
the recovery cost of its updates.

\begin{corollary}[Weighted regret over a full-horizon response class]
\label{thm:response-omd}
Let $T\ge2$ and $H=T-1$, and suppose the costs are convex and globally
$L$-Lipschitz in the Euclidean state--input norm, and fix a recovery pair
$(V,\psi)$.  Then counterfactual tracking from $e_1=0$ with this pair and
with weighted projected subgradient descent (\cref{alg:response-omd}) as the
base learner and any learning rate $\eta>0$ satisfies, for every
$\Phi\in\cC$,
\begin{equation}
 \sum_{t=1}^T
 \bigl[c_t(y_t)-c_t(\bar y_t(\Phi))\bigr]
 \le
 \frac{R^2}{2\eta}
 +\frac{3\eta L^2}{2}\sum_{t=1}^T\Lambda_t\Lambda_{t,V}.
 \label{eq:exact-response-bound}
\end{equation}
\end{corollary}

The learning rate must be set before play, so it cannot use the realized
sensitivity sum.  For $L,R>0$, let $\mathsf S_T>0$ be a known deterministic
bound on $\sum_{t=1}^T\Lambda_t\Lambda_{t,V}$ that is valid for every cost
and disturbance sequence allowed by the model.  Choosing
\[
 \eta=\frac{R}{L\sqrt{3\mathsf S_T}}
\]
yields
\begin{equation}
 \Reg_T(\cC)\le\sqrt{3}\,LR\sqrt{\mathsf S_T}.
 \label{eq:exact-response-tuned}
\end{equation}
If the sensitivity sum is zero for every allowed sequence, keeping any
feasible response fixed gives zero regret.

A uniform bound $\mathsf S_T$ follows from the Euclidean norm.  The proof
below shows
\begin{equation}
 \Lambda_t\le WS_{\beta,H},
 \qquad
 \Lambda_{t,V}\le(1+M_V)WS_{\beta,H}.
 \label{eq:response-scalarization}
\end{equation}
Tuning the learning rate with $M_V$, or with any known upper bound on it,
gives the following bound.

\begin{corollary}[Weighted response ball with a tuned step size]
\label{cor:response-ball}
Under the conditions of \cref{thm:response-omd}, counterfactual tracking
with weighted projected subgradient descent (\cref{alg:response-omd}) as the
base learner and $\eta=R/(LWS_{\beta,H}\sqrt{3(1+M_V)T})$ satisfies, with a
universal constant $C$,
\begin{equation}
 \Reg_T(\cC)
 \le
 C L R W S_{\beta,H}\sqrt{(1+M_V)T}.
 \label{eq:response-ball-upper}
\end{equation}
\end{corollary}

\begin{proof}[Proof of \cref{thm:response-omd}]
We first check feasibility, then apply \cref{thm:convex-reference-omd}.
For a response prefix satisfying \eqref{eq:response-prefix} and
$1\le t<T$,
\begin{align*}
 A\bar x_t(\Phi)+B\bar u_t(\Phi)+w_t
 &=w_t+\sum_{i=1}^{t-1}
   (A\Phi_x[i]+B\Phi_u[i])w_{t-i}\\
 &=w_t+\sum_{i=1}^{t-1}\Phi_x[i+1]w_{t-i}
 =\bar x_{t+1}(\Phi),
\end{align*}
because every displayed block lies in the prefix of length $H=T-1$.  At
$t=T$, the same left side is the terminal extension
\eqref{eq:response-terminal-extension}.  Thus the references used by all $T$
costs are feasible, and every installed update $t<T$ has the next state
$\bar x_{t+1}(\Phi)$ used in the switching bound.

For the loss $h_t(\Phi)=c_t(\bar y_t(\Phi))$, the chain rule on
$\mathsf H_{\cC}$ gives
\[
 \norm{g_t}_{\beta,H,\circ,*}\le L\Lambda_t.
\]
Moreover,
\[
 V\bigl(\bar x_{t+1}(\Phi)-\bar x_{t+1}(\Phi')\bigr)
 \le
 \norm{\bar y_t(\Phi)-\bar y_t(\Phi')}_V,
\]
so the recovery sensitivity \eqref{eq:reference-recovery-sensitivity} is at
most $\Lambda_{t,V}$.  Half the squared weighted norm is $1$-strongly convex
with respect to that norm, its Bregman divergence from the initial point is
at most $R^2/2$ on $\cC$, and the Jensen gain is zero.  Hence
\cref{thm:convex-reference-omd} bounds the left side of
\eqref{eq:exact-response-bound} by
\[
 \frac{R^2}{2\eta}
 +\frac{\eta L^2}{2}\sum_{t=1}^T\Lambda_t^2
 +\eta L^2\sum_{t=1}^{T-1}\Lambda_t\Lambda_{t,V}.
\]
Since $\Lambda_t\le\Lambda_{t,V}$ for $t<T$ and
$\Lambda_{T,V}=\Lambda_T$, this is at most the right side of
\eqref{eq:exact-response-bound}.
\end{proof}

\begin{proof}[Proof of \cref{cor:response-ball}]
Weighted Cauchy--Schwarz gives, for every $t\le T$,
\[
 \norm{\bar y_t(\Phi)-\bar y_t(\Phi')}_2
 \le WS_{\beta,H}\norm{\Phi-\Phi'}_{\beta,H,\circ},
\]
and, for $t<T$,
\[
 \norm{\bar x_{t+1}(\Phi)-\bar x_{t+1}(\Phi')}_2
 \le WS_{\beta,H}\norm{\Phi-\Phi'}_{\beta,H,\circ}.
\]
By \cref{lem:relative-reachability},
$\bar x_{t+1}(\Phi)-\bar x_{t+1}(\Phi')\in\Reach$, so
$V(z)\le M_V\norm{z}_2$ applies and proves the bound on $\Lambda_{t,V}$ for
$t<T$.  At $t=T$, it follows from $\Lambda_{T,V}=\Lambda_T$ and
$1+M_V\ge1$.  This proves \eqref{eq:response-scalarization}, so
$\mathsf S_T=(1+M_V)W^2S_{\beta,H}^2T$ is a valid bound.  Substituting it in
\eqref{eq:exact-response-tuned} gives \eqref{eq:response-ball-upper}.
\end{proof}

\subsection{The full \texorpdfstring{$\ell_1$}{l1} response ball}
\label{app:near-l1-response-proof}

Which benchmark policies does the full response ball contain?  When the
baseline is absolutely summable, every internally stabilizing finite-dimensional linear policy does, for some finite
radius, and every rational response in the ball is realized by such a
policy.

\begin{proposition}[Finite-dimensional linear policies in the full response ball]
\label{prop:full-response-scope}
Every zero-initialized finite-dimensional linear time-invariant dynamic
state-feedback policy that internally stabilizes the system has an
absolutely summable feasible response.  If the baseline $\Phi^0$ is
absolutely summable, this response therefore belongs to
$\mathfrak S^{(1)}(R)$ for some finite $R$.  Conversely, every absolutely
summable feasible response whose transfer matrices are real-rational admits
a finite-dimensional internally stabilizing SLS realization.
\end{proposition}

\begin{proof}[Proof of \cref{prop:full-response-scope}]
Write a finite-dimensional dynamic policy as
\[
 s_{t+1}=A_cs_t+B_cx_t,
 \qquad
 u_t=C_cs_t+D_cx_t,
 \qquad s_1=0,
\]
and define
\[
 F_\pi=
 \begin{bmatrix}A+BD_c&BC_c\\ B_c&A_c\end{bmatrix},
 \quad
 E=\begin{bmatrix}I\\0\end{bmatrix},
 \quad
 C_x=\begin{bmatrix}I&0\end{bmatrix},
 \quad
 C_u=\begin{bmatrix}D_c&C_c\end{bmatrix}.
\]
Its response coefficients are
\[
 \Phi_{\pi,x}[i]=C_xF_\pi^{i-1}E,
 \qquad
 \Phi_{\pi,u}[i]=C_uF_\pi^{i-1}E.
\]
Internal stability makes $F_\pi$ Schur (its spectral radius is below one),
so both sequences are absolutely summable.  The response is feasible:
$\Phi_{\pi,x}[1]=C_xE=I$ and $C_xF_\pi=AC_x+BC_u$, so
$\Phi_{\pi,x}[i+1]=A\Phi_{\pi,x}[i]+B\Phi_{\pi,u}[i]$.  Its difference from
the absolutely summable baseline has finite summed Frobenius norm, which
gives the radius $R$.

Conversely, suppose an absolutely summable feasible response has
real-rational transfer matrices.  Its strictly proper stable transfer matrices
$\boldsymbol\Phi_x(z)=\sum_{i\ge1}\Phi_x[i]z^{-i}$ and
$\boldsymbol\Phi_u(z)=\sum_{i\ge1}\Phi_u[i]z^{-i}$ satisfy
\[
 (zI-A)\boldsymbol\Phi_x(z)-B\boldsymbol\Phi_u(z)=I.
\]
Absolute summability and real-rationality place both transfer matrices in
$z^{-1}\mathcal{RH}_\infty$, the class of strictly proper real-rational stable
transfer matrices.  The SLS realization theorem for full-state feedback
\cite[Theorem~1]{wang2019sls} therefore gives a finite-dimensional internally
stabilizing realization.
\end{proof}

Nonrational responses need no realization.  The regret theorem uses only the
causal convolution \eqref{eq:response-reference-infinite} and the affine feasibility relations; it needs no internal stability.

\paragraph{Relation to peak disturbance amplification.}
\label{app:response-peak-gain}
Classical $\ell_1$ control limits the largest output coordinate under
coordinatewise bounded disturbances \cite{dahleh1987l1}; see also
\cite{anderson2019sls}.  How does this peak gain compare with the radius of
the response ball?  For a centered response $\Delta\Phi=\Phi-\Phi^0$, put
$m=d_x+d_u$ and $n=d_x$, and define
\[
 Q(\Delta\Phi)=\sum_{i\ge1}\norm{\Delta\Phi[i]}_{\mathrm F},
 \qquad
 G_\infty(\Delta\Phi)
 =\max_{1\le a\le m}\sum_{i\ge1}\sum_{b=1}^n
       \abs{\Delta\Phi_{ab}[i]}.
\]
For an absolutely summable response, $G_\infty$ is the induced
$\ell_\infty$-to-$\ell_\infty$ norm on sequences:
\[
 G_\infty(\Delta\Phi)
 =\sup_{\sup_s\norm{w_s}_\infty\le1}\sup_{t\ge1}
   \left\|\sum_{i=1}^{t-1}\Delta\Phi[i]w_{t-i}\right\|_\infty.
\]
The upper bound follows from the triangle inequality.  For the reverse
bound, fix a row and a finite response prefix, and choose the signs of the
disturbance coordinates to match that row at a selected time.  Letting the
prefix length grow recovers the absolute row sum.  Maximizing over the
finitely many rows proves the equality.

The two norms satisfy
\begin{equation*}
 G_\infty(\Delta\Phi)\le\sqrt n\,Q(\Delta\Phi),
 \qquad
 Q(\Delta\Phi)\le mG_\infty(\Delta\Phi).
\end{equation*}
For the first inequality, each row of $\Delta\Phi[i]$ has $\ell_1$ norm at
most $\sqrt n$ times its Euclidean norm, hence at most
$\sqrt n\norm{\Delta\Phi[i]}_{\mathrm F}$.  For the second, the Frobenius
norm is at most the sum of the absolute entries; summing over ages and then
over the $m$ rows gives the bound.  The constants depend only on dimensions.

In particular, responses with $G_\infty(\Phi-\Phi^0)\le b$ lie in
$\mathfrak S^{(1)}(mb)$.  Running the learner with radius $mb$ therefore
gives the regret bound of \cref{thm:l1-response-upper} against this
peak-gain-bounded class, under the same cost and disturbance assumptions.
The comparison concerns benchmark classes only.  It does not identify the
two norms, and it does not constrain the peak gain of the online-control
algorithm.

Can the ball be learned through a common finite memory?  No: its mass can
sit entirely beyond any cutoff.

\begin{proposition}[No common finite-memory approximation]
\label{prop:no-common-response-tail}
Consider the scalar system $x_{t+1}=u_t+w_t$ and its zero-input baseline
$\Phi_x^0[1]=1$, $\Phi_x^0[i]=0$ for $i\ge2$, and
$\Phi_u^0[i]=0$ for $i\ge1$.  For every $R>0$ and cutoff $h\ge1$,
\begin{equation*}
 \sup_{\Phi\in\mathfrak S^{(1)}(R)}
 \sum_{i>h}\norm{\Phi[i]-\Phi^0[i]}_{\mathrm F}=R.
\end{equation*}
The supremum is attained by a finite-memory policy whose closed loop is
nilpotent.
\end{proposition}

\begin{proof}[Proof of \cref{prop:no-common-response-tail}]
The upper bound is the definition of the ball.  For the reverse inequality,
choose $j=h+1$ and the finite-memory policy $u_t=(R/2)w_{t-j}$.  Its only
nonzero centered coefficients are $\Phi_u[j]=R/2$ and $\Phi_x[j+1]=R/2$, so
the entire centered mass $R$ lies after the cutoff.  On this system
$w_r=x_{r+1}-u_r$, so a policy that stores the last $j$ reconstructed
disturbances implements $u_t$ causally.  Its closed loop is nilpotent, hence
internally stable.
\end{proof}

\paragraph{Slow decay with bounded total response.}
The tail need not come from a response concentrated at one age.  On
$x_{t+1}=u_t+w_t$, keep the zero-input baseline and fix $a>0$.
For $0<\gamma<1$, let
\[
 u_t^\gamma=\sum_{j=1}^{t-1}a\gamma(1-\gamma)^{j-1}w_{t-j}.
\]
The input coefficients sum to $a$, and the state response adds their
one-round shift.  Hence
\[
 \sum_{i\ge1}\norm{\Phi^\gamma[i]-\Phi^0[i]}_{\mathrm F}\le2a,
 \qquad
 \sum_{j>h}a\gamma(1-\gamma)^{j-1}=a(1-\gamma)^h.
\]
Taking $a=R/2$ places every member in the fixed-radius ball.  For each
$\gamma>0$, the input mass at ages greater than $h$ tends to zero as
$h\to\infty$, but for every fixed $h$ its supremum over $\gamma$ is $a$.
Small coefficients offset slow decay, so the radius does not grow.  In
contrast, the state-feedback policy $u_t=(1-\gamma)x_t$ has state-response
coefficients $(1-\gamma)^{i-1}$, whose sum is $1/\gamma$.  As $\gamma\to0$,
these policies leave every fixed-radius ball.  The ball thus limits total
response.  Individual stability alone does not keep a class in a
fixed-radius ball, so it does not yield the theorem's bound.

\paragraph{The near-$\ell_1$ regularizer.}
The ball is an $\ell_1$ sum of $H$ coefficient norms.  A Euclidean
regularizer would pay a factor $\sqrt H$ in this geometry
(\cref{prop:hilbert-separation}).  We therefore use a mixed norm with
exponent just above one, which stays within a constant factor of the
$\ell_1$ sum while making the regularizer strongly convex.  Set
\begin{equation}
 p=1+\frac1{\log(eH)},
 \qquad
 \norm{Z}_{2,p}
 =\left(\sum_{i=1}^H\norm{Z[i]}_{\mathrm F}^p\right)^{1/p},
 \qquad
 \Omega(Z)=\frac12\norm{Z}_{2,p}^2.
 \label{eq:near-l1-regularizer}
\end{equation}
On the centered feasible set
\[
 \mathcal Z_H(R)
 =\{Z:\Phi^0+Z\in\cS_H^{(1)}(R)\},
\]
write
\[
 D_\Omega(Z,Z')
 =\Omega(Z)-\Omega(Z')-\ip{\nabla\Omega(Z')}{Z-Z'}.
\]

\begin{lemma}[Strong convexity of the mixed response norm]
\label{lem:mixed-response-strong-convexity}
For $1<p\le2$, the function $\Omega(Z)=\tfrac12\norm{Z}_{2,p}^2$ is
$(p-1)$-strongly convex with respect to $\norm{\cdot}_{2,p}$.
\end{lemma}

\begin{proof}
H\"older's inequality gives a uniform lower bound on the Hessian.
To allow zero blocks, fix $\varepsilon>0$, put
$r_i=(\norm{Z[i]}_{\mathrm F}^2+\varepsilon^2)^{1/2}$ and $S=\sum_i r_i^p$, and
differentiate the smooth function $\Omega_\varepsilon(Z)=S^{2/p}/2$.
For a direction $U$, the nonnegative square term from the outer derivative
and $\ip{Z[i]}{U[i]}^2\le r_i^2\norm{U[i]}_{\mathrm F}^2$ give
\[
 D^2\Omega_\varepsilon(Z)[U,U]
 \ge(p-1)S^{(2-p)/p}
          \sum_i r_i^{p-2}\norm{U[i]}_{\mathrm F}^2
 \ge(p-1)\norm{U}_{2,p}^2.
\]
The last inequality is H\"older's inequality with exponents $2/p$ and
$2/(2-p)$; for $p=2$ it is an equality.  Integrating along a line segment
proves the strong-convexity inequality for $\Omega_\varepsilon$.  Letting
$\varepsilon\downarrow0$ proves it for $\Omega$.
\end{proof}

\begin{algorithm}[H]
\caption{Near-$\ell_1$ response mirror-descent base learner}
\label{alg:near-l1-response}
\small
\textbf{Input:} horizons $T,H$; baseline $\Phi^0$ and radius $R$; learning
rate $\eta>0$; regularizer \eqref{eq:near-l1-regularizer}.
\begin{algsteps}
\item \textbf{Initialize:} set $Z_1\gets0$ and return
$\Phi_1\gets\Phi^0$.
\item \textbf{Update after round $t$:} receive
$h_t(\Phi)=c_t(\bar y_t(\Phi))$ and:\par
  \begin{algsteps}
  \item Choose $s_t\in\partial c_t(\bar y_t(\Phi_t))$ and define the loss
  functional $g_t$ on feasible response directions by
  $\ip{g_t}{Z}=\ip{s_t}{D\bar y_t[Z]}$, where
  $D\bar y_t[Z]=\sum_{i=1}^{\min\{H,t-1\}}Z[i]w_{t-i}$ is the linear part of
  \eqref{eq:response-reference}.
  \item Set
  \[
   Z_{t+1}\in\argmin_{Z\in\mathcal Z_H(R)}
   \left\{\eta\ip{g_t}{Z}+D_\Omega(Z,Z_t)\right\}.
  \]
  \item Return $\Phi_{t+1}\gets\Phi^0+Z_{t+1}$.
  \end{algsteps}
\end{algsteps}
\end{algorithm}

\begin{theorem}[Learning the full response ball around a baseline]
\label{thm:l1-response-upper}
Under the conventions of \cref{sec:response-learning-update}, let
$H=T-1$, fix a baseline $\Phi^0\in\cS_\infty$, and let
$R\ge0$.  Suppose the costs are convex and globally $L$-Lipschitz, and fix a
recovery pair $(V,\psi)$.  Counterfactual tracking with this pair and with
near-$\ell_1$ mirror descent (\cref{alg:near-l1-response}) as the base learner
and step size \eqref{eq:l1-response-step}, or with the baseline held fixed if
$LWR=0$, satisfies
\begin{equation*}
 \Reg_T\bigl(\mathfrak S^{(1)}(R)\bigr)
 \le\Reg_T\bigl(\cS_H^{(1)}(R)\bigr)
 \le eLWR\sqrt{(1+2M_V)T\log(eH)}.
\end{equation*}
\end{theorem}

With $(V,\psi)=(\Vstar,\psi_\star)$, the bound holds with $M_V=M_\star$.  A
computable recovery pair gives an implementable algorithm.

\begin{proof}[Proof of \cref{thm:l1-response-upper}]
The proof reduces to the finite prefix class and applies
\cref{thm:convex-reference-omd} in the norm $\norm{\cdot}_{2,p}$.

Every comparator in $\mathfrak S^{(1)}(R)$ has its first $H$ coefficients in
$\cS_H^{(1)}(R)$, and these coefficients determine all $T$ comparator costs.
Hence $\Reg_T(\mathfrak S^{(1)}(R))\le\Reg_T(\cS_H^{(1)}(R))$.

By \cref{lem:mixed-response-strong-convexity}, the regularizer is
$(p-1)$-strongly convex with respect to $\norm{\cdot}_{2,p}$, so
\cref{thm:convex-reference-omd} applies with $\sigma=p-1$.  The dual
exponent is $q=p/(p-1)=1+\log(eH)$, and $H^{1/q}\le e$.

\emph{Loss sensitivity.}  For every $t\le T$, the loss functional has blocks
$g_t[i]=s_tw_{t-i}^\top$ with $\norm{s_t}_2\le L$, so
$\norm{g_t}_{2,q}\le H^{1/q}\max_i\norm{s_tw_{t-i}^\top}_{\mathrm F}\le eLW$.
Thus $G_t\le eLW$.

\emph{Recovery sensitivity.}  For every installed switch $t<T$,
$\bar x_{t+1}(\Phi)-\bar x_{t+1}(\Phi')=\sum_{i=2}^{\min\{H,t\}}(\Phi_x[i]-\Phi'_x[i])w_{t+1-i}$
has Euclidean norm at most $W\norm{Z-Z'}_{2,1}\le eW\norm{Z-Z'}_{2,p}$.
Since $V(z)\le M_V\norm{z}_2$ on $\Reach$, the recovery sensitivity
\eqref{eq:reference-recovery-sensitivity} is at most $eM_VW$.

\emph{Diameter.}  Since $Z_1=0$ and $\nabla\Omega(0)=0$,
$D_\Omega(Z,Z_1)=\tfrac12\norm{Z}_{2,p}^2\le\tfrac12\norm{Z}_{2,1}^2\le R^2/2$
on $\mathcal Z_H(R)$.  The Jensen gain is zero.

\Cref{thm:convex-reference-omd} therefore gives
\[
 \Reg_T(\cS_H^{(1)}(R))
 \le
 \frac{R^2}{2\eta}
 +\frac{\eta e^2L^2W^2(1+2M_V)T}{2(p-1)}.
\]
Substitute \eqref{eq:l1-response-step} and use
$(p-1)^{-1}=\log(eH)$.
\end{proof}

\subsection{Quadratic regularizers lose a factor \texorpdfstring{$\sqrt H$}{sqrt H}}

Why not learn the full response ball with a quadratic regularizer?  On the
delayed policies of the one-step lower bound, every Hilbert
radius--dual-radius bound pays a factor $\sqrt H$, whereas the near-$\ell_1$
geometry pays $\sqrt{\log H}$.
Let $\norm{\theta}_{\mathsf M}=(\theta^\top\mathsf M\theta)^{1/2}$, with
$\mathsf M\succ0$, be any Hilbert norm on the response coefficients
$\theta\in\R^H$ of the policies \eqref{eq:delayed-policies} in
\cref{thm:response-lower}.  For the $\ell_1$ decision ball
$\{\norm{\theta}_1\le R/2\}$, define
\[
 r_{\mathsf M}:=\sup_{\norm{\theta}_1\le R/2}\norm{\theta}_{\mathsf M},
 \qquad
 \mathcal G_H:=\{LW\sigma:\sigma\in\{-1,1\}^H\},
 \qquad
 g_{\mathsf M}:=\sup_{g\in\mathcal G_H}\norm{g}_{\mathsf M^{-1}}.
\]
The sign gradients in $\mathcal G_H$ occur on this linear subclass.

\begin{proposition}[Quadratic regularizers lose a factor $\sqrt H$]
\label{prop:hilbert-separation}
For this subclass with $H$ disturbance ages,
\begin{equation}
 r_{\mathsf M}g_{\mathsf M}\ge cLWR\sqrt H
 \label{eq:hilbert-product}
\end{equation}
for every $\mathsf M\succ0$, where $c>0$ is a universal constant.  Hence a
Hilbert radius--dual-radius guarantee of order $r_{\mathsf M}g_{\mathsf M}\sqrt T$
cannot certify better than $\Omega(LWR\sqrt{HT})$.  On the one-step system,
where $V=\Vstar$ and $M_V=M_\star=1$, \cref{thm:l1-response-upper} gives
$O(LWR\sqrt{T\log(eH)})$.  More generally, if $\Omega$ is
$\alpha$-strongly convex in $\norm{\cdot}_{\mathsf M}$ and has scalar range
$\Delta_\Omega$ on the decision ball, then
\begin{equation}
 \sqrt{\Delta_\Omega/\alpha}\,g_{\mathsf M}
 \ge cLWR\sqrt H.
 \label{eq:kdk-regularizer-product}
\end{equation}
\end{proposition}

\begin{proof}
The $\ell_1$ ball contains its signed coordinate extremes, so
$r_{\mathsf M}\ge(R/2)\max_j\sqrt{\mathsf M_{jj}}$.  Rademacher gradient signs
give $g_{\mathsf M}\ge cLW\sqrt{\tr(\mathsf M^{-1})}$.  Since
$\mathsf M_{jj}(\mathsf M^{-1})_{jj}\ge1$,
\[
 \max_j \mathsf M_{jj}\,\tr(\mathsf M^{-1})\ge H,
\]
which proves \eqref{eq:hilbert-product}.
For the general regularizer, take $a=R/2$ and a coordinate $j$ maximizing
$\mathsf M_{jj}$.  Strong convexity at the midpoint of $ae_j$ and $-ae_j$ gives
\[
 \frac{\Omega(ae_j)+\Omega(-ae_j)}2-\Omega(0)
 \ge\frac{\alpha a^2\mathsf M_{jj}}2.
\]
The left side is at most $\Delta_\Omega$.  Combining this inequality with
$g_{\mathsf M}^2\ge cL^2W^2\tr(\mathsf M^{-1})$ proves
\eqref{eq:kdk-regularizer-product}.
\end{proof}

The proposition covers a quadratic regularizer, as in infinite-memory DAC
\cite{kumar2023}, applied to our full $\ell_1$ response ball with
unrestricted disturbance ages.  It lower-bounds the standard
radius--dual-radius bound in that geometry.  It is not a regret lower bound
for infinite-memory DAC on its own geometrically enveloped decision class.
Mirror descent in general nonquadratic norms is not covered.

\subsection{Stable linear state feedback with convex costs}
\label{app:stable-response-proof}

The full-horizon learner of \cref{cor:stable-feedback} is
\eqref{eq:response-mirror-update} with
$D_\Omega(\Phi,\Phi')=\tfrac12
\norm{\Phi-\Phi'}_{\beta,H,\circ}^2$ and projection onto
$\cC_{\kappa,\gamma,H}$.  By \cref{thm:response-omd}, the step size of \cref{cor:stable-feedback}
gives the stated rate once the comparator radius is bounded, which we do
below.  For the basic affine weighted ball, reference evaluation and
projection take $O(T)$ arithmetic operations per round at fixed dimensions,
after $O(T)$ preprocessing (\cref{prop:response-runtime}).  Additional
response constraints require their own projection analysis.  If $LW=0$,
the learner holds any feasible response fixed.

Policy decay and physical recovery enter separately.  The margin $\gamma$
controls how quickly the benchmark policies forget past disturbances.  The
margin $\gamma_{\mathrm{cor}}$ of a corrective gain controls how quickly the
physical state returns to the chosen reference.  The optimal recovery norm
minimizes $M_V$.  Bounding $M_\star$ by the recovery cost of any policy in
$\cK_{\kappa,\gamma}$ gives the existence bound
$O_\kappa(LW\sqrt{d_xT}/\gamma^{3/2})$
(\cref{cor:stable-feedback-uniform}), but computing an optimal corrective law
may require an optimization oracle.  A known
$(\kappa_{\mathrm{cor}},\gamma_{\mathrm{cor}})$-stable corrective gain
instead gives, by \eqref{eq:linear-gamma-certificate},
\begin{equation*}
 \Reg_T(\cK_{\kappa,\gamma})
 =O_{\kappa,\kappa_{\mathrm{cor}}}\!\left(
 \frac{LW\sqrt{d_xT}}{\gamma\sqrt{\gamma_{\mathrm{cor}}}}\right).
\end{equation*}

\paragraph{Locally bounded gradients.}
Quadratic costs are not globally Lipschitz, so we must control the region
that both the reference and the physical trajectory visit.  A peak tracking
estimate, together with the small movement of projected descent, keeps both
in a ball on which the gradients are bounded.

Suppose $\cK_{\kappa,\gamma}$ is nonempty and a known corrective gain is
$(\kappa_{\mathrm{cor}},\gamma_{\mathrm{cor}})$-stable on the reachable
subspace.  Let $D_0=O_\kappa(W\sqrt{d_x}/\gamma)$ be the explicit radius in
\cref{cor:regular-response}.  If every cost is convex and differentiable on
the ball of radius $2D_0$, with $\norm{\nabla c_t(y)}_2\le2GD_0$ there, the
actual and reference trajectories remain in that ball and
\begin{equation}
 \Reg_T(\cK_{\kappa,\gamma})
 =O_{\kappa,\kappa_{\mathrm{cor}}}\!\left(
 \frac{GW^2d_x\sqrt T}{\gamma^2\sqrt{\gamma_{\mathrm{cor}}}}\right).
 \label{eq:regular-stable-rate}
\end{equation}
This is the local-gradient normalization of \cref{tab:stable-comparison}.
The Lipschitz scale $GD_0$ carries a factor $\gamma^{-1}$, so the
$\gamma^{-1}$ in \eqref{eq:stable-feedback-rate} becomes $\gamma^{-2}$.
\Cref{cor:regular-response,app:regular-response-proof} give the exact tuning,
including the short-horizon branch, and the invariant-region argument.

\paragraph{The comparator radius.}
Recall $\beta^2=1-\gamma$ and
\begin{equation*}
 R_{\kappa,\gamma}^2
 =d_x\kappa^2+
 \frac{d_x\kappa^4(1+\kappa^2)}{\gamma}.
\end{equation*}
The proof below shows that the closed convex response ball
\begin{equation}
 \cC_{\kappa,\gamma,H}
 =\left\{\Phi\in\cS_H:
   \norm{\Phi}_{\beta,H,\circ}\le R_{\kappa,\gamma}\right\}
 \label{eq:stable-response-ball}
\end{equation}
contains every prefix of length $H$ of a comparator in
$\cK_{\kappa,\gamma}$.  It has diameter at most $2R_{\kappa,\gamma}$, and
$S_{\beta,H}\le\gamma^{-1/2}$.

With the optimal recovery norm, $M_\star$ is bounded through the comparator
class itself.

\begin{corollary}[Stable linear state feedback with the optimal recovery norm]
\label{cor:stable-feedback-uniform}
Under the assumptions of \cref{cor:stable-feedback}, use the optimal recovery
cost $\Vstar$ and its minimizing corrective law.  Then counterfactual
tracking with the base learner and step size of \cref{cor:stable-feedback}
satisfies
\begin{align}
 \Reg_T(\cK_{\kappa,\gamma})
 &\le
 C\frac{LW\kappa^2\sqrt{d_x(1+\kappa^2)}}{\gamma}
 \sqrt{\left(1+
 \frac{\kappa^2\sqrt{1+\kappa^2}}{\gamma}\right)T}
 \notag\\
 &=O_\kappa\!\left(
   \frac{LW\sqrt{d_xT}}{\gamma^{3/2}}
   \right).
 \label{eq:stable-feedback-uniform}
\end{align}
\end{corollary}

\begin{proof}[Proof of \cref{cor:stable-feedback}]
Let $F_K=A-BK$.  The response \eqref{eq:feedback-response} satisfies
the affine recursion \eqref{eq:response-prefix}.  Under
\eqref{eq:strong-stability} and $\beta^2=1-\gamma$,
\begin{align*}
 \norm{\Phi^K}_{\beta,H,\circ}^2
 &\le\norm{K}_{\mathrm F}^2
 +\sum_{i=2}^H
   \beta^{-2(i-1)}
   \bigl(\norm{F_K^{i-1}}_{\mathrm F}^2+
         \norm{KF_K^{i-1}}_{\mathrm F}^2\bigr)\\
 &\le d_x\kappa^2+
 d_x\kappa^4(1+\kappa^2)
 \sum_{i=2}^H(1-\gamma)^{i-1}\\
 &\le d_x\kappa^2+
 \frac{d_x\kappa^4(1+\kappa^2)}{\gamma}.
\end{align*}
The ball $\cC_{\kappa,\gamma,H}$ therefore contains these prefixes and has
diameter at most $2R_{\kappa,\gamma}$.  By
\eqref{eq:response-scalarization} and $S_{\beta,H}\le\gamma^{-1/2}$,
$\mathsf S_T=(1+M_V)W^2T/\gamma$ is a valid sensitivity bound, and the step
size of \cref{cor:stable-feedback} is $\eta=R/(L\sqrt{3\mathsf S_T})$ with
$R=2R_{\kappa,\gamma}$.  Substitute these in \eqref{eq:exact-response-tuned},
together with $R_{\kappa,\gamma}^2\le2d_x\kappa^4(1+\kappa^2)/\gamma$, which
holds because $\kappa\ge1>\gamma$.  The universal constant in
\cref{cor:stable-feedback} absorbs the numerical factors.
\end{proof}

\begin{proof}[Proof of \cref{cor:stable-feedback-uniform}]
Use $V=\Vstar$ in \cref{cor:stable-feedback} and substitute the bound
\eqref{eq:comparator-bounds-correction} on $M_\star$.  This gives the first
line of \eqref{eq:stable-feedback-uniform}, and collecting powers of
$\gamma$ gives the second.
\end{proof}

\subsection{Computation for weighted responses}
\label{app:response-computation}

What does a round of the weighted learners cost?  The number of response
coefficients alone does not decide it: each round must evaluate the
changing reference and its gradient and compute a projection.  For the basic
weighted balls, all three operations cost time linear in the number of
disturbance ages.  We count real arithmetic operations at fixed $d_x,d_u$,
with unit-cost access to a cost gradient; these are not bit-complexity
bounds.  The corrective gain is known in advance.

\begin{proposition}[Time per round for weighted responses]
\label{prop:response-runtime}
For the untruncated affine weighted ball
$\cC_{\kappa,\gamma,H}$ and the terminally augmented DAC ball
$\cC_H^{\mathrm{DAC}}(R)$, the projected updates in
\cref{cor:stable-feedback,thm:truncated-dac} have implementations using
$O(H)$ arithmetic operations and one cost-gradient evaluation per round,
after $O(H)$ preprocessing, at fixed dimensions.  These counts include
reference evaluation, the weighted projection, and the physical action under
a known linear corrective gain.  Storage is $O(H)$ at fixed dimensions.
The same implementations apply to the local-gradient variants.
\end{proposition}

\begin{proof}
We reduce the projection to an affine projection followed by a radial
step, and then show how to evaluate the reference and its gradient.

\emph{Projection.}
Stack the free response entries in a vector $z$, omitting the fixed block
$\Phi_x[1]=I$.  For truncated DAC, include the terminal block $Z_\Phi$ as an
additional variable and impose its defining affine equality.  In these
coordinates, either decision set has the form
\[
 \mathcal C=\{z:\mathsf Ez=b,\quad z^\top\mathsf D z\le R^2\},
\]
where $\mathsf D$ is positive diagonal and contains the age weights, so its
quadratic form is the one used by the corresponding mirror map.  Each row
block of $\mathsf E$ couples consecutive disturbance ages.  The affine
constraints have full row rank, because each response recursion contains its
next-state block with an identity coefficient.

Let $z_\circ$ minimize $z^\top\mathsf D z$ on $\mathsf Ez=b$.  Precompute
$z_\circ$ and a factorization of
$\mathsf E\mathsf D^{-1}\mathsf E^\top$.  This matrix is positive definite
and block tridiagonal, with blocks of fixed size when $d_x,d_u$ are fixed,
so block elimination takes $O(H)$ operations and storage.  If there are no
constraints, the affine projection below is the identity.

For a gradient step with target $a$, first compute its weighted affine
projection
\[
 z_{\mathrm{aff}}
 =a-\mathsf D^{-1}\mathsf E^\top
   (\mathsf E\mathsf D^{-1}\mathsf E^\top)^{-1}(\mathsf Ea-b),
\]
which costs $O(H)$ with the stored factorization.  Since $z_\circ$ is
$\mathsf D$-orthogonal to the affine direction space, every $z$ with
$\mathsf Ez=b$ satisfies
\[
 z^\top\mathsf D z
 =z_\circ^\top\mathsf D z_\circ+\norm{z-z_\circ}_{\mathsf D}^2,
 \qquad \norm{v}_{\mathsf D}^2=v^\top\mathsf D v.
\]
On the affine space, $\mathcal C$ is thus a $\mathsf D$-ball about $z_\circ$
of radius $r=(R^2-z_\circ^\top\mathsf D z_\circ)^{1/2}$, where
$r^2\ge0$ because $\mathcal C$ is nonempty.  By Pythagoras, the projection of
$a$ onto $\mathcal C$ is the projection of $z_{\mathrm{aff}}$ onto this
ball:
\[
 z^+=z_\circ+
 \min\!\left\{1,\frac{r}{\norm{z_{\mathrm{aff}}-z_\circ}_{\mathsf D}}\right\}
 (z_{\mathrm{aff}}-z_\circ),
\]
with $z^+=z_\circ$ when the denominator vanishes.  This step costs $O(H)$.
For an ambient loss-gradient vector $g$, the target
$a=z_t-\eta\mathsf D^{-1}g$ makes $z^+$ exactly the weighted mirror-descent
update.  No iterative quadratic-program solver is needed for these sets.

\emph{Reference and gradient.}
For $H=T-1$, evaluating
$\sum_{i\le\min\{H,t-1\}}\Phi[i]w_{t-i}$ and applying the transpose of this
linear map to a cost gradient both take $O(H)$ operations at fixed
dimensions.  A linear corrective input costs $O(1)$.

For truncated DAC, the response beyond age $H$ must also be evaluated at the
\emph{current} coefficients.  Maintain the linear map
\[
 \mathcal T_t(Z)=\sum_{j\ge0}F_0^jZ w_{t-H-1-j},
 \qquad
 \mathcal T_{t+1}(Z)=F_0\mathcal T_t(Z)+Zw_{t-H},
\]
with zero disturbances at nonpositive indices, by storing its values on a
fixed matrix basis, whose size is independent of $H$ and $T$.
Then
\[
 \bar y_t(\widehat\Phi)
 =\sum_{i=1}^{\min\{H,t-1\}}\Phi[i]w_{t-i}
   +\begin{bmatrix}I\\-K_0\end{bmatrix}\mathcal T_t(Z_\Phi).
\]
Updating the map, evaluating it at $Z_{\Phi_t}$, and applying its transpose
to a gradient cost $O(1)$ at fixed dimensions.  Because the whole linear map is maintained, this identity remains valid when $\Phi_t$ changes.  The prefix sums and the DAC action cost $O(H)$, and a
buffer of the $H+1$ most recent disturbances uses $O(H)$ storage.  Combining these operations
with the projection proves the claim.
\end{proof}

A response with $H$ blocks has $Hd_ud_x$ independent scalar parameters and
can be stored directly in $O(Hd_x(d_x+d_u))$ entries.  The proposition
covers only the basic affine quadratic balls.  Additional locality,
sparsity, or other convex constraints need their own projection analysis,
and the projection formula does not cover the near-$\ell_1$ mirror map of
Application~II.  An optimal or MPC corrective law likewise has its own
evaluation cost.

\paragraph{Prior DAC time entries.}
The truncated-DAC update \cite{agarwal2019online} uses
$\widetilde O(\gamma^{-1})$ blocks.  Its fixed-gain, finite-window state
maps and their transposes can be maintained blockwise.  Its Frobenius
projections onto the individual spectral-norm balls clip singular values.
At fixed dimensions, its arithmetic cost is therefore linear in the cutoff.
For example, each finite-window sensitivity map has the form
\[
 \mathcal J_{t,j}(U)=\sum_{r=0}^{m}F_0^rBUw_{t-r-j}.
\]
Its next value follows by subtracting the term with $r=m$, multiplying by $F_0$,
and adding $BUw_{t+1-j}$.  Storing these maps on a matrix basis updates all
coefficient blocks in linear time.

For the infinite-memory DAC method of online convex optimization with
unbounded memory \cite{kumar2023}, evaluating the chosen action uses at most
$t$ disturbance-response terms on round $t$, hence $O(T)$ amortized
arithmetic operations.  Choosing its coefficients, however, requires
minimizing the cumulative nonlinear FTRL objective.  We denote this
additional amortized cost by $\mathsf C_{\mathrm{FTRL}}$, including all
evaluations of past losses used by the optimization.  The implementation
discussion of that method \cite[Appendix~G]{kumar2023} bounds only the oracle work for evaluating the objective.
\Cref{tab:stable-comparison} therefore does not infer a linear-time FTRL
update from the number of active response coefficients.

\subsection{Response learning with locally bounded gradients}
\label{app:regular-response-proof}

With gradients bounded only on a ball, the response learner must keep the
physical trajectory in that ball.  Small steps keep the tracking error below
the reference radius; the same argument gives
\cref{cor:spectral-local-gradients}.

\begin{corollary}[Response learning with locally bounded gradients]
\label{cor:regular-response}
Let $H=T-1$ and let $\cC\subseteq\cS_H$ be the response class of diameter
$R$ used in \cref{thm:response-omd}.  Let $K_{\mathrm{cor}}$ be a corrective
state-feedback gain for which every eigenvalue of
$(A-BK_{\mathrm{cor}})|_{\Reach}$ lies strictly inside the unit disk,
and let $P_{\mathrm{cor}}=P_{K_{\mathrm{cor}}}$ be defined by
\eqref{eq:peak-tracking-gain}.  Suppose every cost is convex and
differentiable on the Euclidean ball of radius $2D_0$, with
$\norm{\nabla c_t(y)}_2\le 2GD_0$ throughout that ball.  Suppose also that every
reference from $\cC$ has Euclidean norm at most $D_0$ and that
$RWS_{\beta,H}\le D_0$.  Run counterfactual tracking with the corrective
law $e\mapsto-K_{\mathrm{cor}}e$ and the following base learner.  It starts
from any $\Phi_1\in\cC$.  For $T\ge P_{\mathrm{cor}}$, it runs weighted
projected subgradient descent (\cref{alg:response-omd}) with
\begin{equation}
 \eta=\frac{R}
 {2GD_0WS_{\beta,H}\sqrt{(1+P_{\mathrm{cor}})T}}.
 \label{eq:regular-step}
\end{equation}
For $T<P_{\mathrm{cor}}$, it keeps $\Phi_t=\Phi_1$.  Zero scales use the
fixed-reference convention at the start of this appendix.  The resulting
algorithm keeps the realized trajectory in the $2D_0$ ball and satisfies,
for a universal constant $C$,
\begin{equation*}
 \Reg_T(\cC)
 \le
 C G D_0^2
 \sqrt{(1+P_{\mathrm{cor}})T}.
\end{equation*}

Assume $\cK_{\kappa,\gamma}$ is nonempty.  For
$\cC=\cC_{\kappa,\gamma,H}$ in \eqref{eq:stable-response-ball}, take
$R=2R_{\kappa,\gamma}$ and
$D_0=2R_{\kappa,\gamma}WS_{\beta,H}$; then
$D_0=O_\kappa(W\sqrt{d_x}/\gamma)$.  Every reference from
$\cC_{\kappa,\gamma,H}$ has Euclidean norm at most
$W+WS_{\beta,H}R_{\kappa,\gamma}\le D_0$, where the first term comes from the
fixed coefficient $\Phi_x[1]=I$ and the inequality uses $S_{\beta,H}\ge1$ and
$R_{\kappa,\gamma}\ge1$.  If $K_{\mathrm{cor}}$ satisfies the
$(\kappa_{\mathrm{cor}},\gamma_{\mathrm{cor}})$ decay bound preceding
\eqref{eq:linear-gamma-certificate}, this gives the rate in
\eqref{eq:regular-stable-rate}.
\end{corollary}

\begin{proof}[Proof of \cref{cor:regular-response}]
The proof combines the per-round movement of projected descent with a
convolution bound for the corrective gain.  We treat the long-horizon regime
first and then the short-horizon fallback.

\emph{Long horizons.}  Suppose $T\ge P_{\mathrm{cor}}$ and use
\eqref{eq:regular-step}.  As long as the actual and reference trajectories
remain in the $2D_0$ ball, the gradient bound gives a local Lipschitz
constant at most $2GD_0$.  The movement inequality then implies
\[
 \sup_{t<T}\norm{
 \bar x_{t+1}(\Phi_t)-\bar x_{t+1}(\Phi_{t+1})}_2
 \le
 2\eta GD_0(WS_{\beta,H})^2
 \le
 \frac{RWS_{\beta,H}}
 {\sqrt{(1+P_{\mathrm{cor}})T}}.
\]
Convolution with the linear correction response gives
\begin{align*}
 \sup_{t\le T}\norm{(e_t,-K_{\mathrm{cor}}e_t)}_2
 &\le
 P_{\mathrm{cor}}
 \sup_{s<T}\norm{
 \bar x_{s+1}(\Phi_s)-\bar x_{s+1}(\Phi_{s+1})}_2\\
 &\le D_0
 \frac{P_{\mathrm{cor}}}
 {\sqrt{(1+P_{\mathrm{cor}})T}}
 \le D_0.
\end{align*}
Hence the actual path stays in the radius-$2D_0$ ball.  Substituting the
local Lipschitz constant into the online-mirror-descent and movement bounds
yields
\[
 \Reg_T(\cC)
 \le
 C GD_0(RWS_{\beta,H})
 \sqrt{(1+P_{\mathrm{cor}})T}
 \le
 C GD_0^2
 \sqrt{(1+P_{\mathrm{cor}})T}.
\]

\emph{Short horizons.}  If $T<P_{\mathrm{cor}}$, keep one feasible response
fixed.  There is no switch, and comparing two points in the radius-$D_0$
reference region costs at most $CGD_0^2$ per round.  Hence
\[
 \Reg_T(\cC)
 \le CGD_0^2T
 \le CGD_0^2
 \sqrt{P_{\mathrm{cor}}T}.
\]

\emph{Stable response ball.}  Finally,
$D_0\le C_{\kappa}W\sqrt{d_x}/\gamma$ and
$P_{\mathrm{cor}}
=O_{\kappa_{\mathrm{cor}}}(\gamma_{\mathrm{cor}}^{-1})$, which gives
\eqref{eq:regular-stable-rate}.
\end{proof}

\subsection{Comparison with infinite-memory DAC}
\label{app:kdk-comparison}

Where does the $\gamma^{-9/2}$ in the infinite-memory DAC row of
\cref{tab:stable-comparison} come from?  We trace it through the published
parameter calculation \cite[Theorem~4.1 and Appendix~E]{kumar2023}, writing
$\gamma=1-\rho$ in its notation.  Let $d$ denote its dimension
parameter and $G$ its local-gradient constant.  The control specialization
instantiates a generic follow-the-regularized-leader theorem for unbounded
memory with regularizer range $D_{\mathrm{reg}}$, strong-convexity parameter
$\alpha$, uniform trajectory radius $D_X$, history-loss Lipschitz constant
$L_{\mathrm{hist}}$, frozen-loss Lipschitz constant $\widetilde L$, and
effective-memory capacity $H_2$, which satisfy
\begin{align*}
 \frac{D_{\mathrm{reg}}}{\alpha}
 &=O\!\left(d\kappa^8\gamma^{-1}\right),
 &
 D_X
 &=O\!\left(W\kappa^8\gamma^{-2}\right),
\\
 L_{\mathrm{hist}}
 &=O\!\left(GW^2\kappa^9\gamma^{-3}\right),
 &
 \widetilde L
 &=O\!\left(GW^2\kappa^{13}\gamma^{-7/2}\right),
\\
 H_2
 &=O\!\left(\kappa^4\gamma^{-3/2}\right).
\end{align*}
The generic upper bound has the schematic form
\begin{equation*}
 \frac{D_{\mathrm{reg}}}{\eta}
 +\frac{\eta T}{\alpha}\widetilde L^2
 +\frac{\eta T}{\alpha}
   L_{\mathrm{hist}}\widetilde L H_2.
\end{equation*}
Balancing the first term against the memory term gives the published
control dependence:
\begin{align*}
 \sqrt{
   \frac{D_{\mathrm{reg}}}{\alpha}\,
   T L_{\mathrm{hist}}\widetilde L H_2}
 &=O\!\left(
   GW^2\sqrt{dT}\,
   \kappa^{17}\gamma^{-9/2}
   \right).
\end{align*}

Without the memory term, the same calculation already gives
\begin{equation*}
 \sqrt{\frac{D_{\mathrm{reg}}}{\alpha}}\,
 \widetilde L\sqrt T
 =O\!\left(
   GW^2\sqrt{dT}\,
   \kappa^{17}\gamma^{-4}
   \right).
\end{equation*}
The ratio of the two displayed scales is
\begin{equation*}
 \sqrt{\frac{(\kappa^9\gamma^{-3})
                    (\kappa^4\gamma^{-3/2})}
                   {\kappa^{13}\gamma^{-7/2}}}
 =\gamma^{-1/2}.
\end{equation*}
The powers of $\kappa$ cancel.  In the published calculation, the memory
term thus contributes only the final factor $\gamma^{-1/2}$, and the full
$\kappa^{17}$ power is already present in the static scale.  This compares only the stated bounds.  In particular, the exponent is not implied by the lower
bound for arbitrary online convex optimization with memory.

The construction shows where the static factor comes from.  The DAC response
class used as an outer approximation is bounded uniformly over a class much
larger than the linear state-feedback comparators.  Its trajectory radius
$D_X$ already contributes $\kappa^8\gamma^{-2}$ under the local-gradient
normalization, and the generic theorem then multiplies separately bounded
regularizer and loss-sensitivity scales.  The two uses of $\kappa$ also
differ.  In the infinite-memory analysis, $\kappa$ controls the
system matrices, the baseline and comparator gains, similarity transforms,
transient amplification, and the DAC coefficient envelope.  Our $\kappa$
describes only the comparator class, while $V$ or
$(\kappa_{\mathrm{cor}},\gamma_{\mathrm{cor}})$ quantifies the recovery cost.

A DAC algorithm already tracks a reference.  If $K_0$ is its baseline gain
and $v_t(M_t)$ its disturbance-dependent input, let $\bar x_t(M_t)$ be the
frozen-parameter counterfactual state and
$\bar u_t(M_t)=-K_0\bar x_t(M_t)+v_t(M_t)$.  Then
\begin{equation*}
 u_t=-K_0x_t+v_t(M_t)
 =\bar u_t(M_t)-K_0\bigl(x_t-\bar x_t(M_t)\bigr).
\end{equation*}
The difference therefore lies elsewhere.  The infinite-memory analysis
absorbs correction into global Lipschitz and effective-memory quantities.
Counterfactual tracking exposes the switching cost
$V(\bar x_{t+1}(q)-\bar x_{t+1}(q'))$ and the corrective law that pays it
as separate objects.

\subsection{Truncated DAC}
\label{app:truncated-dac-proof}

Truncated DAC learns only the coefficients at the first $H$ disturbance ages and continues the
response beyond them under a fixed stabilizing gain.  Two questions arise:
is the completed class still feasible, so that the reduction applies, and
how can the learner control the continuation it does not learn?  We answer
the first by construction and the second by adding a terminal term to the
weighted norm.

Fix a known gain $K_0$ stabilizing the full system, write $F_0=A-BK_0$,
and define the peak response sum
\begin{equation*}
 \mathsf P_0=\sum_{j\ge0}
 \left\|\begin{bmatrix}I\\-K_0\end{bmatrix}F_0^j\right\|_{\mathrm{op}}.
\end{equation*}
Given a feasible response prefix $\Phi\in\cS_H$ of length $H$, define its
terminal state response and its continuation under $K_0$ by
\begin{align*}
 Z_\Phi
 &=A\Phi_x[H]+B\Phi_u[H],
\\
 \widehat\Phi[i]&=\Phi[i] \quad (i\le H),
 \\[-0.25em]
 \widehat\Phi[H+1+j]
 &=
 \begin{bmatrix}I\\-K_0\end{bmatrix}F_0^jZ_\Phi
 \quad (j\ge0).
\end{align*}
The coefficients at the first $H$ ages are learned, and at larger ages the response follows $K_0$.  These
completions form an affine family of feasible responses.  We also use
$-K_0e_t$ as the corrective input.  The completion and the correction then
combine into the DAC action \eqref{eq:literal-truncated-dac}.

Put $\beta^2=1-\gamma$.  The continuation is determined by $Z_\Phi$, so the
norm charges it at the weight of age $H+1$.  On a feasible response
direction $\Delta\Phi$, use the terminally augmented norm
\begin{align*}
 \norm{\Delta\Phi}_{\beta,H,+}^2
 ={}&\norm{\Delta\Phi_u[1]}_{\mathrm F}^2
 +\sum_{i=2}^H\beta^{-2(i-1)}
 \left(
  \norm{\Delta\Phi_x[i]}_{\mathrm F}^2+
  \norm{\Delta\Phi_u[i]}_{\mathrm F}^2
 \right)
 \\
 &+\beta^{-2H}\norm{Z_{\Delta\Phi}}_{\mathrm F}^2.
\end{align*}
For a response prefix itself, let $Q_{\beta,H,+}(\Phi)$ denote the same
quadratic expression with $Z_\Phi$ in the last term, and define
\begin{equation}
 \cC_H^{\mathrm{DAC}}(R)
 =\{\Phi\in\cS_H:Q_{\beta,H,+}(\Phi)\le R^2\}.
 \label{eq:truncated-dac-response-class}
\end{equation}
Finally set
\begin{equation}
 \zeta_{H,T}
 =S_{\beta,H}
  +\mathbf 1_{\{H<T-1\}}\beta^H\mathsf P_0.
 \label{eq:dac-sensitivity-scale}
\end{equation}
The indicator removes the continuation term when $H=T-1$: no coefficient
beyond age $T-1$ affects a cost or an installed reference change.

\begin{theorem}[Weighted learning over a truncated DAC response class]
\label{thm:truncated-dac}
Let $T\ge2$ and $1\le H\le T-1$, suppose the costs are convex and globally
$L$-Lipschitz, and assume that $\cC_H^{\mathrm{DAC}}(R)$ is nonempty.  Give
each decision $\Phi\in\cC_H^{\mathrm{DAC}}(R)$ the reference
$\bar y_t(\widehat\Phi)$ and the loss
\(
 h_t(\Phi)=c_t(\bar y_t(\widehat\Phi))
\).
Counterfactual tracking with the corrective law $e\mapsto-K_0e$ and with
mirror descent \eqref{eq:reference-mirror-descent} as the base learner,
started from any $\Phi_1\in\cC_H^{\mathrm{DAC}}(R)$ and run with
$D_\Omega(\Phi,\Phi')=\frac12\norm{\Phi-\Phi'}_{\beta,H,+}^2$ and the
constant step size
\begin{equation}
 \eta
 =\frac{2R}
 {LW\zeta_{H,T}\sqrt{(1+2\mathsf P_0)T}}
 \label{eq:truncated-dac-step-size}
\end{equation}
satisfies, for every $\Phi\in\cC_H^{\mathrm{DAC}}(R)$,
\begin{equation}
 \sum_{t=1}^T
 \bigl[c_t(y_t)-c_t(\bar y_t(\widehat\Phi))\bigr]
 \le
 2LWR\zeta_{H,T}\sqrt{(1+2\mathsf P_0)T}.
 \label{eq:truncated-dac-class-bound}
\end{equation}
Moreover, with
\begin{equation}
 M_{\Phi_t}^{[i]}
 =\Phi_{u,t}[i]+K_0\Phi_{x,t}[i],
 \qquad 1\le i\le H,
 \label{eq:response-to-dac-map}
\end{equation}
the action applied to the physical system is exactly
\eqref{eq:literal-truncated-dac}, and the online decision has exactly
$Hd_ud_x$ free parameters.
\end{theorem}

Computing each input uses $H$ matrices, and the weighted projection has an
$O(H)$ implementation at fixed dimensions after preprocessing
(\cref{prop:response-runtime}).  The terminal term lets $H$ scale with the
comparator's decay time $\gamma^{-1}$ and only logarithmically with the
horizon (\cref{cor:compact-dac-main}).

\begin{proof}[Proof of \cref{thm:truncated-dac}]
We verify feasibility, bound the two sensitivities, apply
\cref{thm:convex-reference-omd}, and identify the implemented action.

\emph{Feasibility.}  The recursion \eqref{eq:infinite-response-space} holds
through age $H$ because $\Phi\in\cS_H$, and at the next age by the
definition of $Z_\Phi$.  Thereafter,
\[
 A\widehat\Phi_x[H+1+j]+B\widehat\Phi_u[H+1+j]
 =F_0^{j+1}Z_\Phi
 =\widehat\Phi_x[H+2+j].
\]
Thus
\begin{equation*}
 \bar y_t(\widehat\Phi)
 =\sum_{i=1}^{t-1}\widehat\Phi[i]w_{t-i}
\end{equation*}
is an exactly feasible counterfactual reference.

\emph{Sensitivities.}  Let $\Phi,\Phi'\in\cC_H^{\mathrm{DAC}}(R)$.  Weighted
Cauchy--Schwarz on the first $H$ ages gives
\[
 \left\|
  \sum_{i=1}^{\min\{H,t-1\}}
  (\Phi[i]-\Phi'[i])w_{t-i}
 \right\|_2
 \le
 WS_{\beta,H}\norm{\Phi-\Phi'}_{\beta,H,+}.
\]
For the terminal state response,
\(
 \norm{Z_\Phi-Z_{\Phi'}}_{\mathrm F}
 \le\beta^H\norm{\Phi-\Phi'}_{\beta,H,+}
\).
Summing the continuation under $K_0$ therefore yields
\begin{equation*}
 \norm{\bar y_t(\widehat\Phi)
       -\bar y_t(\widehat\Phi')}_2
 \le
 W\zeta_{H,T}\norm{\Phi-\Phi'}_{\beta,H,+}.
\end{equation*}
The same calculation for the next state gives, for every installed update
$t<T$,
\begin{equation}
 \norm{\bar x_{t+1}(\widehat\Phi)-\bar x_{t+1}(\widehat\Phi')}_2
 \le
 W\zeta_{H,T}\norm{\Phi-\Phi'}_{\beta,H,+}.
 \label{eq:completed-endpoint-sensitivity}
\end{equation}
When $H=T-1$, neither calculation involves a coefficient beyond age $H$,
which explains the indicator in \eqref{eq:dac-sensitivity-scale}.

\emph{Regret.}  Since $K_0$ stabilizes the full state,
\cref{prop:fixed-feedback-certificate} applies on the whole state space:
$(V_{K_0},e\mapsto-K_0e)$ is a recovery pair, and
\[
 V_{K_0}(z)
 =\sum_{j\ge0}
 \left\|
 \begin{bmatrix}I\\-K_0\end{bmatrix}F_0^jz
 \right\|_2
 \le\mathsf P_0\norm{z}_2.
\]
Hence, in \cref{thm:convex-reference-omd}, the loss sensitivity is at most
$LW\zeta_{H,T}$, the recovery sensitivity is at most
$\mathsf P_0W\zeta_{H,T}$, and the Jensen gain is zero.  The regularizer
$Q_{\beta,H,+}/2$ has Bregman divergence
\[
 D_Q(\Phi,\Phi')
 =\frac12\norm{\Phi-\Phi'}_{\beta,H,+}^2,
\]
because the constant part of the affine map $\Phi\mapsto Z_\Phi$ cancels.
The class \eqref{eq:truncated-dac-response-class} is convex with diameter at
most $2R$, so $D_Q\le2R^2$.  \Cref{thm:convex-reference-omd} now gives
\[
 \frac{2R^2}{\eta}
 +\frac{\eta L^2W^2\zeta_{H,T}^2T}{2}
 +\eta L^2\mathsf P_0W^2\zeta_{H,T}^2T,
\]
and the step size \eqref{eq:truncated-dac-step-size} minimizes this bound,
proving \eqref{eq:truncated-dac-class-bound}.

\emph{Implementation.}  By construction,
\[
 \widehat\Phi_u[i]+K_0\widehat\Phi_x[i]=0
 \qquad(i>H).
\]
Counterfactual tracking therefore applies
\begin{align*}
 u_t
 &=\bar u_t(\widehat\Phi_t)
   -K_0\bigl(x_t-\bar x_t(\widehat\Phi_t)\bigr)\\
 &=-K_0x_t+
   \sum_{i=1}^{\min\{H,t-1\}}
   \bigl(\Phi_{u,t}[i]+K_0\Phi_{x,t}[i]\bigr)w_{t-i},
\end{align*}
which is \eqref{eq:literal-truncated-dac}.  Conversely, given the $H$ DAC
coefficient matrices, the recursions
\[
 \Phi_x[1]=I,
 \qquad
 \Phi_u[i]=M^{[i]}-K_0\Phi_x[i],
 \qquad
 \Phi_x[i+1]=F_0\Phi_x[i]+BM^{[i]}
\]
recover a unique feasible prefix.  Thus \eqref{eq:response-to-dac-map} is an
affine bijection, and the free dimension is $Hd_ud_x$.
\end{proof}

\paragraph{Stable state-feedback comparators.}
We now take the comparators from $\cK_{\kappa,\gamma}$.  We say that $K_0$
is \emph{$(\kappa_0,\gamma_0)$-stable} on the full state if it satisfies
\eqref{eq:strong-stability} with $(\kappa_0,\gamma_0)$ in place of
$(\kappa,\gamma)$.  Assume that $\cK_{\kappa,\gamma}$ is nonempty, and define
\begin{equation*}
 R_{\kappa,\gamma,+}^2
 =d_x\kappa^2+
 \frac{d_x\kappa^4(1+\kappa^2)}{\gamma}
 +d_x\kappa^4.
\end{equation*}
Every prefix of length $H$ of a comparator in $\cK_{\kappa,\gamma}$ belongs
to $\cC_H^{\mathrm{DAC}}(R_{\kappa,\gamma,+})$.  Completing that prefix with
$K_0$ changes its state--input point on any round by at most $E_H$, where
\begin{equation}
 \overline E_H
 =W\kappa^2(1-\gamma)^H
 \left(
  \mathsf P_0+\frac{\sqrt{1+\kappa^2}}{\gamma}
 \right),
 \qquad
 E_H=\mathbf 1_{\{H<T-1\}}\overline E_H.
 \label{eq:dac-completion-error}
\end{equation}
Consequently, \cref{thm:truncated-dac} gives
\begin{equation}
 \Reg_T(\cK_{\kappa,\gamma})
 \le
 2LWR_{\kappa,\gamma,+}\zeta_{H,T}
 \sqrt{(1+2\mathsf P_0)T}
 +LTE_H.
 \label{eq:stable-truncated-dac-bound}
\end{equation}
Choosing $H$ so that both terms are small gives the following summary.

\begin{corollary}[Truncated DAC learning for stable state feedback]
\label{cor:compact-dac-main}
Suppose $\cK_{\kappa,\gamma}$ is nonempty and the costs are convex and
globally $L$-Lipschitz.  For every cumulative-cost tolerance
$\varepsilon>0$, there is
$H_\varepsilon=\widetilde O_\kappa(\gamma^{-1})$ with
$\beta^{H_\varepsilon}\mathsf P_0\le S_{\beta,H_\varepsilon}$ and
$LT\overline E_{H_\varepsilon}\le\varepsilon$.  Let $H=\min\{T-1,H_\varepsilon\}$.  Counterfactual tracking with the
corrective law, base learner, and step size of \cref{thm:truncated-dac},
with $R=R_{\kappa,\gamma,+}$, plays the DAC action
\eqref{eq:literal-truncated-dac} and satisfies
\begin{equation}
 \Reg_T(\cK_{\kappa,\gamma})
 =O_\kappa\!\left(\frac{LW\sqrt{d_x}}{\gamma}
 \sqrt{(1+\mathsf P_0)T}\right)+\varepsilon.
 \label{eq:stable-truncated-dac-rate}
\end{equation}
The hidden logarithm depends on $T,L,W,\mathsf P_0,1/\gamma$, and
$1/\varepsilon$.  The online decision has $Hd_ud_x$ scalar parameters.  At
fixed dimensions, reference evaluation, the weighted projection, and the
physical action take $O(H)$ arithmetic operations per round after $O(H)$
preprocessing (\cref{prop:response-runtime}).  If $K_0$ is
$(\kappa_0,\gamma_0)$-stable on the full state, then
$\mathsf P_0=O_{\kappa_0}(\gamma_0^{-1})$, and
\begin{equation}
 \Reg_T(\cK_{\kappa,\gamma})
 =O_{\kappa,\kappa_0}\!\left(
 \frac{LW\sqrt{d_xT}}{\gamma\sqrt{\gamma_0}}\right)+\varepsilon.
 \label{eq:stable-truncated-dac-two-margin}
\end{equation}
\end{corollary}

\begin{proof}[Proof of \cref{cor:compact-dac-main}]
We first verify the comparator radius, \eqref{eq:dac-completion-error}, and
\eqref{eq:stable-truncated-dac-bound}.  Let $F_K=A-BK$.  For the prefix of \eqref{eq:feedback-response},
$Z_{\Phi^K}=F_K^H$.  Since $\beta^2=1-\gamma$,
\begin{align*}
 Q_{\beta,H,+}(\Phi^K)
 &\le d_x\kappa^2
 +d_x\kappa^4(1+\kappa^2)
   \sum_{i=2}^H(1-\gamma)^{i-1}
 +d_x\kappa^4(1-\gamma)^H
 \\
 &\le R_{\kappa,\gamma,+}^2.
\end{align*}
The completed response agrees with the response of $K$ through age $H$.
Over the remaining ages, the norms of the completed coefficient pairs sum
to at most
\[
 \mathsf P_0\norm{F_K^H}_{\mathrm{op}}
 \le\kappa^2\mathsf P_0(1-\gamma)^H,
\]
and those of $K$ sum to at most
\[
 \sqrt{1+\kappa^2}
 \sum_{j\ge0}\norm{F_K^{H+j}}_{\mathrm{op}}
 \le
 \frac{\kappa^2\sqrt{1+\kappa^2}}{\gamma}
 (1-\gamma)^H.
\]
Multiplying by $W$ and applying the triangle inequality proves
\eqref{eq:dac-completion-error}.  Global Lipschitz continuity, followed by
\eqref{eq:truncated-dac-class-bound}, proves
\eqref{eq:stable-truncated-dac-bound}.

For the corollary, if $H<T-1$, the two conditions on $H_\varepsilon$ give
$\zeta_{H,T}\le2S_{\beta,H}$ and $LTE_H\le\varepsilon$; if $H=T-1$, neither
is needed.  Since
\[
 R_{\kappa,\gamma,+}
 =O_\kappa\!\left(\sqrt{d_x/\gamma}\right),
 \qquad
 S_{\beta,H}\le\gamma^{-1/2},
\]
we obtain \eqref{eq:stable-truncated-dac-rate}.  The prefix length is
$\widetilde O_\kappa(\gamma^{-1})$ because
$\beta^H\le e^{-\gamma H/2}$ and
$(1-\gamma)^H\le e^{-\gamma H}$.  Finally, if $K_0$ is
$(\kappa_0,\gamma_0)$-stable on the full state, then
\[
 \mathsf P_0
 \le
 \frac{\kappa_0^2\sqrt{1+\kappa_0^2}}{\gamma_0},
\]
which proves \eqref{eq:stable-truncated-dac-two-margin}.
\end{proof}

The same construction gives the normalization commonly used for quadratic
costs, in which a gradient bound on a ball replaces global Lipschitzness.

\begin{corollary}[Truncated DAC with locally bounded gradients]
\label{cor:regular-truncated-dac}
Assume $\cK_{\kappa,\gamma}$ is nonempty.
Put
\(
 D_0=4W\left(
 R_{\kappa,\gamma,+}/\sqrt\gamma
 +\kappa^2\sqrt{1+\kappa^2}/\gamma
 \right)
\).
Let $H_{\mathrm{loc}}$ be the least positive integer such that
\begin{equation*}
 \beta^{H_{\mathrm{loc}}}\mathsf P_0\le1,
 \qquad
 2GD_0T\overline E_{H_{\mathrm{loc}}}\le1,
\end{equation*}
and take $H=\min\{T-1,H_{\mathrm{loc}}\}$; then
$H_{\mathrm{loc}}=\widetilde O_\kappa(\gamma^{-1})$.
Suppose every cost is convex and differentiable on the Euclidean ball of
radius $2D_0$, with
$\norm{\nabla c_t(y)}_2\le2GD_0$ throughout that ball.  Run
counterfactual tracking with the corrective law $e\mapsto-K_0e$ and the
following base learner on $\cC_H^{\mathrm{DAC}}(R_{\kappa,\gamma,+})$.  It
starts from any $\Phi_1$ in this set.  If $T\ge\mathsf P_0$, it runs mirror
descent \eqref{eq:reference-mirror-descent} with the losses and Bregman
divergence of \cref{thm:truncated-dac} and step size
\begin{equation}
 \eta=\frac{R_{\kappa,\gamma,+}}
 {GD_0W\zeta_{H,T}\sqrt{(1+2\mathsf P_0)T}}.
 \label{eq:local-dac-step}
\end{equation}
If $T<\mathsf P_0$, it keeps $\Phi_t=\Phi_1$ throughout.  Zero scales follow
the fixed-reference convention at the start of this appendix.  The
resulting DAC algorithm \eqref{eq:literal-truncated-dac}, with prefix length
$H$, keeps the realized trajectory in the ball of radius $2D_0$ and
satisfies, for a universal constant $C$,
\begin{equation}
 \Reg_T(\cK_{\kappa,\gamma})
 \le C G D_0^2\sqrt{(1+\mathsf P_0)T}+1.
 \label{eq:regular-truncated-dac-abstract}
\end{equation}
If $K_0$ is $(\kappa_0,\gamma_0)$-stable on the full state, this gives
\begin{equation}
 \Reg_T(\cK_{\kappa,\gamma})
 =O_{\kappa,\kappa_0}\!\left(
 \frac{GW^2d_x\sqrt T}{\gamma^2\sqrt{\gamma_0}}\right)+1.
 \label{eq:regular-truncated-dac-rate}
\end{equation}
\end{corollary}

This construction needs a known gain that stabilizes the full state, since
the same gain completes the response beyond age $H$ and corrects the
tracking error.  The full-horizon guarantee of \cref{cor:stable-feedback}
allows a separately chosen corrective law.

\begin{proof}[Proof of \cref{cor:regular-truncated-dac}]
We show that the reference and the physical trajectory stay in the ball of
radius $2D_0$ and then apply \cref{thm:truncated-dac} with the local
Lipschitz constant $2GD_0$.

If $H<T-1$, then
$\beta^H\mathsf P_0\le1\le S_{\beta,H}$ and
$2GD_0TE_H\le1$; if $H=T-1$, the continuation and approximation terms vanish.
Hence $\zeta_{H,T}\le2/\sqrt\gamma$.  For any prefix
in $\cC_H^{\mathrm{DAC}}(R_{\kappa,\gamma,+})$, weighted
Cauchy--Schwarz and the bound on the continuation under $K_0$ give
\[
 \sup_{t\le T}\norm{\bar y_t(\widehat\Phi)}_2
 \le W\bigl(1+R_{\kappa,\gamma,+}\zeta_{H,T}\bigr)
 \le D_0,
\]
where the term $1$ accounts for the fixed coefficient $\Phi_x[1]=I$.  Every
comparator also satisfies
\[
 \sup_{t\le T}\norm{y_t^K}_2
 \le
 \frac{W\kappa^2\sqrt{1+\kappa^2}}{\gamma}
 \le D_0.
\]

If $T\ge\mathsf P_0$, the step size \eqref{eq:local-dac-step} is
\eqref{eq:truncated-dac-step-size} with $L=2GD_0$.
Projection and \eqref{eq:completed-endpoint-sensitivity} give
\[
 \sup_{t<T}\norm{\bar x_{t+1}(\widehat\Phi_t)
                  -\bar x_{t+1}(\widehat\Phi_{t+1})}_2
 \le\frac{2WR_{\kappa,\gamma,+}\zeta_{H,T}}
 {\sqrt{(1+2\mathsf P_0)T}}.
\]
Since $2WR_{\kappa,\gamma,+}\zeta_{H,T}\le D_0$, convolving these switch
discrepancies with the response of the $K_0$ correction bounds the physical
tracking deviation by
$D_0\mathsf P_0/\sqrt{(1+2\mathsf P_0)T}\le D_0$, using $T\ge\mathsf P_0$.
Hence the physical trajectory stays in the ball of radius $2D_0$, and
\eqref{eq:truncated-dac-class-bound} bounds the class regret by
\[
 C GD_0\bigl(WR_{\kappa,\gamma,+}\zeta_{H,T}\bigr)
 \sqrt{(1+\mathsf P_0)T}
 \le
 C GD_0^2\sqrt{(1+\mathsf P_0)T}.
\]
If $T<\mathsf P_0$, the fixed feasible prefix has no tracking error.
Its regret to any completed comparator is at most
$4GD_0^2T\le4GD_0^2\sqrt{\mathsf P_0T}$, so the same bound holds.

The completed comparator and the exact comparator both lie in the local
region, so their cumulative costs differ by at most
$2GD_0TE_H\le1$.  This proves
\eqref{eq:regular-truncated-dac-abstract}.  Substituting the bounds on $D_0$
and $\mathsf P_0$ proves \eqref{eq:regular-truncated-dac-rate}.
\end{proof}

\subsection{Spectral filters: construction and statements}
\label{app:spectral-construction}

The full-response learner needs no decay assumption, but its decision
dimension grows with the horizon.  Can a narrower class be learned with few
parameters?  For stable state feedback with real nonnegative closed-loop
spectra, spectral filters compress the long disturbance response.  The
representation needs a stable baseline gain to turn filtered inputs into
feasible references.  The corrective law need not use the same gain.

Fix a \emph{spectral baseline} $K_0$ and constants
$\kappa_0\ge1$ and $\gamma_0\in(0,1)$, write $F_0=A-BK_0$, and assume
\begin{equation}
 \opnorm{K_0}\le\kappa_0,
 \qquad
 \opnorm{F_0^r}
 \le\kappa_0^2(1-\gamma_0)^r
 \quad(r\ge0).
 \label{eq:spectral-baseline}
\end{equation}
The baseline enters the bounds through
\begin{equation}
 \Lambda_{\mathrm{sp},B}
 =\sum_{r\ge0}\opnorm{F_0^rB},
 \qquad
 C_{\mathrm{sp}}^2
 =\Lambda_{\mathrm{sp},B}^2+
 \bigl(1+\opnorm{K_0}\Lambda_{\mathrm{sp},B}\bigr)^2.
 \label{eq:spectral-baseline-scales}
\end{equation}
For $\chi\ge1$, let
$\cK^{\mathrm{diag}}_{\kappa,\gamma,\chi}$ contain the gains
$K\in\cK_{\kappa,\gamma}$ for which
\begin{equation}
 A-BK
 =S_K\operatorname{diag}(\alpha_1,\ldots,\alpha_{d_x})S_K^{-1},
 \qquad
 0\le\alpha_i\le1-\gamma,
 \qquad
 \opnorm{S_K}\opnorm{S_K^{-1}}\le\chi.
 \label{eq:spectral-comparator-class}
\end{equation}
The real nonnegative spectrum and the uniform conditioning of the eigenbasis
genuinely restrict the class.  They are what let one Hankel basis serve
every gain in it.  Assume $\cK^{\mathrm{diag}}_{\kappa,\gamma,\chi}$ is
nonempty, and set
\begin{equation*}
 R_{\mathrm{sp}}
 =\sup_{K\in\cK^{\mathrm{diag}}_{\kappa,\gamma,\chi}}
   \fnorm{K-K_0},
 \qquad
 R_{\mathrm{sp},\mathrm{op}}
 =\sup_{K\in\cK^{\mathrm{diag}}_{\kappa,\gamma,\chi}}
   \opnorm{K-K_0}.
\end{equation*}
With $r_K=\min\{d_x,d_u\}$,
\begin{equation*}
 R_{\mathrm{sp}}
 \le\sqrt{r_K}(\kappa+\kappa_0),
 \qquad
 R_{\mathrm{sp},\mathrm{op}}
 \le\kappa+\kappa_0.
\end{equation*}

A scalar closed-loop eigenvalue $\alpha$ generates the temporal pattern
$\mu_\alpha=(1,\alpha,\ldots,\alpha^{m-1})^\top$.  Let $1\le h\le m\le T-1$,
and average the outer products of these patterns over the allowed
eigenvalues:
\begin{equation*}
 \mathsf H_{m,\gamma}
 =\int_0^{1-\gamma}\mu_\alpha\mu_\alpha^\top\,d\alpha,
 \qquad
 (\mathsf H_{m,\gamma})_{rs}
 =\frac{(1-\gamma)^{r+s-1}}{r+s-1}.
\end{equation*}
Let $(\sigma_j,\phi_j)_{j=1}^m$ be an orthonormal eigenbasis of this Hankel
matrix with nonincreasing eigenvalues, and retain the first $h$ eigenvectors
as filters.  Set $w_s=0$ for $s\le0$, and define, for
$Q_j\in\R^{d_u\times d_x}$,
\begin{equation}
 z_{t,j}
 =\sigma_j^{1/4}[w_{t-1},\ldots,w_{t-m}]\phi_j,
 \qquad
 v_t(Q)=\sum_{j=1}^hQ_jz_{t,j},
 \qquad
 \norm{Q}_{2,\mathrm F}^2=\sum_{j=1}^h\fnorm{Q_j}^2.
 \label{eq:spectral-features}
\end{equation}
The analysis uses three scales: a coefficient radius $R_{m,h}$, a feature
bound $Z_{m,h}$, and an approximation error $\varepsilon_{m,h}$
(\cref{lem:shifted-spectral}).  They are
\begin{align*}
 R_{m,h}
 &=\chi R_{\mathrm{sp}}\sqrt{\frac{2h}{\gamma}},
 &
 Z_{m,h}
 &=W\sqrt m
   \left(\frac12\log\frac2\gamma\right)^{1/4},
\\
 \tau_m
 &=\begin{cases}
   (1-\gamma)^m/\gamma,&m<T-1,\\
   0,&m=T-1,
   \end{cases}
 \end{align*}
and, for $1\le h\le m$,
\begin{equation}
 \varepsilon_{m,h}
 =\chi R_{\mathrm{sp},\mathrm{op}}W
 \left[
  \tau_m+
  \begin{cases}
  50\sqrt{\frac m\gamma}\log T
  \log^{1/4}\!\dfrac2\gamma
  \exp\!\left(-\dfrac{\pi^2h}{16\log T}\right),&h<m,\\
  0,&h=m
  \end{cases}
 \right].
 \label{eq:spectral-approximation}
\end{equation}
Here $\tau_m$ accounts for disturbance ages beyond $m$, and the second term
for the discarded Hankel modes.

\begin{lemma}[Shifted spectral representation]
\label{lem:shifted-spectral}
Assume $x_{\mathrm{ini}}=0$, $T\ge10$, $0<\gamma\le2/3$, and
$1\le h\le m\le T-1$.  For
$K\in\cK^{\mathrm{diag}}_{\kappa,\gamma,\chi}$, write
\[
 F_K=A-BK
 =S_K\operatorname{diag}(\alpha_1,\ldots,\alpha_{d_x})S_K^{-1},
 \qquad D_K=K_0-K,
\]
and define
\begin{equation}
 Q_j^K
 =\sigma_j^{-1/4}D_KS_K
 \operatorname{diag}\!\left(
  \ip{\phi_j}{\mu_{\alpha_1}},\ldots,
  \ip{\phi_j}{\mu_{\alpha_{d_x}}}
 \right)S_K^{-1},
 \qquad 1\le j\le h.
 \label{eq:shifted-spectral-coefficients}
\end{equation}
Then $Q^K$ depends on $K$ but not on the disturbances and, uniformly over
every sequence satisfying $\norm{w_t}_2\le W$,
\begin{equation}
 \norm{Q^K}_{2,\mathrm F}\le R_{m,h},
 \qquad
 \sup_{t\le T}
 \norm{v_t(Q^K)-(K_0-K)x_t^K}_2
 \le\varepsilon_{m,h},
 \label{eq:spectral-interface-approximation}
\end{equation}
and
\begin{equation}
 \left(\sum_{j=1}^h\norm{z_{t,j}}_2^2\right)^{1/2}
 \le Z_{m,h}
 \qquad(t\le T).
 \label{eq:spectral-feature-bound}
\end{equation}
For fixed $Q$, let the feasible reference satisfy
\begin{equation}
 \bar x_1(Q)=0,
 \qquad
 \bar u_t(Q)=-K_0\bar x_t(Q)+v_t(Q),
 \qquad
 \bar x_{t+1}(Q)=A\bar x_t(Q)+B\bar u_t(Q)+w_t.
 \label{eq:spectral-reference}
\end{equation}
Then, for all $Q,Q'$ and $t\le T$,
\begin{align}
 \norm{\bar y_t(Q)-\bar y_t(Q')}_2
 &\le C_{\mathrm{sp}}Z_{m,h}\norm{Q-Q'}_{2,\mathrm F},
 \label{eq:spectral-loss-sensitivity}\\
 \norm{\bar x_{t+1}(Q)-\bar x_{t+1}(Q')}_2
 &\le\Lambda_{\mathrm{sp},B}Z_{m,h}\norm{Q-Q'}_{2,\mathrm F}.
 \label{eq:spectral-next-state-sensitivity}
\end{align}
\end{lemma}

\begin{proof}
Let $\mathsf W_t=[w_{t-1},\ldots,w_{t-m}]$.  Expanding each pattern
$\mu_{\alpha_i}$ in the Hankel eigenbasis gives the exact identity
\begin{equation}
 D_K\sum_{r=1}^mF_K^{r-1}w_{t-r}
 =\sum_{j=1}^m D_KS_K
  \operatorname{diag}\!\left(\ip{\phi_j}{\mu_{\alpha_i}}
  \right)_{i=1}^{d_x}
  S_K^{-1}\mathsf W_t\phi_j.
 \label{eq:shifted-spectral-expansion}
\end{equation}
Thus \eqref{eq:shifted-spectral-coefficients} makes $v_t(Q^K)$ the
projection onto the leading $h$ Hankel modes of the first $m$ terms of
$(K_0-K)x_t^K$.

The Hankel leverage and eigenvalue estimates
\cite[Lemmas~7.2 and~7.4]{brahmbhatt2025} give, uniformly for
$\alpha\in[0,1-\gamma]$,
\begin{align}
 \abs{\ip{\phi_j}{\mu_\alpha}}
 &\le\sqrt{\frac2\gamma}\,\sigma_j^{1/4},
 \label{eq:spectral-leverage}\\
 \abs{\ip{\phi_j}{\mu_\alpha}}
 &\le \frac{30}{\sqrt\gamma}
 \log^{1/4}\!\frac2\gamma
 \exp\!\left(-\frac{\pi^2j}{16\log T}\right).
 \label{eq:spectral-leverage-decay}
\end{align}

\emph{Coefficient radius.}
Since $\norm{D_K}_{\mathrm F}\le R_{\mathrm{sp}}$ and the eigenbasis condition
number is at most $\chi$, \eqref{eq:spectral-leverage} yields
\[
 \fnorm{Q_j^K}
 \le\chi R_{\mathrm{sp}}\sqrt{\frac2\gamma}.
\]
Summing the squares over $j\le h$ proves the first bound in
\eqref{eq:spectral-interface-approximation}.

\emph{Approximation error.}
For the omitted Hankel modes, use the operator norm of the diagonal matrix in \eqref{eq:shifted-spectral-expansion}.  Since $\norm{\mathsf W_t\phi_j}_2\le W\sqrt m$,
\eqref{eq:spectral-leverage-decay} and a geometric sum give
\[
 \left\|\sum_{j=h+1}^m D_KS_K
  \operatorname{diag}\!\left(\ip{\phi_j}{\mu_{\alpha_i}}
  \right)_{i=1}^{d_x}
  S_K^{-1}\mathsf W_t\phi_j\right\|_2
 \le 50\chi R_{\mathrm{sp},\mathrm{op}}W
 \sqrt{\frac m\gamma}\log T\log^{1/4}\!\frac2\gamma
 e^{-\pi^2h/(16\log T)}.
\]
Disturbance ages beyond $m$ contribute at most
$\chi R_{\mathrm{sp},\mathrm{op}}W\tau_m$.  This proves the second bound in
\eqref{eq:spectral-interface-approximation}, including the exact cases $h=m$
and $m=T-1$.

\emph{Feature bound.}
Orthogonality gives a bound on all $h$ features together.  If $P_h$ projects
onto $\Span\{\phi_1,\ldots,\phi_h\}$, then
\begin{align*}
 \sum_{j=1}^h\norm{z_{t,j}}_2^2
 &=\sum_{j=1}^h\sigma_j^{1/2}
       \norm{\mathsf W_t\phi_j}_2^2
 \le\sigma_1^{1/2}\norm{\mathsf W_tP_h}_{\mathrm F}^2
 \le\sigma_1^{1/2}mW^2.
\end{align*}
Moreover,
\[
 \sigma_1\le\operatorname{tr}\mathsf H_{m,\gamma}
 \le\int_0^{1-\gamma}\frac{d\alpha}{1-\alpha^2}
 \le\frac12\log\frac2\gamma.
\]
This proves \eqref{eq:spectral-feature-bound}.

\emph{Sensitivities.}
By Cauchy--Schwarz and \eqref{eq:spectral-feature-bound},
$\norm{v_t(Q)-v_t(Q')}_2\le
Z_{m,h}\norm{Q-Q'}_{2,\mathrm F}$.  Reference differences obey
\[
 \Delta\bar x_{t+1}
 =F_0\Delta\bar x_t+B\Delta v_t,
 \qquad
 \Delta\bar u_t=-K_0\Delta\bar x_t+\Delta v_t.
\]
Convolution through $F_0$ and \eqref{eq:spectral-baseline-scales} give
\eqref{eq:spectral-loss-sensitivity} and
\eqref{eq:spectral-next-state-sensitivity}.
\end{proof}

\begin{theorem}[Spectral filters with any recovery pair]
\label{thm:spectral-correction}
Assume $x_{\mathrm{ini}}=0$, $T\ge10$, $0<\gamma\le2/3$, and
$1\le h\le m\le T-1$.  Suppose the costs are convex and globally
$L$-Lipschitz.  Fix a recovery pair $(V,\psi)$, and let $M_V$ be its
worst-direction constant in \eqref{eq:euclidean-correction-factor}.  Give
each coefficient $Q$ the reference \eqref{eq:spectral-reference} and the loss
$h_t(Q)=c_t(\bar y_t(Q))$.  Counterfactual tracking with mirror descent
\eqref{eq:reference-mirror-descent} on the ball
$\{Q:\norm{Q}_{2,\mathrm F}\le R_{m,h}\}$ as the base learner, with
$D_\Omega(Q,Q')=\frac12\norm{Q-Q'}_{2,\mathrm F}^2$, $Q_1=0$, and step size
\begin{equation}
 \eta
 =\frac{R_{m,h}}
 {LZ_{m,h}
  \sqrt{(C_{\mathrm{sp}}^2+
  2C_{\mathrm{sp}}M_V\Lambda_{\mathrm{sp},B})T}}
 \label{eq:spectral-step-size}
\end{equation}
plays
\begin{equation}
 e_t=x_t-\bar x_t(Q_t),
 \qquad
 u_t=\bar u_t(Q_t)+\psi(e_t).
 \label{eq:spectral-general-action}
\end{equation}
If $LZ_{m,h}R_{m,h}=0$, the base learner keeps $Q_t=0$ instead.  In both
cases,
\begin{equation}
 \Reg_T(\cK^{\mathrm{diag}}_{\kappa,\gamma,\chi})
 \le
 L R_{m,h}Z_{m,h}
 \sqrt{(C_{\mathrm{sp}}^2+
  2C_{\mathrm{sp}}M_V\Lambda_{\mathrm{sp},B})T}
 +LC_{\mathrm{sp}}T\varepsilon_{m,h}.
 \label{eq:spectral-global-bound}
\end{equation}
\end{theorem}

The theorem separates three choices.  The comparator margin $\gamma$
controls the spectral approximation.  The baseline $K_0$ defines the
compressed reference.  The recovery pair $(V,\psi)$ corrects the physical
state after each switch.

\begin{corollary}[Spectral filters with tuned parameters]
\label{cor:spectral-tuned}
Suppose $x_{\mathrm{ini}}=0$, $T\ge10$, $0<\gamma\le2/3$, the spectral
comparator class is nonempty, the representation baseline satisfies
\eqref{eq:spectral-baseline}, the costs are convex and globally
$L$-Lipschitz, and a recovery pair $(V,\psi)$ is fixed.  Then there exist
choices
\[
 m=\widetilde O(\gamma^{-1})
 \qquad\text{and}\qquad
 h=\operatorname{polylog}(T/\gamma)
\]
for which
$LC_{\mathrm{sp}}T\varepsilon_{m,h}\le1$ and counterfactual tracking with the
base learner and step size of \cref{thm:spectral-correction} satisfies
\begin{equation*}
 \Reg_T(\cK^{\mathrm{diag}}_{\kappa,\gamma,\chi})
 \le
 \widetilde O\!\left(
  \frac{L\chi R_{\mathrm{sp}}W}{\gamma}
  \sqrt{C_{\mathrm{sp}}^2+
  2C_{\mathrm{sp}}M_V\Lambda_{\mathrm{sp},B}}
  \sqrt T
 \right)+1.
\end{equation*}
The logarithms depend on the displayed parameters and on
$R_{\mathrm{sp},\mathrm{op}}$.
\end{corollary}

Which recovery pair should be used?  Every recovery norm dominates $\Vstar$,
so $M_\star\le M_V$.  With $(V,\psi)=(\Vstar,\psi_\star)$,
\eqref{eq:spectral-baseline-scales} and \eqref{eq:spectral-baseline} give,
after absorbing factors polynomial in $\kappa_0$ and $1+\opnorm B$,
\begin{equation*}
 \Reg_T
 \le
 \widetilde O\!\left(
  \frac{L\chi R_{\mathrm{sp}}W}
       {\gamma\gamma_0}
  \sqrt{(1+M_\star)T}
 \right)+1.
\end{equation*}
This is an oracle bound, since evaluating $\Vstar$ and $\psi_\star$ may be
hard.  A linear corrective gain gives a computable alternative.

\begin{corollary}[Spectral filters with a linear corrective law]
\label{cor:spectral-linear-correction}
Under the conditions of \cref{cor:spectral-tuned}, let
$\kappa_{\mathrm{cor}}\ge1$ and $\gamma_{\mathrm{cor}}\in(0,1)$, and let
$K_{\mathrm{cor}}$ satisfy
\begin{equation}
 \opnorm{K_{\mathrm{cor}}}\le\kappa_{\mathrm{cor}},
 \qquad
 \opnorm{(A-BK_{\mathrm{cor}})^r}
 \le\kappa_{\mathrm{cor}}^2(1-\gamma_{\mathrm{cor}})^r.
 \label{eq:spectral-correction-gain}
\end{equation}
Let $V_{\mathrm{cor}}$ be its state-feedback recovery cost and write
$M_{\mathrm{cor}}=M_{V_{\mathrm{cor}}}$.  Then
\begin{equation}
 M_{\mathrm{cor}}
 \le\sum_{r\ge0}
 \opnorm{\begin{bmatrix}I\\-K_{\mathrm{cor}}\end{bmatrix}
 (A-BK_{\mathrm{cor}})^r}
 \le
 \frac{\sqrt{1+\kappa_{\mathrm{cor}}^2}\,
       \kappa_{\mathrm{cor}}^2}{\gamma_{\mathrm{cor}}}.
 \label{eq:spectral-linear-correction-factor}
\end{equation}
With the recovery pair $(V_{\mathrm{cor}},e\mapsto-K_{\mathrm{cor}}e)$,
counterfactual tracking with the base learner and step size of
\cref{thm:spectral-correction} plays
\begin{equation}
 u_t
 =-K_{\mathrm{cor}}x_t
 +(K_{\mathrm{cor}}-K_0)\bar x_t(Q_t)
 +v_t(Q_t).
 \label{eq:spectral-linear-action}
\end{equation}
After the tuning in \cref{cor:spectral-tuned}, its regret satisfies
\begin{equation}
 \Reg_T
 \le
 \widetilde O\!\left(
  \frac{L\chi R_{\mathrm{sp}}W}
       {\gamma\gamma_0}
  \sqrt{(1+M_{\mathrm{cor}})T}
 \right)+1,
 \label{eq:spectral-linear-rate}
\end{equation}
up to factors polynomial in
$\kappa_0$, $\kappa_{\mathrm{cor}}$, and $1+\opnorm B$.
Thus the corrective-margin factor is $\gamma_{\mathrm{cor}}^{-1/2}$ up to
conditioning.  Taking $K_{\mathrm{cor}}=K_0$ gives
$u_t=-K_0x_t+v_t(Q_t)$.  Allowing the two gains to differ keeps the same
filtering complexity.
\end{corollary}

The spectral representation has $h$ learned coefficient matrices.
Naive feature evaluation costs $O(hm)$ at fixed dimensions.  Fast online
convolution reduces this to $O(h\log^2(eT))$ amortized arithmetic operations,
and recursively maintained linear reference maps give the action and gradient
without replaying the disturbance history.  \Cref{prop:spectral-runtime}
proves the complete online count, including the extra term from
$K_{\mathrm{cor}}\ne K_0$.  Constructing the filters requires a one-time
eigendecomposition of the $m\times m$ Hankel matrix, which the per-round
bound excludes.  Projection onto the coefficient ball is radial.  The
reference-bounded set used below requires its own projection oracle.

What if gradients are bounded only on a ball?  Then the algorithm must keep
the physical trajectory inside that ball.  We restrict the coefficients to
those whose references stay in half of it, and show that the tracking error
fills at most the other half.  For $D\ge0$, define the convex set
\begin{equation}
 \cQ_{\mathrm{sp}}(D)
 =\left\{Q:\norm{Q}_{2,\mathrm F}\le R_{m,h},\quad
   \sup_{t\le T}\sup_{\norm{w_s}_2\le W}
   \norm{\bar y_t(Q;w)}_2\le D/2
  \right\}.
 \label{eq:spectral-bounded-set}
\end{equation}
The proof of \cref{cor:spectral-local-gradients} uses the radius and
accuracy conditions
\begin{equation}
 D\ge
 4\frac{\kappa^2W}{\gamma}\sqrt{1+\kappa^2},
 \qquad
 C_{\mathrm{sp}}\varepsilon_{m,h}\le D/4,
 \label{eq:spectral-local-radius}
\end{equation}
together with $GDC_{\mathrm{sp}}T\varepsilon_{m,h}\le1$.

\begin{corollary}[Spectral filters with locally bounded gradients]
\label{cor:spectral-local-gradients}
Assume $x_{\mathrm{ini}}=0$, $T\ge10$, and $0<\gamma\le2/3$, and retain the
comparator class \eqref{eq:spectral-comparator-class} and the spectral
baseline \eqref{eq:spectral-baseline}.  Suppose an explicit corrective gain
$K_{\mathrm{cor}}$ satisfies, for some $\kappa_{\mathrm{cor}}\ge1$ and
$\gamma_{\mathrm{cor}}\in(0,1)$,
\[
 \norm{K_{\mathrm{cor}}}_{\mathrm{op}}\le\kappa_{\mathrm{cor}},
 \qquad
 \norm{(A-BK_{\mathrm{cor}})^r}_{\mathrm{op}}
 \le\kappa_{\mathrm{cor}}^2(1-\gamma_{\mathrm{cor}})^r
 \quad(r\ge0).
\]
Set
\[
 D=4\frac{\kappa^2W}{\gamma}\sqrt{1+\kappa^2}.
\]
Suppose each cost is convex and differentiable, with
$\norm{\nabla c_t(y)}_2\le GD$ on the Euclidean ball of radius $D$.
Use dimensions $1\le h\le m\le T-1$ satisfying
$m=\widetilde O(\gamma^{-1})$,
$h=\operatorname{polylog}(T/\gamma)$, and
\[
 C_{\mathrm{sp}}\varepsilon_{m,h}\le D/4,
 \qquad
 GDC_{\mathrm{sp}}T\varepsilon_{m,h}\le1;
\]
such dimensions exist by the tuning of \cref{cor:spectral-tuned} with this
accuracy target.  Assume an oracle for projection onto
$\cQ_{\mathrm{sp}}(D)$ in \eqref{eq:spectral-bounded-set}, the convex set of
filter coefficients whose references stay in the radius-$D/2$ ball.
Let $P_{\mathrm{cor}}=P_{K_{\mathrm{cor}}}$ be the peak tracking factor
in \eqref{eq:peak-tracking-gain}.  Run counterfactual tracking with the
corrective law $e\mapsto-K_{\mathrm{cor}}e$ and the following base learner.
It starts from any $Q_1\in\cQ_{\mathrm{sp}}(D)$.  If $GWR_{\mathrm{sp}}=0$,
it keeps $Q_t=Q_1$.  Otherwise, for
\begin{equation}
 T\ge T_{\mathrm{sp}}
 :=\frac{4R_{m,h}^2Z_{m,h}^2
   (C_{\mathrm{sp}}^2+
    2C_{\mathrm{sp}}P_{\mathrm{cor}}\Lambda_{\mathrm{sp},B})}{D^2},
 \label{eq:spectral-local-threshold}
\end{equation}
it runs mirror descent \eqref{eq:reference-mirror-descent} on
$\cQ_{\mathrm{sp}}(D)$ with the losses of \cref{thm:spectral-correction},
$D_\Omega(Q,Q')=\frac12\norm{Q-Q'}_{2,\mathrm F}^2$, and step size
\begin{equation}
 \eta=\frac{2R_{m,h}}
 {GDZ_{m,h}\sqrt{
  (C_{\mathrm{sp}}^2+2C_{\mathrm{sp}}P_{\mathrm{cor}}
   \Lambda_{\mathrm{sp},B})T}}.
 \label{eq:spectral-local-step}
\end{equation}
For $T<T_{\mathrm{sp}}$, it keeps $Q_t=Q_1$.  The resulting algorithm
plays \eqref{eq:spectral-linear-action} in every branch.  Its actual
trajectory stays in the radius-$D$ ball, and
\begin{equation*}
 \Reg_T(\cK^{\mathrm{diag}}_{\kappa,\gamma,\chi})
 \le\widetilde O\!\left(
 \frac{GW^2\chi R_{\mathrm{sp}}}
 {\gamma^2\gamma_0\sqrt{\gamma_{\mathrm{cor}}}}\sqrt T
 \right)+1.
\end{equation*}
The $\widetilde O$ notation also hides dimension and conditioning factors.
\end{corollary}

The globally Lipschitz consequence summarized in
\cref{sec:spectral-compression} is the following.

\begin{corollary}[Spectral-filter learning for stable state feedback]
\label{cor:main-spectral-response}
Assume the spectral policy class is nonempty, $T\ge10$, and
$0<\gamma\le2/3$.  For convex globally $L$-Lipschitz costs and any recovery
pair $(V,\psi)$, there are a filter length
$m=\widetilde O(\gamma^{-1})$ and
$h=\operatorname{polylog}(T/\gamma)$ learned coefficient matrices for which
counterfactual tracking with $(V,\psi)$, using the base learner and step size
of \cref{thm:spectral-correction}, satisfies
\begin{equation*}
 \Reg_T(\cK^{\mathrm{diag}}_{\kappa,\gamma,\chi})
 \le\widetilde O_{\kappa_0,\,1+\norm{B}_{\mathrm{op}}}\!\left(
 \frac{L\chi R_{\mathrm{sp}}W}{\gamma\gamma_0}
 \sqrt{(1+M_V)T}\right)+1,
\end{equation*}
where the subscript suppresses polynomial factors.  \Cref{cor:spectral-tuned}
gives the logarithmic factors and a more detailed bound.
\end{corollary}

\subsection{Spectral filters: proofs}
\label{app:spectral-proofs}

\begin{proof}[Proof of \cref{thm:spectral-correction}]
The proof bounds the base learner's regret and movement on the filtered
class, then adds the representation error of the comparators.

Let $h_t(Q)=c_t(\bar y_t(Q))$.  By
\eqref{eq:spectral-loss-sensitivity}, every subgradient of $h_t$ has
$\norm{\cdot}_{2,\mathrm F}$ norm at most $LC_{\mathrm{sp}}Z_{m,h}$.
Projected subgradient descent therefore gives, for every $Q$ in the
radius-$R_{m,h}$ ball,
\begin{equation*}
 \sum_{t=1}^T[h_t(Q_t)-h_t(Q)]
 \le
 \frac{R_{m,h}^2}{2\eta}
 +\frac{\eta L^2C_{\mathrm{sp}}^2Z_{m,h}^2T}{2}.
\end{equation*}
Since projection is nonexpansive,
\[
 \norm{Q_{t+1}-Q_t}_{2,\mathrm F}
 \le\eta LC_{\mathrm{sp}}Z_{m,h}.
\]
By \eqref{eq:spectral-next-state-sensitivity}, each switching cost is
therefore at most
\[
 V(\bar x_{t+1}(Q_t)-\bar x_{t+1}(Q_{t+1}))
 \le
 \eta L C_{\mathrm{sp}}M_V
 \Lambda_{\mathrm{sp},B}Z_{m,h}^2.
\]
Applying \cref{thm:master-reduction} and substituting
\eqref{eq:spectral-step-size} yields the first term in
\eqref{eq:spectral-global-bound}.

For a comparator $K$, take $Q=Q^K$ from
\cref{lem:shifted-spectral} and put
$r_t=v_t(Q^K)-(K_0-K)x_t^K$.  The reference--comparator state
difference satisfies
\[
 \bar x_{t+1}(Q^K)-x_{t+1}^K
 =F_0
  (\bar x_t(Q^K)-x_t^K)+Br_t.
\]
The same convolution through $F_0$ that defines $C_{\mathrm{sp}}$ gives
\(
 \sup_{t\le T}\norm{\bar y_t(Q^K)-y_t^K}_2
 \le C_{\mathrm{sp}}\varepsilon_{m,h}
\).
Lipschitzness then bounds the representation error by
$LC_{\mathrm{sp}}T\varepsilon_{m,h}$, completing the proof.
\end{proof}

\begin{proof}[Proof of \cref{cor:spectral-tuned}]
By the zero-scale convention at the start of \cref{app:response-proofs}, it
remains to consider $L>0$.  Let
$P=\max\{1,LC_{\mathrm{sp}}T
 \chi R_{\mathrm{sp},\mathrm{op}}W\}$.  Choose
\[
 m=\min\left\{T-1,
  \left\lceil\frac{c}{\gamma}
   \log\!\left(e+\frac P\gamma\right)\right\rceil\right\}
\]
and then
\[
 h=\min\left\{m,
  \left\lceil c\log(eT)
   \log\!\left(e+P\sqrt{\frac m\gamma}\log(eT)
     \log^{1/4}\!\frac2\gamma\right)\right\rceil\right\}
\]
for a sufficiently large universal constant $c$.  Substitution in
\eqref{eq:spectral-approximation} makes each nonzero approximation term at
most $(2LC_{\mathrm{sp}}T)^{-1}$, so the representation term is at most
one.  Moreover,
$m=\widetilde O(\gamma^{-1})$,
$h=\operatorname{polylog}(T/\gamma)$, and
\[
 R_{m,h}Z_{m,h}
 =\widetilde O\!\left(
  \frac{\chi R_{\mathrm{sp}}W}{\gamma}
 \right).
\]
Substitute these estimates into \eqref{eq:spectral-global-bound}.
\end{proof}

\begin{proof}[Proof of \cref{cor:spectral-linear-correction}]
Let $F_{\mathrm{cor}}=A-BK_{\mathrm{cor}}$ and define
\[
 V_{\mathrm{cor}}(e)
 =\sum_{r\ge0}
  \norm{\begin{bmatrix}I\\-K_{\mathrm{cor}}\end{bmatrix}
  F_{\mathrm{cor}}^re}_2.
\]
Policy evaluation gives \eqref{eq:correction-certificate} with the
corrective law $e\mapsto-K_{\mathrm{cor}}e$, so this is a recovery pair.
Its worst-direction constant is $M_{\mathrm{cor}}$, and
\eqref{eq:spectral-linear-correction-factor} follows from
\eqref{eq:spectral-correction-gain}.  Substitution in
\cref{cor:spectral-tuned} proves \eqref{eq:spectral-linear-rate}.  Expanding
\eqref{eq:spectral-general-action} gives
\eqref{eq:spectral-linear-action}.
\end{proof}

\begin{proof}[Proof of \cref{cor:spectral-local-gradients}]
The proof shows that every comparator is represented in
$\cQ_{\mathrm{sp}}(D)$, runs the proof of \cref{thm:spectral-correction}
with Lipschitz constant $GD$, and checks that the tracking error keeps the
physical trajectory in the radius-$D$ ball.

\emph{Comparators.}  Every comparator satisfies
\[
 \sup_{t\le T}\norm{y_t^K}_2
 \le\frac{\kappa^2W}{\gamma}\sqrt{1+\kappa^2}.
\]
With the reference approximation from the proof of
\cref{thm:spectral-correction}, conditions
\eqref{eq:spectral-local-radius} place every $Q^K$ in
$\cQ_{\mathrm{sp}}(D)$.

\emph{Long horizons.}  Suppose $T\ge T_{\mathrm{sp}}$ and use the step size
\eqref{eq:spectral-local-step}.  The diameter of $\cQ_{\mathrm{sp}}(D)$ is at
most $2R_{m,h}$.  Repeating the proof of
\cref{thm:spectral-correction} with Lipschitz constant $GD$ gives
\[
 \Reg_T(K)
 \le
 2GD R_{m,h}Z_{m,h}
 \sqrt{C_{\mathrm{sp}}^2+2C_{\mathrm{sp}}P_{\mathrm{cor}}\Lambda_{\mathrm{sp},B}}\sqrt T+1,
\]
provided the physical trajectory remains in the radius-$D$ ball.

We verify this condition.  Projection and the gradient bound give
\[
 \norm{Q_{t+1}-Q_t}_{2,\mathrm F}
 \le
 \frac{2R_{m,h}C_{\mathrm{sp}}}
 {\sqrt{C_{\mathrm{sp}}^2+2C_{\mathrm{sp}}P_{\mathrm{cor}}\Lambda_{\mathrm{sp},B}}\sqrt T}.
\]
The switch discrepancy is at most $\Lambda_{\mathrm{sp},B}Z_{m,h}$ times this
quantity.  Since the tracking error obeys
\[
 e_{t+1}=(A-BK_{\mathrm{cor}})e_t+\delta_{t+1},
\]
the convolution bound in the definition of $P_{\mathrm{cor}}$ yields
\begin{align*}
 \norm{y_t-\bar y_t(Q_t)}_2
 &\le
 \frac{2R_{m,h}Z_{m,h}
  P_{\mathrm{cor}}
  \Lambda_{\mathrm{sp},B}C_{\mathrm{sp}}}
  {\sqrt{C_{\mathrm{sp}}^2+2C_{\mathrm{sp}}P_{\mathrm{cor}}\Lambda_{\mathrm{sp},B}}\sqrt T}\\
 &\le
 \frac{R_{m,h}Z_{m,h}
  \sqrt{C_{\mathrm{sp}}^2+2C_{\mathrm{sp}}P_{\mathrm{cor}}\Lambda_{\mathrm{sp},B}}}{\sqrt T}.
\end{align*}
The threshold \eqref{eq:spectral-local-threshold} therefore makes the
tracking error at most $D/2$.  Every reference in $\cQ_{\mathrm{sp}}(D)$ has
norm at most $D/2$, so the physical trajectory stays in the radius-$D$ ball.

\emph{Short horizons.}  If $T<T_{\mathrm{sp}}$, hold $Q_1$ fixed.  The
physical trajectory then equals its feasible reference, and both it and
every comparator lie in the radius-$D/2$ ball.  The crude bound
\[
 \Reg_T(K)\le GD^2T
 <2GD R_{m,h}Z_{m,h}
  \sqrt{C_{\mathrm{sp}}^2+2C_{\mathrm{sp}}P_{\mathrm{cor}}\Lambda_{\mathrm{sp},B}}\sqrt T
\]
proves the same result.

\emph{Rate.}  Finally, substitute
$D=O(W/\gamma)$,
$R_{m,h}Z_{m,h}=\widetilde O(\chi R_{\mathrm{sp}}W/\gamma)$,
$C_{\mathrm{sp}},\Lambda_{\mathrm{sp},B}
 =O(\gamma_0^{-1})$ up to spectral-baseline conditioning, and
$P_{\mathrm{cor}}
 =O(\gamma_{\mathrm{cor}}^{-1})$ up to corrective-gain conditioning.
\end{proof}

\subsection{Spectral filters: computation}
\label{app:spectral-computation}

How can each round avoid replaying the disturbance history?  The filters are
fixed before play, so fast online convolution computes the features.  The
coefficients change every round, so we store the reference as a linear map of the coefficients.

\begin{proposition}[Time per round for spectral filters]
\label{prop:spectral-runtime}
Fix the $h$ filters of length $m$ in \eqref{eq:spectral-features}, a known
stable baseline $K_0$, and a known linear corrective gain $K_{\mathrm{cor}}$.
At fixed $d_x,d_u$, counterfactual tracking with the corrective law
$e\mapsto-K_{\mathrm{cor}}e$ and the base learner of
\cref{thm:spectral-correction} or \cref{cor:spectral-local-gradients} admits
an implementation with
$O(h\log^2(eT))$ amortized arithmetic operations per round, one cost-gradient
evaluation, and one coefficient projection.  This count includes feature,
reference, and gradient evaluation and the physical action.
For the coefficient ball in \cref{thm:spectral-correction}, projection takes
$O(h)$ operations and is absorbed in this bound.  For
\cref{cor:spectral-local-gradients}, the total is
\[
 O(h\log^2(eT))+\mathsf C_{\mathrm{proj}},
\]
where $\mathsf C_{\mathrm{proj}}$ is the amortized cost of projection onto
$\cQ_{\mathrm{sp}}(D)$.  With the tuned filter count, the arithmetic term is
polylogarithmic in $T$ and the inverse decay margins.  Filter construction
and the computation of the two gains are offline.
\end{proposition}

\begin{proof}
\emph{Fixed-filter features.}
Apply Continuous-FutureFill
\cite[Theorem~3]{agarwal2025futurefill} to each filter and each disturbance
coordinate, padding the filters with zeros to the horizon.  It computes each
online convolution in $O(T\log^2(eT))$ total arithmetic operations.  The
scaled features $z_{t,j}$ therefore cost $O(h\log^2(eT))$ per round on
average at fixed dimensions.  They depend only on disturbances available
before that round's action.

\emph{Changing reference coefficients.}
Let $b_1=0$ and $b_{t+1}=F_0b_t+w_t$ be the baseline state.  For each filter,
maintain a linear map from $\R^{d_u\times d_x}$ to $\R^{d_x}$ by
\[
 \mathcal S_{1,j}=0,
 \qquad
 \mathcal S_{t+1,j}(U)=F_0\mathcal S_{t,j}(U)+BUz_{t,j}.
\]
Store its values on the elementary matrix basis.  At fixed dimensions this
takes a fixed number of operations per filter.  By induction from
\eqref{eq:spectral-reference},
\[
 \bar x_t(Q)=b_t+\sum_{j=1}^h\mathcal S_{t,j}(Q_j),
 \qquad
 \bar u_t(Q)=-K_0\bar x_t(Q)+\sum_{j=1}^hQ_jz_{t,j}.
\]
The stored maps depend only on the disturbances and filters, so the identity holds for any current $Q_t$.
Evaluating it and the action \eqref{eq:spectral-linear-action} costs $O(h)$.

\emph{Gradient and projection.}
If $(g_x,g_u)$ is the cost gradient at $\bar y_t(Q_t)$, the Frobenius gradient
of the reference loss has blocks
\[
 \nabla_{Q_j}h_t(Q_t)
 =\mathcal S_{t,j}^{*}(g_x-K_0^\top g_u)+g_u z_{t,j}^{\top}.
\]
The adjoint maps are available from the same stored coefficients, so
computing all blocks and taking the gradient step costs $O(h)$.
The coefficient-ball projection is radial in $\norm{\cdot}_{2,\mathrm F}$
and also costs $O(h)$.  Projection onto $\cQ_{\mathrm{sp}}(D)$ is the
separate oracle of \cref{cor:spectral-local-gradients}.  Summing these costs
proves the proposition.
\end{proof}

The time bound is amortized.  It bounds neither the worst-round latency nor the memory.  At fixed dimensions, the cited
convolution routine can use $O(hT)$ additional storage, the reference maps
use $O(h)$, and the filters use $O(hm)$.  A dense eigendecomposition for
filter construction costs $O(m^3)$ operations.  These counts assume exact
real arithmetic with the filters provided; numerical conditioning and bit
precision are separate issues.

The prior spectral entry of \cref{tab:stable-comparison} rests on the same
online-convolution argument, which gives an $O(\log^4(T/\gamma))$ average
running time \cite[Corollary~2.2]{brahmbhatt2025}.  That algorithm also
projects onto a reference-bounded coefficient set
\cite[Algorithm~1]{brahmbhatt2025}.  For both spectral rows, we display the projection cost $\mathsf C_{\mathrm{proj}}$ separately.  The comparison thus uses one convention for both
methods, and the polylogarithmic arithmetic term is not inferred from the
parameter count alone.

\section{Proofs of the lower bounds}
\label{app:lower-proofs}

\paragraph{Randomization and expectations.}
Expected regret is the upper expectation
\eqref{eq:minimax-upper-expectation} of the whole regret.  Each sign
construction below has finite support, so once the algorithm's private seed is
fixed, all trajectories and sign averages are finite.  To pass from fixed
seeds to randomized algorithms, let $R_j(a)$ be the regret for seed $a$ on
environment sequence $j$, which has probability $q_j>0$, and suppose
$\sum_jq_jR_j(a)\ge b$ for every $a$.  If $\E R_j^+=\infty$ for some $j$,
the upper expected regret is $+\infty$, both on that sequence and jointly.
Otherwise every positive part is integrable, and the seedwise bound forces
every negative part to be integrable as well.  Averaging over seeds then
gives joint expected regret at least $b$, and some deterministic sequence
attains it.

For the quadratic constructions, the costs are uniformly bounded below, and
at each fixed horizon the comparator trajectories are uniformly bounded.
Each whole regret therefore has an ordinary extended expectation.  We
average the seedwise inequalities for these whole regrets and never
separate a possibly nonintegrable linear term from its quadratic term.

\subsection{The one-step system}
\label{app:one-step-correction}

\begin{proof}[Proof of \eqref{eq:one-step-correction}]
For $A=0$, the input $v_0=0$ sends every initial discrepancy $z$ to
$e_1=0$, at total cost $\norm{(z,0)}=\abs{z}$.  Conversely, every correction
sequence pays at least $\abs{z}$ in its first stage.  Hence
$\Vstar(z)=\abs{z}$.
\end{proof}

\subsection{Finite classes}
\label{app:finite-library-lower}

\begin{proposition}[Finite-class lower bound with loss-independent policies]
\label{prop:finite-policy-lower}
There is a universal constant $c>0$ such that the following holds for every
$T\ge1$, $N\ge2$, and $L>0$.  On the fixed scalar system
$x_{t+1}=u_t+w_t$, with $x_1=0$ and zero disturbances, there is a known class
$\Pi$ of $N$ deterministic, loss-independent open-loop policies, whose
inputs are fixed in advance, such that, under the Euclidean state--input norm,
\[
 D_t\le2\sqrt2,
 \qquad
 D_{t,V}\le2(1+\sqrt2)
 \quad\text{for }V=\Vstar.
\]
For every randomized causal algorithm, some oblivious sequence of convex,
globally $L$-Lipschitz costs satisfies
\begin{equation}
 \E\Reg_T(\Pi)
 \ge cL\min\left\{T,\sqrt{T\log N}\right\}.
 \label{eq:finite-policy-lower}
\end{equation}
\end{proposition}

\begin{proof}[Proof of \cref{prop:finite-policy-lower}]
The hard instance consists of about $\log_2N$ independent sign-prediction
problems, one per block of rounds, and the policies are the block-constant
sign sequences.  Set
\[
 m=\min\{T,\lfloor\log_2N\rfloor\},
\]
and partition the $T$ rounds into $m$ nonempty consecutive blocks
$B_1,\ldots,B_m$ whose sizes differ by at most one.  For each
$s=(s_1,\ldots,s_m)\in\{-1,1\}^m$, include the open-loop policy
\[
 u_t^s=s_j\qquad(t\in B_j),
\]
with inputs set to zero after round $T$.  These are $2^m\le N$ distinct
policies.  If $2^m<N$, add distinct open-loop policies whose first $T$
inputs lie in $[-1,1]^T$ until the class has $N$ elements.  Adding
policies can only decrease the comparator cost, so it can only strengthen
the lower bound.

All policies are deterministic, causal, and loss-independent.  On the
one-step system \eqref{eq:one-step-system} with zero disturbances,
$x_t^\pi=u_{t-1}^\pi$ for $t\ge2$, so every state and input coordinate of a
policy trajectory lies in $[-1,1]$.  This gives
\[
 D_t\le2\sqrt2.
\]
By \eqref{eq:one-step-correction}, $\Vstar(z)=\abs{z}$, so the trajectory
norm \eqref{eq:induced-trajectory-norm} is
\[
 \norm{(e,v)}_{\Vstar}=\sqrt{e^2+v^2}+\abs{v}.
\]
Hence $D_{t,V}\le2(1+\sqrt2)$ for $V=\Vstar$.

Draw independent Rademacher signs $\sigma_1,\ldots,\sigma_T$ before play and
set
\[
 c_t(x,u)=L\abs{u-\sigma_t}.
\]
These costs are convex and globally $L$-Lipschitz in the Euclidean
state--input norm.  The fresh sign is uniform and independent of the
algorithm's input, so
\[
 \E\!\left[c_t(x_t,u_t)\mid u_t,\sigma_{1:t-1}\right]
 =\frac L2\bigl(\abs{u_t-1}+\abs{u_t+1}\bigr)
 =L\max\{1,\abs{u_t}\}
 \ge L.
\]
The algorithm's expected cumulative cost is thus at least $LT$, whatever
inputs it plays.

For each block, write $S_j=\sum_{t\in B_j}\sigma_t$.  Since
$\abs{s_j-\sigma_t}=1-s_j\sigma_t$, the best block-sign policy has cost
\[
 \min_{s\in\{-1,1\}^m}
 L\sum_{j=1}^m\sum_{t\in B_j}\abs{s_j-\sigma_t}
 =L\left(T-\sum_{j=1}^m\abs{S_j}\right).
\]
Consequently,
\begin{equation}
 \E\Reg_T(\Pi)\ge L\sum_{j=1}^m\E\abs{S_j}.
 \label{eq:finite-policy-block-lower}
\end{equation}

For a sum $S_q$ of $q$ independent Rademacher signs,
\[
 \E S_q^2=q,
 \qquad
 \E S_q^4=3q^2-2q\le3q^2.
\]
H\"older's inequality gives
\[
 q=\E\abs{S_q}^2
 \le\bigl(\E\abs{S_q}\bigr)^{2/3}
     \bigl(\E\abs{S_q}^4\bigr)^{1/3},
\]
and therefore $\E\abs{S_q}\ge\sqrt{q/3}$.  Every block has size at least
$\lfloor T/m\rfloor\ge T/(2m)$, so the right side of
\eqref{eq:finite-policy-block-lower} is at least
\[
 L\sqrt{\frac{mT}{6}}.
\]
Since $N\ge2$,
\[
 m\ge\frac12\min\{T,\log N\}.
\]
This proves \eqref{eq:finite-policy-lower}, after adjusting the universal
constant.  The signs are drawn obliviously from a finite support, so the
averaging argument at the start of this appendix gives, for each randomized
algorithm, a deterministic sign sequence that attains the same bound in
expectation over the algorithm's private randomness.
\end{proof}

\subsection{Shifted Rademacher sums}
\label{app:delayed-signs}

The response lower bounds correlate a cost-sign sequence with shifted
copies of a disturbance-sign sequence.  For linear costs, the best single
disturbance age gives the required gain.  Quadratic costs also penalize the
energy of the response.  We therefore group the disturbance ages, choose one
favorable age in each group, and spread the input coefficients over these
choices.  This keeps the order of the linear gain while reducing the squared
coefficient norm.

\Cref{lem:rademacher-separated} gives the correlation gain for separated
sign vectors.  \Cref{lem:shift-convolution} shows that the shifted vectors
are separated and bounds the energy of their convolution, and
\cref{lem:shifted-signs} combines the two.  The first lemma transfers
Gaussian Sudakov minoration to Rademacher signs, which the bounded
coordinates permit.

\begin{lemma}[Rademacher maximum for separated vectors]
\label{lem:rademacher-separated}
There are universal constants $c>0$ and $n_0$ such that the following
holds.  Let $n\ge n_0$, let
$\mathcal A\subset[-1,1]^n$ contain $2\le m\le2n$ vectors, and suppose
\begin{equation*}
 \norm{a-b}_2\ge\sqrt n
 \qquad(a\ne b).
\end{equation*}
If $\zeta$ has independent Rademacher coordinates, then
\begin{equation*}
 \E_\zeta\max_{a\in\mathcal A}\ip{\zeta}{a}
 \ge c\sqrt{n\log m}.
\end{equation*}
\end{lemma}

\begin{proof}
Let $g$ be a standard Gaussian vector.  Gaussian Sudakov minoration
\cite[Theorem~3.18]{ledouxtalagrand1991} gives
\begin{equation}
 \E_g\max_{a\in\mathcal A}\ip{g}{a}
 \ge c_G\sqrt{n\log m}.
 \label{eq:gaussian-sudakov}
\end{equation}
For $\beta>0$, the smooth maximum
\[
 F_\beta(z)
 =\frac1\beta\log\sum_{a\in\mathcal A}
   \exp\!\bigl(\beta\ip{a}{z}\bigr)
\]
satisfies
\begin{equation*}
 \max_{a\in\mathcal A}\ip{a}{z}
 \le F_\beta(z)
 \le \max_{a\in\mathcal A}\ip{a}{z}
      +\frac{\log m}{\beta}.
\end{equation*}
For each coordinate $i$, differentiation under the finite sum gives
\[
 \partial_i^3F_\beta(z)
 =\beta^2
   \E_{a\sim p_z}
   \left[(a_i-\E_{p_z}a_i)^3\right],
\]
where $p_z(a)$ is proportional to
$\exp(\beta\ip{a}{z})$.  Since $\abs{a_i}\le1$,
$\sup_z\abs{\partial_i^3F_\beta(z)}\le8\beta^2$.
Replace the coordinates of $g$ by Rademacher signs one at a time.  The two
variables match in their first two moments, so the third-order Taylor
remainder yields
\begin{equation}
 \abs{\E F_\beta(g)-\E F_\beta(\zeta)}
 \le C\beta^2n.
 \label{eq:lindeberg-smooth-max}
\end{equation}
Combining \eqref{eq:gaussian-sudakov}--\eqref{eq:lindeberg-smooth-max} gives
\[
 \E\max_{a\in\mathcal A}\ip{\zeta}{a}
 \ge c_G\sqrt{n\log m}
    -C\beta^2n-\frac{\log m}{\beta}.
\]
Choose $\beta=(\log m/n)^{1/3}$.  Both error terms are then bounded by
\[
 Cn^{1/3}(\log m)^{2/3}
 =C\left(\frac{\log m}{n}\right)^{1/6}
   \sqrt{n\log m}.
\]
Because $m\le2n$, the prefactor tends to zero uniformly as
$n\to\infty$.  Increasing $n_0$ makes the error at most half of the right
side of \eqref{eq:gaussian-sudakov}.
\end{proof}

The shifted sign vectors satisfy the separation hypothesis with high
probability, and their convolution matrix has small operator norm, which
keeps the comparator's energy small.

\begin{lemma}[Separated shifts with bounded convolution norm]
\label{lem:shift-convolution}
There are universal constants $c,C>0$ and $T_0$ such that the following
holds for every $T\ge T_0$.  Let $M=\lfloor T/4\rfloor$, draw independent
Rademacher signs $\xi_1,\ldots,\xi_T$, extend them by zero outside
$\{1,\ldots,T\}$, and define
\[
 v_j=(\xi_{t-j})_{t=1}^T,
 \qquad
 X_\xi=[v_1\ \cdots\ v_M]\in\R^{T\times M}.
\]
With probability at least $1/2$,
\begin{align}
 \norm{X_\xi}_{\mathrm{op}}^2
 &\le CT\log(eT),
 \label{eq:shift-convolution-norm}\\
 \E_\varepsilon\max_{j\in J}
   \abs{\ip{\varepsilon}{v_j}}
 &\ge c\sqrt{T\log(e\abs{J})}
 \quad\text{for every nonempty }J\subseteq\{1,\ldots,M\},
 \label{eq:shift-group-correlation}
\end{align}
where $\varepsilon$ is an independent Rademacher vector.
\end{lemma}

\begin{proof}
For $j\le M$, $\norm{v_j}_2^2=T-j\ge3T/4$.  For $j<k$,
\[
 \ip{v_j}{v_k}
 =\sum_{s=1}^{T-k}\xi_s\xi_{s+k-j}.
\]
This sum has mean zero, and changing one coordinate of $\xi$ changes at
most two summands, each by at most two.  The bounded-differences
inequality therefore gives
\[
 \Pr\!\left(\abs{\ip{v_j}{v_k}}>T/8\right)
 \le2e^{-T/512}.
\]
A union bound over the pairs shows that, for all sufficiently large $T$,
with probability at least $3/4$,
\begin{equation*}
 \abs{\ip{v_j}{v_k}}\le T/8
 \qquad(j\ne k).
\end{equation*}
On this event, for every nonempty $J$, the set
$\{v_j,-v_j:j\in J\}\subset[-1,1]^T$ has pairwise Euclidean
separation at least $\sqrt T$ and at most $2M\le2T$ elements.
Applying \cref{lem:rademacher-separated} conditionally on $\xi$
proves \eqref{eq:shift-group-correlation} simultaneously for every
$J$.

For \eqref{eq:shift-convolution-norm}, pad $\xi$ with zeros and embed
$X_\xi$ in the $2T$-dimensional circulant convolution matrix.  The
operator norm of that matrix is
\[
 \max_{0\le q<2T}
 \left|
   \sum_{s=1}^T\xi_s e^{-2\pi\mathrm{i}q s/(2T)}
 \right|.
\]
Hoeffding's inequality for the real and imaginary parts, followed by a
union bound over the $2T$ frequencies, bounds this maximum by
$C\sqrt{T\log(eT)}$ with probability at least $3/4$.  A submatrix
cannot have larger operator norm.  Intersecting the two
probability-$3/4$ events proves the lemma.
\end{proof}

The next lemma puts these estimates in the form used by both parts of
\cref{thm:response-lower}.

\begin{lemma}[Shifted-sign packing]
\label{lem:shifted-signs}
There are universal constants $c,C>0$ and $T_0$ such that the following
holds for every $T\ge T_0$.  Let $\xi_1,\ldots,\xi_T$ and
$\varepsilon_1,\ldots,\varepsilon_T$ be independent Rademacher
sequences, extended by zero outside $\{1,\ldots,T\}$, and let
\[
 S_j=\sum_{t=1}^T\varepsilon_t\xi_{t-j},
 \qquad 1\le j\le H,
\]
where $T/4\le H\le T/2$.  Then
\begin{equation}
 \E\max_{j\le H}\abs{S_j}
 \ge c\sqrt{T\log(eH)}.
 \label{eq:shifted-max}
\end{equation}
Moreover, for every $r>0$ there is a vector $\theta\in\R^H$, measurable
from the two sign sequences, such that
\begin{align}
 \norm{\theta}_1&\le r,
 \notag\\
 \E\sum_j\theta_jS_j
 &\le-cr\sqrt{T\log(eH)},
 \label{eq:theta-correlation}\\
 \sum_{t=1}^T
 \left(\sum_j\theta_j\xi_{t-j}\right)^2
 &\le Cr^2\sqrt{T\log(eH)}
 \quad\text{almost surely}.
 \label{eq:theta-energy}
\end{align}
\end{lemma}

\begin{proof}
Let $M=\lfloor T/4\rfloor$ and use the vectors and matrix from
\cref{lem:shift-convolution}.  Since $M\le H$, their ages are among the
disturbance ages in the statement.  On the event in that lemma,
\eqref{eq:shift-group-correlation} with
$J=\{1,\ldots,M\}$ gives
\[
 \E_\varepsilon\max_{j\le H}\abs{S_j}
 \ge c\sqrt{T\log(eM)}.
\]
The event has probability at least $1/2$, and
$\log(eM)\asymp\log(eH)$.  Averaging over $\xi$ proves
\eqref{eq:shifted-max}.

For the second conclusion, $\theta$ spreads its mass over $k$ ages, one
from each of $k$ groups of $m$ ages, where
\[
 k=\left\lfloor c_1\sqrt{T\log(eT)}\right\rfloor,
 \qquad
 m=\left\lfloor c_2\sqrt{\frac{T}{\log(eT)}}\right\rfloor,
\]
and the positive universal constants $c_1,c_2$ are small enough that
$km\le M$ for all sufficiently large $T$.  Partition the first $km$
disturbance ages into $k$ disjoint groups $J_1,\ldots,J_k$ of size $m$.
On the event in \cref{lem:shift-convolution}, in each group let $j_g$
be the least index maximizing $\abs{\ip{\varepsilon}{v_j}}$, and set
\[
 \theta_{j_g}
 =-\frac r k\operatorname{sign}\ip{\varepsilon}{v_{j_g}},
\]
with all other entries zero.  Set $\theta=0$ off that event.  The
least-index rule makes the choice measurable, and
$\norm{\theta}_1\le r$.  Conditional on $\xi$, applying
\eqref{eq:shift-group-correlation} to every $J_g$ gives
\begin{align*}
 \E_\varepsilon\sum_j\theta_jS_j
 &=-\frac r k\sum_{g=1}^k
   \E_\varepsilon\max_{j\in J_g}
      \abs{\ip{\varepsilon}{v_j}}\\
 &\le-cr\sqrt{T\log(em)}
 \le-cr\sqrt{T\log(eH)}.
\end{align*}
The last inequality uses
$\log(em)\asymp\log(eT)\asymp\log(eH)$.  Averaging over the
probability-$1/2$ event proves \eqref{eq:theta-correlation}.

Finally, $\norm{\theta}_2^2\le r^2/k$, so pointwise on the same event,
\[
 \sum_t\left(\sum_j\theta_j\xi_{t-j}\right)^2
 =\norm{X_\xi\theta_{1:M}}_2^2
 \le CT\log(eT)\frac{r^2}{k}
 \le Cr^2\sqrt{T\log(eH)}.
\]
Off the event, the bound holds because $\theta=0$.  This proves
\eqref{eq:theta-energy}.
\end{proof}

\subsection{The full response ball with convex costs}
\label{app:convex-response-lower}

\begin{proof}[Proof of \cref{thm:response-lower}\textup{(i)}]
Fix the algorithm's private seed and set $H=H_0=\lfloor T/2\rfloor$.  With
the sign sequences of \cref{lem:shifted-signs}, take $w_t=W\xi_t$ and
\[
 c_t(x,u)=L\varepsilon_tu.
\]
The algorithm chooses $u_t$ before observing $\varepsilon_t$, so its
expected total cost over the finite sign experiment is zero.  A delayed
policy \eqref{eq:delayed-policies} has total cost
\[
 LW\sum_{j=1}^H\theta_jS_j.
\]
Minimizing over $\norm{\theta}_1\le R/2$ and using
\eqref{eq:shifted-max} gives
\[
 \E\inf_{\norm{\theta}_1\le R/2}
 \sum_t c_t(y_t^\theta)
 \le-cLWR\sqrt{T\log(eT)}.
\]
The minimizing policy has summed state and input response-coefficient
deviations at most $2\norm{\theta}_1\le R$, with finite support, so it
belongs to \eqref{eq:infinite-l1-response-ball}.  The finite-support
argument at the start of this appendix proves
\eqref{eq:convex-response-lower} for randomized algorithms and provides the
claimed deterministic disturbance--cost sequence.
\end{proof}

\subsection{The full response ball with strongly convex costs}
\label{app:strong-response-lower}

For the second part of \cref{thm:response-lower}, quadratic penalties make
the costs strongly convex:
\begin{equation}
 c_t(x,u)
 =
 \frac{\mu}{2}(x-w_{t-1})^2
 +\frac{\mu}{2}u^2
 +\frac{GD}{4}\,\varepsilon_tu.
 \label{eq:strong-response-cost}
\end{equation}

\begin{proof}[Proof of \cref{thm:response-lower}\textup{(ii)}]
Fix the algorithm's private seed.  Use the disturbances of part~(i) and the
costs \eqref{eq:strong-response-cost}.  On \eqref{eq:one-step-system} with
zero prehistory, $x_t-w_{t-1}=u_{t-1}$.  In this finite sign experiment,
the algorithm's expected linear cost is zero and its quadratic cost is
nonnegative.

Take the comparator from \cref{lem:shifted-signs} with $r=R/2$.  Its
expected linear cost is at most
\[
 -cGDWR\sqrt{T\log(eT)}.
\]
As in part~(i), its response deviations have finite support and summed norm
at most $R$, so it belongs to \eqref{eq:infinite-l1-response-ball}.
Up to endpoint terms, each of its two quadratic sums is at most $W^2$ times
the left side of \eqref{eq:theta-energy}.  Together they contribute at most
\[
 C\mu W^2R^2\sqrt{T\log(eT)}.
\]
Under the assumptions $G\ge4\mu$ and $D=4W(1+R)$, this is absorbed by the
linear gain when the universal constant $c_0$ in $R\le c_0$ is small
enough.  The factor $1+R$ in $D$ ensures that the comparator trajectories,
which satisfy $\abs{x_t^\theta}\le W(1+R/2)$ and
$\abs{u_t^\theta}\le WR/2$, lie in the radius-$D$ ball on which the gradient
bound is stated.  Thus the comparator has expected cost at most
$-cGDWR\sqrt{T\log(eT)}$, which proves
\eqref{eq:strong-response-lower}.

The Hessian of \eqref{eq:strong-response-cost} is $\mu I_2$, so every cost
is globally $\mu$-strongly convex.  On $\norm{(x,u)}_2\le D$, with
$\abs{w_{t-1}}\le W\le D$,
\[
 \norm{\nabla c_t(x,u)}_2
 \le
 \sqrt{\mu^2(D+W)^2+
       (\mu D+GD/4)^2}
 \le GD,
\]
because $G\ge4\mu$ and $W\le D/4$.  The hard distribution is oblivious.
The bounded-below observation at the start of this appendix justifies
averaging whole regret over the private seed.  Finite-support averaging then
gives each randomized algorithm a deterministic sequence with the same lower
bound.
\end{proof}

\clearpage

\subsection{Stable linear state feedback with convex costs}
\label{app:convex-feedback-lower}

\begin{theorem}[Lower bounds for convex costs and stable linear state-feedback policies]
\label{thm:convex-lower}
There are universal constants $c,C>0$ such that the following holds.  Fix
$W>0$, $\gamma\in(0,1/2]$, and $T\ge C/\gamma$.  On
\eqref{eq:one-step-system} with $x_{\mathrm{ini}}=0$ and $w_t\equiv W$, both
conclusions already hold when the benchmark is restricted to
$K^{(0)}=0$ and $K^{(1)}=-(1-\gamma)$, both in
$\cK_{1,\gamma}$.
\begin{enumerate}[label=(\roman*),leftmargin=2.2em]
\item For every $L>0$, an oblivious distribution over globally
$L$-Lipschitz linear costs forces every causal randomized algorithm to
satisfy
\begin{equation}
 \E\Reg_T(\cK_{1,\gamma})
 \ge cLW\gamma^{-1}\sqrt T.
 \label{eq:lower-L}
\end{equation}
\item For every $G>0$, put $D=2W/\gamma$.  An oblivious distribution over
convex costs satisfying
$\norm{\nabla c_t(y)}_2\le GD$ for every $y$ forces every causal randomized
algorithm to satisfy
\begin{equation}
 \E\Reg_T(\cK_{1,\gamma})
 \ge cGW^2\gamma^{-2}\sqrt T.
 \label{eq:lower-GD}
\end{equation}
\end{enumerate}
For each randomized algorithm, each conclusion is also attained, in
expectation over its private seed, by a deterministic cost sequence in the
corresponding support.
\end{theorem}

\begin{proof}[Proof of \cref{thm:convex-lower}]
The two policies stay order $W/\gamma$ apart for most of the horizon.
Random linear costs on the state give the algorithm zero expected cost, while
the better policy in hindsight gains order $W\gamma^{-1}\sqrt T$.

On \eqref{eq:one-step-system}, take $x_1=0$, $w_t=W$, and, under the
convention $u_t=-Kx_t$,
\[
 K^{(0)}=0,
 \qquad
 K^{(1)}=-(1-\gamma).
\]
Their closed-loop coefficients are $0$ and $1-\gamma$, so both policies
belong to $\cK_{1,\gamma}$.  Their states are
\[
 x_t^{(0)}=W\quad(t\ge2),
 \qquad
 x_t^{(1)}=\frac{W}{\gamma}
 \left(1-(1-\gamma)^{t-1}\right).
\]
Let $\Delta_t=x_t^{(1)}-x_t^{(0)}$ and
$t_0=1+\lceil2/\gamma\rceil$.  Since
$(1-\gamma)^{t-1}\le e^{-\gamma(t-1)}$ and $\gamma\le1/2$, every $t\ge t_0$
satisfies
\begin{equation*}
\begin{aligned}
 \Delta_t
 &=\frac{W}{\gamma}
   \left(1-(1-\gamma)^{t-1}-\gamma\right)\\
 &\ge\frac{W}{\gamma}
   \left(1-e^{-2}-\frac12\right)
 =:c_\Delta\frac{W}{\gamma}.
\end{aligned}
\end{equation*}
Taking the universal constant $C$ in the theorem large enough ensures
$t_0\le T/2$.  Hence, for a universal $c_1>0$,
\begin{equation}
 \sum_{t=1}^T\Delta_t^2
 \ge c_1W^2\gamma^{-2}T.
 \label{eq:convex-lower-energy}
\end{equation}

Fix the algorithm's private seed, and draw independent Rademacher signs
$\sigma_t$ before play.  For part~(i), set $c_t(x,u)=\sigma_tLx$.  The
algorithm's state is fixed before the fresh sign is revealed, so its expected
cumulative cost over the signs is zero.  The better of the two comparators
has expected cost
\[
\begin{aligned}
 \E\min_{j\in\{0,1\}}
   \sum_{t=1}^T\sigma_tLx_t^{(j)}
 &=-\frac L2
   \E\left|\sum_{t=1}^T\sigma_t\Delta_t\right|\\
 &\le-cL\left(\sum_{t=1}^T\Delta_t^2\right)^{1/2}
 \le-cLW\gamma^{-1}\sqrt T.
\end{aligned}
\]
The equality uses $\min\{a,b\}=(a+b)/2-\abs{a-b}/2$ together with
$\E\sum_t\sigma_tx_t^{(j)}=0$.  The first inequality is Khintchine's, and
the second is \eqref{eq:convex-lower-energy}.  This proves
\eqref{eq:lower-L}.

For part~(ii), use the same construction with
$c_t(x,u)=\sigma_tGD\,x$ and $D=2W/\gamma$.  The gradient norm is $GD$
everywhere, and the same calculation proves \eqref{eq:lower-GD}.  Both sign
distributions are fixed before play, so the finite-support argument at the
start of this appendix extends both claims to randomized algorithms and
gives the deterministic-sequence conclusion.
\end{proof}

\subsection{Stable linear state feedback with strongly convex costs}
\label{app:strong-feedback-lower}

\begin{theorem}[Lower bound for strongly convex costs and stable linear state-feedback policies]
\label{thm:strong-lower}
There are universal constants $c,C>0$ such that the following holds.  Fix
$W,\mu>0$, $\gamma\in(0,1/4]$, and $T\ge C/\gamma$.  On the fixed one-state,
two-input linear system
\begin{equation}
 x_{t+1}=x_t+u_{1,t}+u_{2,t}+W,
 \qquad x_1=0,
 \label{eq:strong-lower-system-main}
\end{equation}
there is an oblivious distribution over quadratic costs such that every
cost is globally $\mu$-strongly convex and, with $D=4W/\gamma$, satisfies
\[
 \norm{\nabla c_t(y)}_2\le2\mu D
 \qquad\text{whenever }\norm{y}_2\le D,
\]
while every causal randomized algorithm obeys
\begin{equation}
 \E\Reg_T(\cK_{1,\gamma})
 \ge c\mu W^2\gamma^{-2}\log T.
 \label{eq:strong-lower}
\end{equation}
This system has the explicit optimal recovery cost
\begin{equation}
 \Vstar(e)=\sqrt{\frac32}\,\abs{e}.
 \label{eq:strong-lower-correction}
\end{equation}
For each randomized algorithm, some deterministic cost sequence in the
support gives the same lower bound in expectation over its private seed.
\end{theorem}

All policies in the construction have the same state trajectory; only the
split of the input between two interchangeable actuators is unknown.
Learning this split is a one-parameter estimation problem.  The Van Trees
inequality, a Bayesian Cram\'er--Rao bound \cite{vantrees1968}, bounds its
risk in each round.

\begin{lemma}[Van Trees inequality]
\label{lem:van-trees}
Let $\vartheta$ have a continuously differentiable density $\varpi$ on
$[a,b]$, with $\varpi(a)=\varpi(b)=0$, and define
\[
 I_\varpi=\int_a^b\frac{\varpi'(v)^2}{\varpi(v)}\,dv<\infty.
\]
Conditionally on $\vartheta$, let $Z_1,\ldots,Z_k$ be independent
observations in finite sets, with probability mass functions
$p_s(\cdot\mid v)$ that are positive and continuously differentiable in $v$.
Their Fisher informations are
\[
 I_s(v)
 =\E\!\left[
   \left(\partial_v\log p_s(Z_s\mid v)\right)^2
   \middle|\vartheta=v
 \right].
\]
Then every possibly randomized estimator $\widehat\vartheta$ based on
$Z_1,\ldots,Z_k$ satisfies
\begin{equation}
 \E(\widehat\vartheta-\vartheta)^2
 \ge
 \frac{1}{I_\varpi+\sum_{s=1}^k\E I_s(\vartheta)}.
 \label{eq:van-trees}
\end{equation}
\end{lemma}

\begin{proof}
If the mean squared error is infinite, there is nothing to prove, so assume
it is finite.  The private seed is independent of $\vartheta$ and the observations,
and the right side of \eqref{eq:van-trees} does not depend on it.  It
therefore suffices to prove the inequality for each fixed seed, that is, for
a deterministic estimator, and then average.  Let $q(v,z)$ be the joint
density of the parameter and observations, with respect to Lebesgue measure
in $v$ and counting measure in $z$, and put
$S(v,z)=\partial_v\log q(v,z)$.  Integration by parts, using the vanishing
prior density at both endpoints, gives
\[
\begin{aligned}
 \E[(\widehat\vartheta-\vartheta)S]
 &=\int_a^b\!\int
   (\widehat\vartheta(z)-v)\,\partial_vq(v,z)\,dz\,dv
 =1.
\end{aligned}
\]
Cauchy--Schwarz yields
$1\le\E(\widehat\vartheta-\vartheta)^2\E S^2$.  The prior score and the
observation scores are mutually orthogonal in $L^2$, because each
observation score has conditional mean zero and the observations are
conditionally independent.  Hence
\[
 \E S^2=I_\varpi+\sum_{s=1}^k\E I_s(\vartheta),
\]
which proves \eqref{eq:van-trees}.
\end{proof}

\begin{proof}[Proof of \cref{thm:strong-lower}]
We first verify the recovery cost.  We then build a one-parameter family of
comparators with a common state trajectory and reduce regret to estimating
the parameter from sign observations.

For the system \eqref{eq:strong-lower-system-main}, $A=1$ and
$B=(1,1)$.  Given an error $e$, the input $v=(-e/2,-e/2)$ clears it in one
step at cost $\sqrt{3/2}\,\abs{e}$.  Conversely, take any finite-cost
correction sequence and put $q_s=v_{1,s}+v_{2,s}$.  Then
$\norm{v_s}_2\ge\abs{q_s}/\sqrt2$, finite cost implies
$\sum_sq_s=-e$, and Minkowski's inequality gives
\[
\begin{aligned}
 \sum_{s\ge0}\norm{(e_s,v_s)}_2
 &\ge\sum_{s\ge0}
   \sqrt{\abs{e_s}^2+\frac12\abs{q_s}^2}\\
 &\ge
 \sqrt{\left(\sum_s\abs{e_s}\right)^2
       +\frac12\left(\sum_s\abs{q_s}\right)^2}
 \ge\sqrt{\frac32}\,\abs{e}.
\end{aligned}
\]
Thus the one-step input is optimal, which proves
\eqref{eq:strong-lower-correction}.

Put
\begin{equation*}
 X=\frac{W}{\gamma},
 \qquad
 \alpha_t=1-(1-\gamma)^{t-1},
 \qquad
 \bar x_t=X\alpha_t.
\end{equation*}
For $\theta\in[-1/2,1/2]$, define
\begin{equation*}
 K_\theta=(\gamma-\theta,\theta)^\top.
\end{equation*}
Its closed-loop coefficient is
$1-(\gamma-\theta)-\theta=1-\gamma$, and
\[
 \norm{K_\theta}_2^2
 \le\left(\gamma+\frac12\right)^2+\frac14
 \le\frac{13}{16},
\]
so $K_\theta\in\cK_{1,\gamma}$.  Every member of the family generates
the same state trajectory $x_t^\theta=\bar x_t$, with inputs
\begin{equation}
 u_{1,t}^\theta=-(\gamma-\theta)\bar x_t,
 \qquad
 u_{2,t}^\theta=-\theta\bar x_t.
 \label{eq:strong-lower-comparator-actions}
\end{equation}
Only the allocation between actuators depends on $\theta$.

Draw $\theta$ from the density
\begin{equation}
 \varpi(\theta)
 =2\cos^2(\pi\theta)\,
  \mathbf 1_{[-1/2,1/2]}(\theta).
 \label{eq:strong-lower-prior}
\end{equation}
Conditionally on $\theta$, draw independent signs
$\xi_1,\ldots,\xi_T\in\{-1,+1\}$ with
\begin{equation*}
 \Pr(\xi_t=+1\mid\theta)
 =\frac{1-\alpha_t\theta}{2}.
\end{equation*}
These probabilities lie in $[1/4,3/4]$.  The parameter and all signs are
drawn before play, so the cost distribution is oblivious.  On round $t$, set
\begin{equation}
\begin{aligned}
 c_t(x,u_1,u_2)
 =\frac\mu2\bigl[{}
 &(x-\bar x_t)^2\\
 &+(u_1+\gamma\bar x_t+X\xi_t)^2
  +(u_2-X\xi_t)^2
 \bigr].
\end{aligned}
\label{eq:strong-lower-cost}
\end{equation}
The Hessian is $\mu I_3$, so every cost is globally
$\mu$-strongly convex.  Its center is
\[
 z_t=(\bar x_t,-\gamma\bar x_t-X\xi_t,X\xi_t).
\]
Since $\gamma\le1/4$, $\norm{z_t}_2\le2X$.  With $D=4X$, every
$y$ satisfying $\norm{y}_2\le D$ obeys
\begin{equation*}
 \norm{\nabla c_t(y)}_2
 =\mu\norm{y-z_t}_2
 \le\frac32\mu D
 \le2\mu D.
\end{equation*}
The full cost function reveals $\xi_t$ after round $t$.

The signs are calibrated to the comparator trajectory:
\[
 \E[X\xi_t\mid\theta]
 =-X\alpha_t\theta
 =-\theta\bar x_t.
\]
Consequently, the conditional mean of the two input coordinates of the
cost center is exactly the comparator input
\eqref{eq:strong-lower-comparator-actions}.

Fix any causal randomized algorithm and condition on its independent
random seed.  Its round-$t$ input is then a function of the previously
revealed signs.  Put
\[
 m_{1,t}(\theta)=-(\gamma-\theta)\bar x_t,
 \qquad
 m_{2,t}(\theta)=-\theta\bar x_t.
\]
For the fresh sign $\xi_t$, the conditional-mean identity for squared loss
gives
\begin{align*}
&\E\!\left[
 (u_{1,t}+\gamma\bar x_t+X\xi_t)^2
 -(m_{1,t}(\theta)+\gamma\bar x_t+X\xi_t)^2
 \middle|\theta,\xi_1,\ldots,\xi_{t-1}
\right]
 =(u_{1,t}-m_{1,t}(\theta))^2,\\
&\E\!\left[
 (u_{2,t}-X\xi_t)^2
 -(m_{2,t}(\theta)-X\xi_t)^2
 \middle|\theta,\xi_1,\ldots,\xi_{t-1}
\right]
 =(u_{2,t}-m_{2,t}(\theta))^2.
\end{align*}
The comparator also makes the state term in
\eqref{eq:strong-lower-cost} vanish.  Regret to the best member of
$\cK_{1,\gamma}$ is at least regret to $K_\theta$, so
\begin{equation}
\begin{aligned}
 \E\Reg_T(\cK_{1,\gamma})
 \ge\frac\mu2\sum_{t=1}^T\E\bigl[{}
 &(u_{1,t}+\gamma\bar x_t-\theta\bar x_t)^2\\
 &+(u_{2,t}+\theta\bar x_t)^2
 \bigr].
\end{aligned}
\label{eq:strong-lower-excess}
\end{equation}
For $t\ge2$, define the estimator
\begin{equation}
 \widehat\theta_t
 =\frac{u_{1,t}+\gamma\bar x_t-u_{2,t}}
        {2\bar x_t}.
 \label{eq:strong-lower-estimator}
\end{equation}
Projection onto the input direction $(1,-1)$ gives
\begin{equation*}
\begin{aligned}
 &(u_{1,t}+\gamma\bar x_t-\theta\bar x_t)^2
 +(u_{2,t}+\theta\bar x_t)^2\\
 &\hspace{4em}\ge
 2\bar x_t^2(\widehat\theta_t-\theta)^2.
\end{aligned}
\end{equation*}
The estimator uses only $\xi_1,\ldots,\xi_{t-1}$ and private randomness.

For the prior \eqref{eq:strong-lower-prior}, direct differentiation gives
\begin{equation*}
 I_\varpi
 =\int_{-1/2}^{1/2}
   \frac{\varpi'(\theta)^2}{\varpi(\theta)}\,d\theta
 =4\pi^2.
\end{equation*}
The Fisher information in $\xi_s$ is
\begin{equation*}
 I_s(\theta)
 =\frac{\alpha_s^2}{1-\alpha_s^2\theta^2}
 \le\frac43\alpha_s^2.
\end{equation*}
Apply \cref{lem:van-trees} to \eqref{eq:strong-lower-estimator}.  Since
$\bar x_t=X\alpha_t$,
\begin{equation}
 \E\!\left[\bar x_t^2
 (\widehat\theta_t-\theta)^2\right]
 \ge
 \frac{X^2\alpha_t^2}
 {I_\varpi+\frac43\sum_{s<t}\alpha_s^2}.
 \label{eq:strong-lower-round-risk}
\end{equation}
Combining
\eqref{eq:strong-lower-excess}--\eqref{eq:strong-lower-round-risk} yields
\begin{equation}
 \E\Reg_T(\cK_{1,\gamma})
 \ge
 \mu X^2\sum_{t=1}^T
 \frac{\alpha_t^2}
 {I_\varpi+\frac43\sum_{s<t}\alpha_s^2}.
 \label{eq:strong-lower-harmonic-sum}
\end{equation}
The $t=1$ summand is zero because $\alpha_1=0$.

It remains to sum the round risks.  Set
$S_t=I_\varpi+\frac43\sum_{s\le t}\alpha_s^2$.  Since
$\log(1+r)\le r$ for $r\ge0$,
\[
 \frac{\alpha_t^2}{S_{t-1}}
 \ge\frac34\log\frac{S_t}{S_{t-1}}.
\]
Telescoping gives
\begin{equation*}
 \sum_{t=1}^T
 \frac{\alpha_t^2}
 {I_\varpi+\frac43\sum_{s<t}\alpha_s^2}
 \ge
 \frac34\log\!\left(
 1+\frac{4}{3I_\varpi}\sum_{t=1}^T\alpha_t^2
 \right).
\end{equation*}
For $t\ge1+\lceil1/\gamma\rceil$,
$\alpha_t\ge1-e^{-1}$.  Thus, for a sufficiently large universal $C$,
$T\ge C/\gamma$ implies
\begin{equation}
 \sum_{t=1}^T\alpha_t^2\ge c_\alpha T
 \label{eq:strong-lower-alpha-sum}
\end{equation}
for a universal $c_\alpha>0$.  Combining
\eqref{eq:strong-lower-harmonic-sum}--\eqref{eq:strong-lower-alpha-sum}
with $X=W/\gamma$ proves \eqref{eq:strong-lower}.

The distribution does not depend on the algorithm.  Its whole regrets are
bounded below, so averaging the seedwise bound is valid for arbitrary
private randomization.  Averaging over the finite support of the cost
sequences gives the deterministic-sequence conclusion.
\end{proof}

\subsection{Matching the strongly convex state-feedback bounds}
\label{app:strong-feedback-normalization}

Does the upper bound of \cref{cor:strong-stable-cover} match
\cref{thm:strong-lower}?  It does, on the quadratic instance itself.  Write $X=W/\gamma$.  On that fixed
system, take $K_{\mathrm{cor}}=(1/2,1/2)^\top$, so
$A-BK_{\mathrm{cor}}=0$ and
\[
 P_{\mathrm{cor}}=\sqrt{3/2},
 \qquad
 \alpha=1+P_{\mathrm{cor}}\norm{[A\ B]}_{\mathrm{op}}
 =1+3/\sqrt2.
\]
Every comparator in $\cK_{1,\gamma}$ satisfies
$\norm{y_t^K}_2\le\sqrt2X$, so
\cref{lem:policy-tracking-region} places the algorithm and its references in
the ball of radius
\[
 D_{\mathrm{tr}}=(1+\sqrt6)\sqrt2X<5X.
\]
The quadratic centers $z_t$ from \eqref{eq:strong-lower-cost} have norm at
most $2X$, so throughout this enlarged ball,
\[
 \norm{\nabla c_t(y)}_2
 =\mu\norm{y-z_t}_2\le7\mu X\le L,
 \qquad L=8\mu W/\gamma.
\]
This verifies the tracking-segment assumption of
\cref{cor:strong-stable-cover}.  By \eqref{eq:stable-cover-scales}, its
trajectory modulus satisfies
$\mathcal L_T^{\mathrm{traj}}\le W(\gamma^{-1}+2\sqrt2\gamma^{-2})$.
With $d_K=2$ and the constant $\alpha$ above, the upper bound is therefore
\[
 O\!\left(\mu W^2\gamma^{-2}\log(eT/\gamma)\right).
\]
For $T\ge C/\gamma$, this is
$O(\mu W^2\gamma^{-2}\log T)$, which matches
\eqref{eq:strong-lower} in $T,W,\gamma$ on this fixed system.
The verification uses the explicit quadratic costs, whose gradients can be
bounded on the enlarged ball.  For a general cost class, a gradient bound
only on the comparator region would not justify the same application.

The Lipschitz scale $L\asymp\mu W/\gamma$ cannot be avoided.  If a function
is both $\mu$-strongly convex and $L$-Lipschitz on a convex set $C$, the
midpoint inequality gives $\operatorname{diam}(C)\le4L/\mu$.  A hull of
stable linear state-feedback trajectories with diameter of order
$W/\gamma$ therefore forces $L\gtrsim\mu W/\gamma$.

\section{Comparison with two general frameworks}
\label{app:prior-general-frameworks}

Two frameworks extend online learning to cover online control as a special case.  The sequential-complexity framework \cite{bhatia2020} bounds the
minimax value nonconstructively, and the unbounded-memory framework
\cite{kumar2023} gives a follow-the-regularized-leader (FTRL) algorithm for
convex decisions with linear memory.  Both cover more general problems than ours.  On simple known linear systems, however, their stated bounds can be loose.

\subsection{The sequential-complexity bound}

In the notation of \cite[Theorem~1]{bhatia2020}, the minimax policy regret is
at most
\begin{equation}
 \sum_{t=1}^T\beta_t
 +2\mathfrak R_T^{\mathrm{seq}}(\ell^\Phi\circ\Pi)
 +2\lambda\sup_{\pi\in\Pi}\Omega(\pi).
 \label{eq:bs-generic-bound}
\end{equation}
Here $\ell_t^\Phi(\pi)$ is the loss of $\pi$ on the state it would have
generated had it been used from the start, $\mathfrak R_T^{\mathrm{seq}}$ is
the sequential complexity of these losses, and $\lambda\Omega$ is a
regularizer.  The term $\beta_t$ concerns the regularized empirical risk
minimizer (RERM) of their dual game, which sees the current loss
distribution before choosing: it bounds the gap between the rule's loss on
the state inherited from earlier policies and its counterfactual loss.  For
$N$ policies with losses in $[-C,C]$, the sequential complexity is at most
$C\sqrt{2T\log N}$, but finiteness does not control $\beta_t$.

\paragraph{The stability term can be linear with two policies.}
Let $x_{t+1}=u_t$ with $x_1=0$, and let $\Pi=\{\pi_{-1},\pi_{+1}\}$ with
$\pi_\sigma(x)=\sigma$.  For $b\in(0,1]$ and $\abs{a_t}\le2$, take the linear
costs $c_1(x,u)=a_1u$ and $c_t(x,u)=a_tu+b(x-u)$ for $t\ge2$.  Then
$\ell_t^\Phi(\pi_\sigma)=a_t\sigma$, and for every policy sequence
$\sigma_{1:T}$ the state is $x_t=\sigma_{t-1}$, so
\begin{equation}
 \sum_{t=1}^T c_t(x_t,\sigma_t)
 -\min_{\sigma\in\{-1,1\}}\sum_{t=1}^T a_t\sigma
 =
 \left(
  \sum_{t=1}^T a_t\sigma_t
  -\min_{\sigma\in\{-1,1\}}\sum_{t=1}^T a_t\sigma
 \right)
 +b(\sigma_1-\sigma_T).
 \label{eq:bs-boundary-identity}
\end{equation}
Counterfactual tracking with exponential weights has regret
$O(\sqrt T)$ (\cref{thm:finite-aggregation-main}), the minimax rate for two
policies (\cref{prop:finite-lower-main}).  The
RERM, however, chooses $\sigma_t=-\operatorname{sign}(S_t)$ with
$S_t=\sum_{s\le t}a_s$, and each change of policy has absolute stability gap $2b$.
With $a_1=-1$ and $a_t=2(-1)^t$, it changes on every round, so
$\sum_t\beta_t=\Omega(bT)$ and \eqref{eq:bs-generic-bound} is linear.  A
nonnegative regularizer only shifts the threshold: either
$\lambda\operatorname{osc}(\Omega)\ge T/8$, and the regularization term is
linear, or bounded increments can make $S_t$ alternate around the shifted
threshold for a constant fraction of the horizon.  Their mini-batching
variant \cite[Proposition~2]{bhatia2020} still gives order
$b^{1/3}T^{2/3}$ on this example, which remains loose for every fixed $b>0$.

\subsection{The unbounded-memory bound}

The unbounded-memory framework \cite{kumar2023} takes a closed convex decision
set in a Hilbert space, a linear history $h_t=\mathsf Ah_{t-1}+\mathsf Bq_t$,
and convex $L$-Lipschitz history losses.  Its effective memory is
$H_1=\sum_{k\ge0}k\norm{\mathsf A^k}_{\mathrm{op}}$.  With
$\widetilde L$ the Lipschitz constant of the loss under a fixed decision, and
a regularizer that is $\alpha$-strongly convex in the Hilbert norm with
oscillation at most $D_{\mathrm{reg}}$, \cite[Theorem~3.1]{kumar2023} gives
\begin{equation}
 \frac{D_{\mathrm{reg}}}\eta
 +\frac{\eta T\widetilde L^2}{\alpha}
 +\frac{\eta T L\widetilde L H_1}{\alpha}
 =O\!\left(
   \sqrt{\frac{D_{\mathrm{reg}}}\alpha\,T\widetilde L
          (\widetilde L+LH_1)}
 \right).
 \label{eq:kdk-framework-bound}
\end{equation}
The theorem takes a convex linear-history representation with finite
effective memory as input.  Even when it applies, two
limits appear.

First, no Hilbert geometry recovers the entropy rate of finite classes.  On
the simplex with no memory ($\mathsf A=0$) and linear losses in $[-C,C]^N$,
Euclidean geometry gives $O(C\sqrt{NT})$, and no other choice does better.

\begin{lemma}[Hilbert-geometry bound for the expert simplex]
\label{lem:kdk-hilbert-simplex}
Let $\norm{\cdot}_{\mathsf H}$ be any Hilbert norm on the affine hull of
$\simplex N$, and let $R$ be $\alpha$-strongly convex in this norm with
oscillation at most $D_{\mathrm{reg}}$.  There is a vector $g\in[-C,C]^N$ such
that
\begin{equation}
 \norm{g}_{\mathsf H,*}\sqrt{\frac{D_{\mathrm{reg}}}\alpha}
 \ge C\sqrt{\frac{N-1}{8}}.
 \label{eq:kdk-hilbert-obstruction}
\end{equation}
Consequently, the uniform memoryless bound obtained from
\eqref{eq:kdk-framework-bound} is at least of order $C\sqrt{NT}$ in its
displayed parameters.
\end{lemma}

\begin{proof}
Strong convexity bounds each edge length of the simplex, and a random sign
vector then has large dual norm.  For $N=1$ both sides of
\eqref{eq:kdk-hilbert-obstruction} are zero, so suppose $N\ge2$.  Let
$v_i=e_i-e_N$ for $i<N$, and let $G$ be their Gram matrix in
$\norm{\cdot}_{\mathsf H}$.  Strong convexity at the midpoint of $e_i$ and
$e_N$ gives
\[
 R\!\left(\frac{e_i+e_N}{2}\right)
 \le \frac{R(e_i)+R(e_N)}2
    -\frac{\alpha}{8}\norm{v_i}_{\mathsf H}^2.
\]
Since the range of $R$ has width at most $D_{\mathrm{reg}}$,
$\norm{v_i}_{\mathsf H}^2\le8D_{\mathrm{reg}}/\alpha$, and hence
$\operatorname{tr}(G)\le8(N-1)D_{\mathrm{reg}}/\alpha$.  For
$\varepsilon\in\{-1,1\}^{N-1}$, set $g_N=0$ and $g_i=C\varepsilon_i$.
Restricted to the tangent space of the simplex, the dual norm of this linear
functional is $C(\varepsilon^\top G^{-1}\varepsilon)^{1/2}$.  Averaging over
independent Rademacher signs and using
$\operatorname{tr}(G)\operatorname{tr}(G^{-1})\ge(N-1)^2$ shows that some
$\varepsilon$ satisfies
\[
 \norm{g}_{\mathsf H,*}^2
 \ge C^2\operatorname{tr}(G^{-1})
 \ge \frac{C^2(N-1)\alpha}{8D_{\mathrm{reg}}}.
\]
This is \eqref{eq:kdk-hilbert-obstruction}.  Since $\widetilde L\ge\norm{g}_{\mathsf H,*}$,
optimizing the first two terms of \eqref{eq:kdk-framework-bound} over $\eta$
contributes a constant multiple of
\[
 \norm{g}_{\mathsf H,*}\sqrt{D_{\mathrm{reg}}T/\alpha},
\]
which is at least of order $C\sqrt{NT}$ by \eqref{eq:kdk-hilbert-obstruction}.
\end{proof}

Negative entropy attains $O(C\sqrt{T\log N})$ in the $\ell_1$ geometry,
which is not Hilbert; \cref{thm:finite-policy-library} uses exactly this
geometry.

Second, the bound grows with effective memory.  For a shift history of
length $m$, $H_1=\Theta(m^2)$, so at fixed remaining scales
\eqref{eq:kdk-framework-bound} is of order $m\sqrt T$ and gives no sublinear
guarantee when $m=\Omega(\sqrt T)$; if the history keeps every past
decision at full weight, $H_1=\infty$.  The online-control specialization
\cite[Theorem~4.1]{kumar2023} keeps $H_1$ finite by assuming a known strongly
stable baseline and geometrically decaying DAC coefficients.  Our
response-ball guarantee needs neither assumption
(\cref{sec:full-response-ball}).
\end{document}